\documentclass[11pt]{amsart} \textwidth=14.5cm \oddsidemargin=1cm
\usepackage[rgb]{xcolor}
\usepackage{tikz}
\usepackage{verbatim}
\usepackage{amsmath}
\usepackage{amsxtra}
\usepackage{amscd}
\usepackage{amsthm}
\usepackage{amsfonts}
\usepackage{amssymb}
\usepackage{eucal}
\usetikzlibrary{arrows}

\usepackage{latexsym,todonotes}

\usepackage[linktocpage=true]{hyperref}
\usepackage[all, cmtip]{xy}
\usepackage{stmaryrd}
\usepackage{color} 

\newtheorem{thm}{Theorem} [section]
\newtheorem{cor}[thm]{Corollary}
\newtheorem{lem}[thm]{Lemma}
\newtheorem{prop}[thm]{Proposition}

\theoremstyle{definition}
\newtheorem{definition}[thm]{Definition}
\newtheorem{example}[thm]{Example}

\theoremstyle{remark}
\newtheorem{rem}[thm]{Remark}

\numberwithin{equation}{section}
\begin{document}

\newcommand{\thmref}[1]{Theorem~\ref{#1}}
\newcommand{\secref}[1]{Section~\ref{#1}}
\newcommand{\lemref}[1]{Lemma~\ref{#1}}
\newcommand{\propref}[1]{Proposition~\ref{#1}}
\newcommand{\corref}[1]{Corollary~\ref{#1}}
\newcommand{\remref}[1]{Remark~\ref{#1}}
\newcommand{\eqnref}[1]{(\ref{#1})}

\newcommand{\exref}[1]{Example~\ref{#1}}

\newtheorem{innercustomthm}{{\bf Theorem}}
\newenvironment{customthm}[1]
  {\renewcommand\theinnercustomthm{#1}\innercustomthm}
  {\endinnercustomthm}
  
  \newtheorem{innercustomcor}{{\bf Corollary}}
\newenvironment{customcor}[1]
  {\renewcommand\theinnercustomcor{#1}\innercustomcor}
  {\endinnercustomthm}
  
  \newtheorem{innercustomprop}{{\bf Proposition}}
\newenvironment{customprop}[1]
  {\renewcommand\theinnercustomprop{#1}\innercustomprop}
  {\endinnercustomthm}

\newcommand{\nc}{\newcommand}
 \nc{\Z}{{\mathbb Z}}
 \nc{\C}{{\mathbb C}}
 \nc{\N}{{\mathbb N}}
 \nc{\F}{{\mf F}}
 \nc{\Q}{\mathbb{Q}}
 \nc{\la}{\lambda}
 \nc{\ep}{\epsilon}
 \nc{\h}{\mathfrak h}
 \nc{\n}{\mf n}
 \nc{\G}{{\mathfrak g}}
 \nc{\DG}{\widetilde{\mathfrak g}}
 \nc{\SG}{\breve{\mathfrak g}}
 \nc{\is}{{\mathbf i}}
 \nc{\V}{\mf V}
 \nc{\bi}{\bibitem}
 \nc{\E}{\mc E}
 \nc{\ba}{\tilde{\pa}}
 \nc{\half}{\frac{1}{2}}
 \nc{\hgt}{\text{ht}}
 \nc{\mc}{\mathcal}
 \nc{\mf}{\mathfrak} 
 \nc{\hf}{\frac{1}{2}}
\nc{\ov}{\overline}
\nc{\ul}{\underline}
\nc{\I}{\mathbb{I}}
\nc{\aaa}{{\mf A}}
\nc{\xx}{{\mf x}}
\nc{\id}{\text{id}}
\nc{\one}{\bold{1}}
\nc{\Qq}{\Q(q)}
\nc{\ua}{\mf{u}}
\nc{\nb}{u}
\nc{\inv}{\theta}
\nc{\mA}{\mathcal{A}}
\newcommand{\tK}{\widetilde{K}}

\nc{\U}{\bold{U}}
\nc{\Udot}{\dot{\U}}
\nc{\f}{\bold{f}}
\nc{\fprime}{\bold{'f}}
\nc{\B}{\bold{B}}
\nc{\Bdot}{\dot{\B}}
\nc{\Dupsilon}{\Upsilon^{\vartriangle}}
\newcommand{\T}{\texttt T}
\newcommand{\vs}{\varsigma}
\newcommand{\Pa}{{\bf{P}}}
\newcommand{\Padot}{\dot{\bf{P}}}

\nc{\ipsi}{\psi_{\imath}}
\nc{\Ui}{{\bold{U}^{\imath}}}
\nc{\Uidot}{\dot{\bold{U}}^{\imath}}
 \nc{\be}{e}
 \nc{\bff}{f}
 \nc{\bk}{k}
 \nc{\bt}{t}
 \nc{\BLambda}{{\Lambda_{\inv}}}
\nc{\Ktilde}{\widetilde{K}}
\nc{\bktilde}{\widetilde{k}}
\nc{\Yi}{Y^{w_0}}
\nc{\bunlambda}{\Lambda^\imath}
\newcommand{\Iwhite}{\I_{\circ}}
\nc{\ile}{\le_\imath}
\nc{\il}{<_{\imath}}

\newcommand{\ff}{B}

\nc{\qq}{(q_i^{-1}-q_i)}
\nc{\qqq}{(1-q_i^{-2})^{-1}}
\nc{\qqqj}{(1-q_j^{-2})^{-1}}

\nc{\etab}{\eta^{\bullet}}
\newcommand{\Iblack}{\I_{\bullet}}
\newcommand{\wb}{w_\bullet}
\newcommand{\UIblack}{\U_{\Iblack}}

\newcommand{\blue}[1]{{\color{blue}#1}}
\newcommand{\red}[1]{{\color{red}#1}}
\newcommand{\green}[1]{{\color{green}#1}}
\newcommand{\white}[1]{{\color{white}#1}}

\newcommand{\huanchentodo}{\todo[inline,color=orange!20, caption={}]}
\newcommand{\wtodo}{\todo[inline,color=green!20, caption={}]}

\title[Canonical bases arising from quantum symmetric pairs]
{Canonical bases arising from quantum \\
symmetric pairs}
 
 \author[Huanchen Bao]{Huanchen Bao}
\address{Department of Mathematics, University of Maryland, College Park, MD 20742.}
\email{huanchen@math.umd.edu}

\author[Weiqiang Wang]{Weiqiang Wang}
\address{Department of Mathematics, University of Virginia, Charlottesville, VA 22904}
\email{ww9c@virginia.edu}

\begin{abstract}
We develop a general theory of canonical bases for quantum symmetric pairs $(\bold{U}, \bold{U}^\imath)$ with parameters of arbitrary finite type. 
We construct new canonical bases for the finite-dimensional simple $\bold U$-modules and their tensor products regarded as $\bold{U}^\imath$-modules. 
We also construct a canonical basis for the modified form of the $\imath$quantum group $\bold{U}^\imath$.
To that end, we establish several new structural results on quantum symmetric pairs,
such as bilinear forms, braid group actions, integral forms, Levi subalgebras (of real rank one),  and integrality of the intertwiners. 
\end{abstract}

\maketitle

\let\thefootnote\relax\footnotetext{{\em 2010 Mathematics Subject Classification.} Primary 17B10.}

\tableofcontents


\section{Introduction}

\subsection{Background}

Let $\U = \U_q(\mathfrak{g})$ be the Drinfeld-Jimbo quantum group with triangular decomposition $\U=\U^-\U^0\U^+$. 
Lusztig \cite{Lu90, Lu91} constructed the canonical basis on an integral $\mA$-form ${}_\mA\U^-$ of $\U^-$ and 
compatible canonical bases on finite-dimensional simple $\U$-modules $L(\lambda)$ (using perverse sheaves for general type, or via PBW basis in finite type as well), 
where $\mA =\Z[q,q^{-1}]$. 
Kashiwara \cite{Ka91} gave a different algebraic construction of the canonical bases by globalizing the crystal bases at $q=0$. 

In \cite{Lu92}, Lusztig constructed the canonical bases on the tensor product of a lowest weight module and a highest weight module 
$^\omega L(\lambda) \otimes L(\mu)$, for dominant integral weights $\la, \mu \in X^+$.
Based on various compatibilities of the canonical bases as $\la, \mu$ vary, he further constructed the canonical bases on the modified form $\Udot$. 
All these constructions fit in the notion of based modules (see \cite{Lu94} for finite type and \cite{BW16} for general type).

Given an involution $\theta$ on a complex simple Lie algebra $\mathfrak{g}$, 
we obtain a symmetric pair $(\mathfrak{g}, \mathfrak{g}^\theta)$, or a pair of enveloping algebras $(\U(\mathfrak{g}), \U(\mathfrak{g}^\theta))$,
where $\mathfrak{g}^\theta$ denotes the fixed point subalgebra. 
The classification of symmetric pairs of finite type is equivalent to the classification of real simple Lie algebras, which goes back to \'Elie~Cartan, cf. \cite{OV};
these classifications are often described in terms of the Satake diagrams; see \cite{Araki}. 
Recall a Satake diagram consists of a partition of the nodes of a Dynkin diagram, $\I = \Iblack \cup \Iwhite$, 
and a (possibly trivial) Dynkin diagram involution $\tau$; see Table~\ref{table:Satake}. 

As a quantization of $(\U(\mathfrak{g}), \U(\mathfrak{g}^\theta))$, 
a theory of quantum symmetric pairs $(\U, \Ui)$ of finite type was systematically developed by Letzter \cite{Le99, Le02}.
Such a $\Ui$ of finite type is constructed from the Satake diagrams. 
In this theory, $\Ui$ is a coideal subalgebra of $\U$ with parameters
(i.e., the comultiplication $\Delta$ on $\U$ satisfies $\Delta: \Ui \rightarrow \Ui \otimes \U$) 
but not a Hopf subalgebra of $\U$, and $\theta$ is quantized as an automorphism (but not of order 2) of $\U$.
The algebra  $\Ui$ has a complicated presentation including nonhomogeneous Serre type relations.
The quantum symmetric pairs (QSP for short) have been further studied and generalized to the Kac-Moody setting by Kolb \cite{Ko14}. 
The algebra $\Ui$ on its own will be also referred to as an $\imath$quantum group.

\subsection{The goal} 
The goal of this paper is to develop systematically a theory of canonical basis for quantum symmetric pairs of arbitrary finite type.
Actually several main constructions of this paper work in the Kac-Moody generality, 
though we shall assume the QSPs are of finite type throughout the paper unless otherwise specified.
We shall construct a new canonical basis (called $\imath$-canonical basis) on the modified form $\Uidot$ of the $\imath$quantum group $\Ui$ 
as well as $\imath$-canonical bases on based $\U$-modules, including simple finite-dimensional $\U$-modules and their tensor products. 

It is instructive for us to view various original constructions of canonical bases (see \cite{Lu94}) 
as constructions for the QSPs of diagonal type $(\U\otimes \U, \U)$ or for the degenerate QSP with $\Ui=\U$ (i.e., $\Iblack =\I$). 

Canonical bases have numerous applications including category $\mc{O}$, algebraic combinatorics, 
total positivity, cluster algebras, categorification, geometric and modular representation theory. 
It is  our hope that the theory of canonical bases arising from QSPs can be further developed,
and it will lead to new advances in some of these areas. 

\subsection{What was known?}

Let us recall the early effort 
toward the constructions of canonical bases for a very special case of quantum symmetric pairs,
which is of type AIII/AIV with $\Iblack =\emptyset$.  
In \cite{BW13}, the authors constructed the intertwiner $\Upsilon$ 
(an analogue of the quasi-$\mathcal R$-matrix in the QSP setting), proved $\Upsilon$ is integral, and used it to define a new bar involution
$\ipsi =\Upsilon \circ \psi$ 
 on any based $\U$-module with a bar involution $\psi$ to obtain a based $\Ui$-module. 
In particular, we obtain the $\imath$-canonical basis on finite-dimensional simple $\U$-modules and their tensor products. 
An application of such an $\imath$-canonical basis (with a particular choice of parameters) may be found in  \cite{BW13}. There the authors formulated a Kazhdan-Lusztig theory for super type B, which was an open problem for decades.


For type AIII/AIV with $\Iblack =\emptyset$, the modified $\imath$quantum group $\Uidot$ 
and its $\imath$-canonical basis have been obtained in  \cite{BKLW, LW15} 
using flag varieties of type $B/C$, 
generalizing the geometric realization of $\Udot$ by Beilinson-Lusztig-MacPherson \cite{BLM}. 
The $\imath$-canonical basis of $\Uidot$ admits positivity with respect to multiplication and comultiplication \cite{LW15, FL15}. 
The $\imath$-canonical basis of the quantum symmetric pair $(\U, \Ui)$ with a different choice of parameters 
was used in \cite{Bao17} to formulate the Kazhdan-Lusztig theory for category $\mathcal O$ of (super) type D.
(Connection between quantum symmetric pair $(\U, \Ui)$ and the type D category $\mc{O}$ was observed in \cite{ES13} independently from \cite{BW13}.)


\subsection{Main results} Let us provide a detailed description of the main results. 

\subsubsection{} 

The $\imath$quantum groups with parameters (and the quantum groups $\U$)
given in \cite{Le99, Ko14} are defined over a field $\mathbb K(q^{1/d})$ for a field $\mathbb K\supset \Q$ containing some suitable roots of 1 and $d>1$. 
Confirming an expectation stated in \cite{BW13}, 
Balagovic and Kolb \cite{BK14} showed the existence of a bar involution $\psi_{\imath}$ of the $\imath$quantum groups and determined the constraints on parameters. 
But for a canonical basis theory, it is more natural to work with algebras over the field $\Qq$.
In this paper we give a definition of the $\imath$quantum group $\Ui$ over the field $\Qq$ with slightly modified parameters; see Definition~\ref{def:Ui} (also compare \cite{BK15}). 
We further explain (see Lemma~\ref{lem:parameter vs}) the parameters can actually be chosen to be in $\mA =\Z[q,q^{-1}]$ 
and the bar map $\ipsi$ makes sense on the $\Qq$-form $\Ui$, as a prerequisite for the theory of $\imath$-canonical bases. 
Denote by $\psi$ the bar involution on $\U$. 

Observe that the inclusion map $\Ui \to \U$ is not compatible with the two bar maps on $\Ui$ and $\U$.  
A basic ingredient which we shall need for $\imath$-canonical basis is the intertwiner $\Upsilon$ for the quantum symmetric pair $(\U,\Ui)$, which lies in (a completion of) $\U^+$; cf. \cite[Theorem~2.10]{BW13} and \cite[Theorem~6.10]{BK15} (for a precise formulation see Theorem~\ref{thm:Upsilon} and Remark~\ref{rem:history}). The intertwiner $\Upsilon$ can be thought as the analog of Lusztig's quasi-$\mc{R}$-matrix $\Theta$, which intertwines the bar involution on $\U$ and the bar involution on $\U \otimes \U$. A twisted version of $\Theta$ is indeed the intertwiner for the QSP of diagonal type $(\U \otimes \U, \U)$, where $\U$ is realized as the subalgebra of $\U \otimes \U$ via the coproduct; see Remark~\ref{rem:QSPLu}. 


%
\subsubsection{}
Recall Lusztig \cite{Lu90, Lu94} constructed braid group operators $\T''_{w,e}$ and $\T'_{w,e}$ on $\U$. 
The braid group action  is used in the definition of $\Ui$ when $\Iblack \neq \emptyset$. 

\begin{customthm}{{\bf A}}  [Theorem~\ref{thm:braidX}] 
  \label{TA}
For any $i\in \Iblack$ and $e =\pm 1$, 
the braid group operators $\T'_{i,e}$ and $\T''_{i,e}$ restrict to  automorphisms of $\Ui$.  
\end{customthm}
In Theorem~\ref{thm:braidX} one finds explicit formulas for the actions of $\T'_{i, e}$ and $\T''_{i, e}$ on the generators of $\Ui$. Let $W$ and $W_{\Iblack}$ be the Weyl groups associated to $\I$ and $\Iblack$,  
and let $w_0$ and $w_\bullet$ denote the longest elements in $W$ and $W_{\Iblack}$, respectively. 
Different braid group action for some class of $\Ui$ has been constructed in the literature (see \cite{KP11} and references therein). Theorem~\ref{TA} verifies a conjecture of \cite{KP11} on the braid group action associated to $W_{\Iblack}$  on $\Ui$ (this conjecture was established for $\Ui$ of type AII therein).

Recall there is a well-known anti-involution $\wp$ on $\U$ \cite{Ka91, Lu94}  (see Proposition~\ref{prop:invol}), which induces a non-degenerate symmetric bilinear form on each finite-dimensional simple $\U$-module $L(\lambda)$, for $\la \in X^+$.

\begin{customprop} {\bf B} [Proposition~\ref{prop:rho}]
   \label{PB}
The anti-involution $\wp$ on $\U$ restricts to an anti-involution on $\U^\imath$.
\end{customprop}
In Proposition~\ref{prop:rho}, whose proof relies crucially on Theorem~\ref{TA}, one further finds explicit formulas for the actions of $\wp$ on the generators of $\Ui$. This result allows us to naturally use the same bilinear form on $L(\lambda)$ viewed as a $\Ui$-module. We shall see this bilinear form plays a basic role in formulating a non-degenerate bilinear form on the modified form $\Uidot$ and Theorem~\ref{TH} below.

\subsubsection{}

Recall the subspace $\U^+(w) =\U^+(w,1)$, for $w\in W$, is defined in \cite[40.2]{Lu94}. The following proposition imposes more constraint on $\Upsilon$ which will be useful later on. 

\begin{customprop} {\bf C}  [Proposition~\ref{prop:wcirc}]
\label{PC}
The intertwiner $\Upsilon$ lies in (a completion of) the subspace $\U^+(w_0 \wb)$. 
\end{customprop}

The following is one of the key properties of $\Upsilon$ in our approach toward the canonical bases for quantum symmetric pairs of finite type. 

\begin{customthm}{{\bf D}}[Theorem~\ref{thm:intUpsilon}]
 \label{TD}
The intertwiner $\Upsilon$ is integral, that is, we have $\Upsilon \in {}_\mA\widehat{\U}^+$.
\end{customthm}
Theorem~\ref{TD} is a generalization of the integrality of the quasi-$\mathcal R$-matrix for quantum groups of finite type \cite[24.1.6]{Lu94},
and in the special case of QSP of type AIII/AIV with $\Iblack =\emptyset$ it was proved in \cite{BW13}. 
The general case here takes much effort to establish. 

\subsubsection{}

The notion of a based $\U$-module $(M, B)$ with a $\U$-compatible bar involution $\psi$ is formulated in \cite[Chapter~27]{Lu94}. We are interested in considering $M$ as a $\Ui$-module by restriction, with a new bar involution $\psi_\imath := \Upsilon \circ \psi$ (see Proposition~\ref{prop:compatibleBbar}) compatible with the bar map $\ipsi$ on $\Ui$. Theorem~\ref{TD} implies that $\psi_\imath$ preserves the $\mA$-form ${}_\mA M$ of the module $M$. 

\begin{customthm}{\bf E}[Theorem~\ref{thm:iCBbased}]
 \label{TE}
The based $\U$-module $(M,B)$ admits a $\psi_\imath$-invariant basis $\{b^\imath \vert b \in B\}$,
whose transition matrix with respect to the basis $B$ is uni-triangular with off-diagonal entries in $q^{-1}\Z[q^{-1}]$. 
(We call $\{b^\imath \vert b \in B\}$ the $\imath$-canonical basis of $M$.)
\end{customthm}

By the fundamental work of Lusztig and Kashiwara \cite{Lu90, Ka91}, every finite-dimensional simple $\U$-module 
admits a canonical basis and hence is a based module. 
Similarly, by \cite{Lu92}, any tensor product of several finite-dimensional simple modules (over $\U$ of finite type) with its canonical basis is a based module. 
Hence we obtain the following corollary to Theorem~\ref{TE}. 

\begin{customcor}{\bf F}  [Theorems~\ref{thm:iCBonL} and \ref{thm:iCBtensor}]
\label{CF}
Finite-dimensional simple $\U$-modules and their tensor products admit $\imath$-canonical bases.
\end{customcor}
For type AIII/AIV with $\Iblack=\emptyset$, Theorem~\ref{TE} and Corollary~\ref{CF} were established in \cite{BW13}.
The $\imath$-canonical bases in $\mathbb V^{\otimes m} \otimes \mathbb V^{*\otimes n}$, 
where $\mathbb V$ is the natural representation of $\U$,
were used to define the Kazhdan-Lusztig polynomials for Lie superalgebras $\mathfrak{osp}$ in \cite{BW13, Bao17}. 

\subsubsection{}

In contrast to $\U$, the $\imath$quantum group $\Ui$ does not admit obvious triangular decomposition. 
So the familiar approach toward canonical bases of quantum groups, starting with $\U^-$ , is not available in the QSP setting. 
Besides, there is no obvious integral $\mA$-form of $\Ui$ in general. 

We study the modified form $\Uidot$ of $\Ui$,  similar to the modified form $\Udot$ of $\U$. 
The bar involution $\ipsi$ on $\Ui$ extends to a bar involution, again denoted by $\ipsi$, on $\Uidot$. 
Even though $\Uidot$ is not a subalgebra of $\Udot$, we can still view $\Udot$ as a  (left) $\Uidot$-module naturally. 
We define ${}_{\mA}\Uidot$ as the maximal $\mA$-subalgebra of $\Uidot$ that preserves the integral form ${}_\mA\Udot$ through the natural action. 
It is not clear at all from the definition but will be proved in the end that ${}_{\mA}\Uidot$ is a free $\mA$-module.


Recall \cite[Chapter~25]{Lu94} Lusztig's construction of canonical basis on $\Udot$ relies essentially on a projective system of based modules of the form
$^{\omega}L(\la+\nu) \otimes L(\mu+\nu)$, for $\la, \mu, \nu\in X^+$ as $\nu$ varies.  We shall formulate a generalization of such projective systems in the QSP setting. 

We denote by $\Pa$ the parabolic subalgebra of $\U$ associated with  $\Iblack$ which contains $\U^-$, and by $\Padot$ the modified form of $\Pa$. 
The intersection of the canonical basis on $\Udot$ with $\Padot$ forms the canonical basis of $\Padot$; cf. \cite{Ka94}. 
We establish a $\Qq$-linear isomorphism $\Uidot \one_{\overline{\lambda}} \cong \Padot \one_{\lambda}$ 
(where $\overline{\la}$ is an $\imath$-weight associated to $\la$;
see  \S\ref{subsec:irootdatum}), which allows one to regard $\Padot \one_{\lambda}$ as an associated graded of $\Uidot\one_{\overline{\lambda}}$.

Denote by $\eta_\la$ and $\eta_{\wb \lambda}$ the highest weight vector and the unique canonical basis element of weight $\wb \lambda$ in 
a finite-dimensional simple module $L(\lambda)$, respectively. For $\la, \mu \in X^+$, 
consider the $\U$-submodule generated by $\eta_{\wb \lambda} \otimes \eta_{\mu}$ in the tensor product $\U$-module
$L(\lambda) \otimes L(\mu)$ 
(which can be shown is the same as $\Ui$- and $\Pa$-submodule generated by $\eta_{\wb \lambda} \otimes \eta_{\mu}$): 
\[
L^\imath(\lambda,\mu)= \U(\eta_{\wb \lambda} \otimes \eta_{\mu})= \Pa(\eta_{\wb \lambda} \otimes \eta_{\mu})= \Ui(\eta_{\wb \lambda} \otimes \eta_{\mu}).
\] 
Recall $\tau$ is the Dynkin diagram involution for a Satake diagram and we set $\nu^\tau =\tau(\nu)$.
The significance of the modules $L^\imath(\lambda,\mu)$ is that 
 
 \begin{equation}
 \label{Intr:PUi}
 \Padot \one_{\wb\lambda+\mu} \cong \Uidot \one_{\overline{\wb\lambda+\mu}} 
 \cong \lim\limits_{\substack{\longleftarrow \\ \nu}} L^\imath(\lambda +\nu^\tau,\mu+\nu),
 \end{equation}
where the inverse limit is understood as $\nu \mapsto \infty$.

One observes that $L^\imath(\lambda,\mu)$ is a based $\U$-module. 
Kashiwara \cite{Ka94} further established the compatibility between the canonical basis on $\Udot$ with the canonical basis on $L^\imath(\lambda,\mu)$ 
under the obvious action, which allows a uniform parametrization of the canonical, and hence the $\imath$-canonical, basis on $L^\imath(\lambda,\mu)$. 

Toward the construction of the universal $\mc{K}$-matrix $\mc K$ for general QSP (which is an analog of Drinfeld's universal $\mathcal{R}$-matrix for $\U$), a $\Ui$-module isomorphism $\mc K'$ was defined in \cite{BK15}. 
(In cases when $\Iblack=\emptyset$ this was constructed by the authors, as it is a straightforward generalization of the construction of an isomorphism $\mathcal T$ in type AIII/AIV in \cite{BW13}.) Even though the setting for \cite{BK15} is mostly over a larger field $\mathbb K(q^{1/d})$, some flexibility in constructing $\mc K'$ allows one to choose a version, denoted by $\mc T$ in \S\ref{subsec:T}, which is defined over the field $\Qq$. Here we keep the notation $\mc{T}$, since this paper follows closely \cite{BW13}. With the help of $\mc T$, we construct in Proposition~\ref{prop:def:contraction} a unique $\Ui$-homomorphism (for $\nu \in X^+$)
\begin{equation*}
\pi =\pi_{\la, \mu, \nu}: L^{\imath} (\lambda+\nu^\tau, \mu+\nu) \longrightarrow L^{\imath} (\lambda, \mu),
 \qquad \pi(\etab_{\lambda+\nu^\tau} \otimes \eta_{\mu+\nu}) = \etab_{\lambda} \otimes \eta_{\mu}.
\end{equation*}
Hence we have constructed a projective system of $\Ui$-modules $\{ L^{\imath} (\lambda+\nu^\tau, \mu+\nu) \}_{\nu \in X^+}$. 
Lusztig's original construction is recovered in the degenerate case when $\I =\Iblack$ and $\Ui=\U$. 

However in contrast to Lusztig's results in the quantum group setting
we cannot claim  the strong form of compatibility of $\imath$-canonical bases in the sense that the projection 
$\pi_{\la, \mu, \nu}$ is a based $\Ui$-module map.
An interesting new phenomenon has already been observed in \cite[\S4.2]{BW13} for QSP 
(where two $\imath$-canonical basis elements can be mapped to the same nonzero $\imath$-canonical basis element). 

In this paper we prove an asymptotic compatibility of $\imath$-canonical bases in the projective system, that is, 
$\imath$-canonical basis elements 
are mapped to $\imath$-canonical basis elements  with the same labels (thanks to \cite{Ka94})
through the projection $\pi_{\la, \mu, \nu}$ when $\nu \to \infty$.
Together with \eqref{Intr:PUi}, this suffices to construct the $\imath$-canonical basis on $\Uidot \one_{\overline{\wb\lambda+\mu}}$. 
Actually, the $\imath$-canonical basis of $\Uidot \one_{\overline{\wb\lambda+\mu}}$ is parameterized by
the canonical basis of $\Padot \one_{\wb\lambda+\mu}$.

\begin{customthm}{\bf G}	 [Theorem~\ref{thm:iCBUi}, Corollary~\ref{cor:generator}] 
 \label{TG}
The algebra $\Uidot$ admits a unique $\imath$-canonical basis $\Bdot^\imath$, 
which is asymptotically compatible with the $\imath$-canonical basis on $L^\imath(\lambda,\mu)$, for $\la, \mu \in X^+$. 
Moreover, the basis $\Bdot^\imath$  is $\psi_\imath$-invariant, 
and  ${}_\mA \Uidot$ is a free $\mA$-module with basis $\dot{\B}^\imath$.
\end{customthm}

\subsubsection{Bilinear forms}

Recall there is a non-degenerate symmetric bilinear form \cite{Ka91, Lu94} on each 
finite-dimensional simple module $L(\la)$ defined via the anti-involution $\wp$ on $\U$, with respect to which the canonical basis is almost orthonormal. It follows by Corollary~\ref{CF} and Lusztig's results \cite[IV]{Lu94} that
the $\imath$-canonical basis on $L(\lambda) \otimes L(\mu)$ is almost orthonormal with respect to the 
tensor product bilinear form $(\cdot ,\cdot )_{\lambda,\mu}$.
We prove in Lemma~\ref{lem:bilinearform} that the bilinear form $(\cdot, \cdot)_{\lambda+\nu^\tau,\mu+\nu}$ 
converges as $\nu$ goes to $\infty$ through the projective system $\{L^\imath(\lambda+\nu^{\tau},\mu+\nu)\}_{\nu\in X^+}$, and hence the limit defines a symmetric bilinear form $(\cdot, \cdot)$ on $\Uidot$; see Definition~\ref{def:Uibilinear}. 
This form is compatible with the anti-involution $\wp$ on $\Ui$ established in Proposition~\ref{PB}; see Corollary~\ref{cor:formUi}. 

The following theorem is a generalization of a similar characterization of the signed canonical basis for
modified quantum groups \cite[Chapter~26]{Lu94}. 

\begin{customthm}{\bf H}   [Theorems~\ref{thm:orth} and \ref{thm:orth2}]
 \label{TH}
The $\imath$-canonical basis $\Bdot^\imath$ is almost orthonormal with respect to  
the symmetric bilinear form $(\cdot, \cdot)$ on $\Uidot$. 
Moreover, the signed $\imath$-canonical basis $(- \Bdot^\imath) \cup \Bdot^\imath$ is characterized by the almost orthonormality, integrality, and $\psi_\imath$-invariance. 
\end{customthm}

\subsection{Strategy of proofs}  

Recall $\U_q(\mathfrak{sl}_2)$ plays a fundamental role in the crystal and canonical basis theory of Lusztig and Kashiwara. 
To study the general quantum symmetric pairs we need to study first in depth the quantum symmetric pair of {\em real rank one} (also cf. \cite[Section 4]{Le04}). 
There are 8 different types of real rank one QSP; see Table~\ref{table:localSatake}. 
We formulate a notion of Levi subalgebras of $\Ui$ (which are $\imath$quantum groups by themselves). 
In particular, a general $\imath$quantum group  is generated by its Levi subalgebras of real and compact rank one. 
(Here a Levi subalgebra of compact rank one is a copy of $\U_q(\mathfrak{sl}_2)$ associated to any $i \in \Iblack$.)

Note that for QSP of real rank one, $\U^+(w_0 \wb)$ is a relatively small subspace of $\U^+$. 
Proposition~\ref{PC}  makes it feasible for us to prove Theorem~\ref{TD} for QSP of real rank one through brute force case-by-case computation. 
Actually we essentially obtain inductive formulas for $\Upsilon$ in the real rank one cases; see Appendix~\ref{sec:Upsilonrank1}.

The proofs of some main theorems are proceeded in the following steps:
\begin{enumerate}
	\item		 prove  Theorem~\ref{TD} for QSP of real rank one via case-by-case computations;
	\item		prove   Theorem~\ref{TE} and then Theorem~\ref{TG} for QSP of real rank one;
	\item		prove Theorem \ref{TD} for QSP of arbitrary finite type using Theorem \ref{TE} and Theorem \ref{TG} for QSP of real rank one;
	\item		prove  Theorem \ref{TE} and then Theorem \ref{TG} for QSP of arbitrary finite type; 
	\item		prove Theorem~\ref{TH} for QSP of arbitrary finite type. 
\end{enumerate}

%
\subsection{The organization}

The paper is organized as follows. 
In Section~\ref{sec:prelim}, we review various basic constructions for the quantum group $\U$. 
We study the based submodule $L^\imath (\lambda,\mu)$ of the tensor product $L(\lambda) \otimes L(\mu)$ and the parabolic subalgebra $\Pa$ of $\U$. 
We establish the compatibility between the canonical basis on the modified form $\Padot$ and the canonical basis on $L^\imath (\lambda,\mu)$.

In Section~\ref{sec:QSP}, we introduce the $\imath$-root datum associated with a Satake diagram 
and define the corresponding coideal $\Qq$-subalgebra $\Ui$ of $\U$ with parameters. 
We also define the modified form $\Uidot$ and an $\mA$-subalgebra ${}_\mA\Uidot$. 

In Section~\ref{sec:braid}, we prove the braid group operators $\T''_{w,e}$  for $w \in W_{\Iblack}$ and the anti-involution $\wp$
on $\U$ restrict to automorphisms and an anti-involution, respectively,  of $\Ui$. 
We show that $\Upsilon$ lies in (the completion of) the  subspace $\U^+ (w_0w_\bullet)$. 

In Section~\ref{sec:integral}, we prove the integrality of the intertwiner $\Upsilon$. 
The long computational proof for real rank one is given in Appendix~\ref{sec:Upsilonrank1}. 
We then establish the $\imath$-canonical bases on based $\U$-modules. 
We comment on the validity of several constructions for quantum symmetric pairs of Kac-Moody type. 
We plan to return in a future work to the construction of $\imath$-canonical bases arising from QSP of Kac-Moody type.

In Section~\ref{sec:iCBU}, 
we construct the projective system of $\Ui$-modules $\{L^\imath(\lambda +\nu^{\tau}, \mu+\nu) \}_{\nu\in X^+}$,
 prove the $\imath$-canonical basis elements stabilize when $\nu \to \infty$,
and construct the canonical basis on $\Uidot$ .  
As a consequence, we show that ${}_\mA\Uidot$ is generated by the canonical basis elements of its real and compact rank one subalgebras.
We also construct a non-degenerate symmetric bilinear form on $\Uidot$, with respect to which the signed canonical basis is 
shown to be almost orthonormal.

\vspace{.2cm}
\noindent {\bf Acknowledgements.} 
HB is partially supported by an AMS-Simons travel grant, and   
WW is partially supported by an NSF grant. 
We thank the following institutions whose support and hospitality helped to greatly facilitate the progress of this project:
East China Normal University, Institute for Advanced Study, Institute of Mathematics at Academia Sinica, and Max-Planck Institute for Mathematics.
We would like to thank Jeffrey Adams, Xuhua He, Stefan Kolb, and George Lusztig for helpful discussions, comments and their interest.
An earlier version on the $\imath$-canonical basis construction in the special cases when $\Iblack=\emptyset$ (for finite and affine types) was completed in Spring 2015. 
Balagovic and Kolb's work (in their goal of showing the universal $\mc{K}$-matrix provides solutions to the reflection equation)
helped to address several foundational issues on QSP raised in our 2013 announcement on the program of canonical basis for general QSP, 
and we in turn use their results in the current version. We warmly thank them for their valuable contributions.
We are grateful to a referee for careful readings and numerous suggestions and corrections.

\section{Quantum groups and canonical bases}
  \label{sec:prelim}

In this preliminary section, we review the basics and set up notations for quantum groups and their modified forms,  braid group actions, and canonical bases. 
We follow closely the book of Lusztig \cite{Lu94}. We also review the less familiar construction of  parabolic subalgebras and their canonical bases, following \cite{Ka94}. 
Theorem~\ref{thm:Betabullet} is new in this generality. 

\subsection{The algebras  $\f$ and $\U$}
  \label{subsec:f}
Let $(Y, X, \langle \cdot, \cdot \rangle, \cdots)$ be a root datum of finite type $(\I, \cdot)$ \cite[1.1.1, 2.2.1]{Lu94}. We have  a symmetric bilinear form
$\nu, \nu' \mapsto \nu \cdot \nu'$ on $\Z[\I]$. 
For $\mu= \sum_{i \in \I} \mu_i i \in \Z[\I]$, we let $\rm{ht}(\mu) = \sum_{i \in \I} \mu_i$.
We have an embedding $\I \subset X$ ($i \mapsto i'$), an embedding $I \subset Y$ ($i \mapsto i$) 
and a perfect pairing $\langle \cdot, \cdot \rangle: Y \times X \rightarrow \Z$ such that $\langle i, j' \rangle = \frac{2 i \cdot j}{i \cdot i}$, for $i$, $j \in \I$. 
The matrix $(\langle i, j'\rangle)= (a_{ij})$ is  
the Cartan matrix. 
We define a partial order $\leq$ on the weight lattice $X$ as follows: for $\la, \la' \in X$, 
 \begin{equation}
   \label{eq:leq}
  \lambda \le \lambda' \text{ if and only if } \lambda' -\lambda \in \N[\I]. 
 \end{equation}
Let $W$ be the corresponding Weyl group generated by the simple reflections $s_i$, for $i \in \I$. It naturally acts on $Y$ and $X$. Let $R^{\vee} \subset Y$ be the set of coroots. 
We denote by $\rho^\vee \in Y$ the half sum of all positive coroots.  Let $R  \subset X$ be the set of  roots. We denote by $\rho  \in X$ the half sum of all positive roots. We denote the longest element of $W$  by $w_0$.

Let $q$ be an indeterminate. For any $i \in \I$, we set $q_{i} = q^{\frac{i \cdot i}{2}}$. Consider a free $\Qq$-algebra $'\f$ generated by $\theta_{i}$ for ${i \in \I}$ associated with the 
Cartan datum of type $(\I, \cdot)$. As a $\Qq$-vector space, $'\f$ has a weight space decomposition as 
$'\f = \bigoplus_{\mu \in {\N}[\I]}~ '\f_{\mu},$ 
where $\theta_{i}$ has weight $i$ for all $i \in \I$.
For any $x \in \fprime_\mu$, we set $|x| = \mu$.

For each $i \in \I$, we define $r_{i}, {}_i r: \fprime \rightarrow \fprime$ to be the unique $\Q(q)$-linear maps  such that 
\begin{align}  \label{eq:rr}
\begin{split}
r_{i}(1) = 0, \quad r_{i}(\theta_{j}) = \delta_{ij},
\quad r_{i}(xx') = xr_{i}(x') + q^{i \cdot \mu'}r_{i}(x)x',
 \\
{}_{i}r(1) = 0, \quad {}_{i}r(\theta_{j}) = \delta_{ij},
\quad {}_{i}r(xx') =q^{i \cdot \mu }x \,_{i}r(x') +{ _{i}r(x)x'},
\end{split}
\end{align}
for all $x \in \fprime_{\mu}$ and $x' \in \fprime_{\mu'}$.


Let $(\cdot, \cdot)$ be the symmetric bilinear form on $\fprime$ defined in  \cite[1.2.3]{Lu94}.
Let ${\bf I}$ be the radical of the symmetric bilinear form $(\cdot,\cdot)$ on $\fprime$.
 For $i\in \I, n \in \Z$ and $s \in \N$, we define 
\[
[n]_i = \frac{q_i^n-q_i^{-n}}{q_i - q^{-1}_i} \quad  \text { and } \quad [s]^!_i = \prod^s_{j=1} [j]_i.
\] 
We shall also use the notation 
\[
\begin{bmatrix}
n\\
s
\end{bmatrix}_i
=
\frac{[n]^!_i}{[s]^!_i [n-s]^!_i}, \quad \text{ for } 0 \le s \le n.
\]
It is known \cite{Lu94} that ${\bf I}$ is 
generated by the quantum Serre relators $S(\theta_i, \theta_j)$, for $ i \neq  j \in \I$, where
\begin{equation}  \label{eq:Sij}
S(\theta_i, \theta_j) =  \sum^{1-a_{ij}}_{s=0}(-1)^{s} \begin{bmatrix}1-a_{ij}\\s \end{bmatrix}_i\theta_{i}^s \theta_{j} \theta_{i}^{1-a_{ij}-s}. 
\end{equation}
Let $\f =\fprime/\bf{I}$.
We have 
$r_{\ell}\big(S(\theta_i, \theta_j) \big) =\,_{\ell}r \big(S(\theta_i, \theta_j) \big) =0,$ for all $\ell, i , j \in \I \; (i\neq j).$ 
Hence $r_{\ell}$ and $_{\ell}r$ descend to well-defined $\Qq$-linear maps on $\f$.

We introduce the divided power $\theta^{(a)}_{i} = \theta^a_{i}/[a]_i^!$ for  $a \ge 0$. Let $\mA =\Z[q,q^{-1}]$.
Let $_\mA\f$ be the $\mA$-subalgebra of $\f$ 
generated by $\theta^{(a)}_{i}$ for various $a \ge 0$ and $i \in \I$.

Let $\U$ be the quantum group associated with the root datum $(Y, X, \langle \cdot, \cdot \rangle, \dots)$ of type $(\I, \cdot)$. The quantum group $\U$ is the associative $\Qq$-algebra generated by $E_{i}$, $F_{i}$ for $i \in \I$ and $K_{\mu}$ for $\mu \in Y$, subject to the following relations:
\begin{align*}
K_{0} =1, \qquad &K_{\mu} K_{\mu'} = K_{\mu +\mu'}, \text{ for all } \mu, \mu' \in Y,
  \\
 K_{\mu} E_{j} = q^{\langle \mu, j' \rangle} E_{j} K_{\mu}, &  \qquad 
 K_{\mu} F_{j} = q^{-\langle \mu, j' \rangle} F_{j} K_{\mu}  , \\
 E_{i} F_{j} -F_{j} E_{i} &= \delta_{i,j} \frac{\Ktilde_{i}
 -\Ktilde_{-i}}{q_i-q_i^{-1}}, \\
S(F_{i}, F_{j}) &= S(E_{i}, E_{j}) = 0, \text{ for all } i \neq j \in \I,
\end{align*}
where $\Ktilde_{\pm i} = K_{\pm \frac{i \cdot i}{2} i}$ and   
$S(E_{i}, E_{j})$ are defined as in \eqref{eq:Sij}.

Let $\U^+$, $\U^0$ and $\U^-$ be the $\Qq$-subalgebra of $\U$ generated by $E_{i} (i \in \I)$, $K_{\mu} (\mu \in Y)$, 
and $F_{i}  (i \in \I)$  respectively.
We identify $\f \cong \U^-$ by matching the 
generators $\theta_{i}$ with $F_{i}$. This identification induces
a bilinear form $(\cdot, \cdot)$ on $\U^{-}$ and $\Qq$-linear maps $r_i, {}_i r$ $(i\in \I)$ on $\U^-$.
Under this identification, we let $\U_{-\mu}^-$ be the image of $\f_\mu$.
Similarly we have $\f \cong \U^+$ by identifying $\theta_{i}$ with $E_{i}$. 
We let $_\mA\U^-$ (respectively, $_\mA\U^+$) denote the image of $_\mA\f$ under this isomorphism, 
which is generated by all divided powers $F^{(a)}_{i} =F_{i}^a/[a]_i^!$ (respectively, $E^{(a)}_{i} =E_{i}^a/[a]_i^!$). 
The coproduct $\Delta: \U \rightarrow \U \otimes \U$ is defined as follows (for $i \in \I, \mu \in Y$): 
\begin{equation}\label{eq:Delta}
\Delta(E_i)  = E_i \otimes 1 + \widetilde{K}_i \otimes E_i, \quad \Delta(F_i) = 1 \otimes F_i + F_i \otimes \widetilde{K}_{-i}, \quad \Delta(K_{\mu}) = K_{\mu} \otimes K_{\mu}.
\end{equation}
The following proposition follows by checking the generating relations, which can also be found in \cite[3.1.3,. 3.1.12, 19.1.1]{Lu94}.

\begin{prop} \label{prop:invol}
{\quad}
\begin{enumerate}
\item There is an  involution $\omega$ of the $\Qq$-algebra $\U$ such that 
$\omega(E_{i}) =F_{i}$, $\omega(F_{i}) =E_{i}$,  and $\omega(K_{\mu}) = K_{-\mu}$ for all $i \in \I$ and $\mu \in Y$.
%

\item There is an anti-involution $\wp$ of the $\Qq$-algebra $\U$ such that 
$\wp(E_{i}) = q^{-1}_i F_{i} \Ktilde_{i}$, $\wp(F_{i}) = q^{-1}_i E_{i} \Ktilde_i^{-1}$  and $\wp(K_{\mu}) = K_{\mu}$ for all $i \in \I$ and $\mu \in Y$.
%

\item	There is an anti-involution $\sigma$ of the $\Qq$-algebra $\U$ such that 
$\sigma(E_{i}) = E_{i}$, $\sigma(F_{i}) = F_{i}$  and $\sigma(K_{\mu}) = K_{-\mu}$ for all $i \in \I$ and $\mu \in Y$.
%

\item There is a bar involution $\overline{\phantom{x}}$ of the $\Q$-algebra $\U$ such that 
$q \mapsto q^{-1}$, $\ov{E}_{i}= E_{i}$, $\ov{F}_{i}=F_{i}$,  and $\ov{K}_{\mu}=K_{-\mu}$ for all $i \in \I$ and $\mu \in Y$.
%
(Sometimes we denote the bar involution on $\U$ by $\psi$.)

\end{enumerate}
\end{prop}

\subsection{Braid group actions and PBW basis}
\label{subsec:PBW}

Recall  from \cite[5.2.1]{Lu94} that for each $i \in \I$, $e \in \{\pm 1\}$, and each finite-dimensional $\U$-module $M$, linear isomorphisms 
$\T'_{i, e}$ and $\T''_{i, e}$ on $M$ are defined:  for $\lambda \in X$ and $m \in M_\lambda$, we set
\begin{align*}
\T''_{i,e}(m) &= \sum_{a, b, c \ge 0; -a+b-c
=\langle i, \lambda \rangle}(-1)^b q_i^{e(b-ac)}E^{(a)}_{i}F^{(b)}_{i}E^{(c)}_{i} m,\\
\T'_{i, e}(m) &= \sum_{a, b, c \ge 0; a-b+c =\langle i, \lambda \rangle}(-1)^b q_i^{e(b-ac)}F^{(a)}_{i}E^{(b)}_{i}F^{(c)}_{i} m.
\end{align*}
These $\T''_{i, e}$ and $\T'_{i, e}$ induce automorphisms of $\U$ in the same notations  such that, for all $u \in \U, m \in M$, we have
$
\T''_{i,e}(um)  = \T''_{i,e}(u)\T''_{i,e}(m),$ and $\T'_{i,e}(um)  = \T'_{i,e}(u)\T'_{i,e}(m).
$ 
More precisely, we have the following formulas for the actions the automorphisms $\T''_{i, e}$, $\T'_{i, e} : \U \rightarrow \U$ on generators ($i, j \in \I$, $\mu \in Y$): 
\begin{equation}\label{eq:braidgroup}
\begin{split}
&\T'_{i, e} (E_i) = - \widetilde{K}_{ei} F_i, \quad \T'_{i, e} (F_i) = -  E_i\widetilde{K}_{-ei},  \quad \T'_{i, e} (K_\mu) = K_{s_i(\mu)}; \\
&\T'_{i, e} (E_j) =  \sum_{r+s = - \langle i, j'\rangle} (-1)^r q^{er}_{i} E^{(r)}_i E_j E^{(s)}_i \quad \text{ for } j\neq i;\\ 
&\T'_{i, e} (F_j) =  \sum_{r+s = - \langle i, j'\rangle} (-1)^r q^{-er}_{i} F^{(s)}_i F_j F^{(r)}_i \quad \text{ for } j\neq i;\\ 
&\T''_{i, -e} (E_i) = -F_i \widetilde{K}_{-ei} , \quad \T''_{i, -e} (F_i) = -  \widetilde{K}_{ei}E_i,  \quad \T''_{i, -e} (K_\mu) = K_{s_i(\mu)}; \\
&\T''_{i, -e} (E_j) =  \sum_{r+s = - \langle i, j'\rangle} (-1)^r q^{er}_{i} E^{(s)}_i E_j E^{(r)}_i \quad \text{ for } j\neq i;\\ 
&\T''_{i, -e} (F_j) =  \sum_{r+s = - \langle i, j'\rangle} (-1)^r q^{-er}_{i} F^{(r)}_i F_j F^{(s)}_i \quad \text{ for } j\neq i.
\end{split}
\end{equation}

As automorphisms on $\U$ and as $\Qq$-linear isomorphisms on $M$,
these $\T''_{i, e}$ and $\T'_{i, e}$ satisfy the braid group relations (\cite[Theorem 39.4.3]{Lu94}) of type $(\I,\cdot)$.
Hence for each $w \in W$, both $\T''_{w, e}$ and $\T'_{w, e}$ can be defined independent of the choices of reduced expressions of $w$. The following proposition can be found in \cite[37.2.4]{Lu94}.

\begin{prop}\label{sec1:prop:braid} 
The following relations among $\T''_{i, e}$ and $\T'_{i, e}$ hold (for $i \in \I, e=\pm 1$): 
\begin{enumerate}
\item	 
$\omega \T'_{i, e} \omega = \T''_{i, e}  \quad \text{ and } \quad \sigma \T'_{i, e} \sigma = \T''_{i, -e} $, as automorphisms on $\U,$
\item	 
$
\overline{\T'_{i, e}(u)} = \T'_{i, -e}(\overline{u}) \quad \text{ and } \quad \overline{\T''_{i, e}(u)} = \T''_{i, -e}(\overline{u})$, for $u \in \U$. 
\end{enumerate}
\end{prop}

We shall focus on the automorphisms $\T''_{i, +1}$ and $\T''_{w, +1}$. Hence to simplify the notation, throughout the paper we shall often write 
\[
\T_i = \T''_{i, +1}, \quad \text{ and } \quad \T_w = \T''_{w, +1}.
\]

\subsection{Canonical bases}
  \label{subsec:Udot}

Let $M(\lambda)$ be the Verma module of $\U$ with highest weight $\lambda\in X$ and
with a highest weight vector denoted by $\eta$ or $\eta_{\lambda}$.  
We define a lowest weight $\U$-module $^\omega M(\lambda)$, which 
has the same underlying vector space as $M(\lambda)$ but 
with the action twisted by the involution $\omega$ given in Proposition~\ref{prop:invol}.
When considering $\eta_{\lambda}$ as a vector in $^\omega M(\lambda)$, 
we shall denote it by $\xi$ or $\xi_{-\lambda}$. 

Let 
$$X^+ = \big\{\lambda \in X \mid \langle i, \lambda \rangle \in {\N}, \forall i \in \I \big \}$$ 
 be the set of dominant integral weights. 
 By $\la \gg 0$ we shall mean that the integers $\langle i, \lambda \rangle$ for all $i$ are sufficiently large (in particular, we have $\la \in X^+$).
The Verma module $M(\lambda)$ associated to $\la \in X^+$ has a unique finite-dimensional 
simple quotient $\U$-module, denoted by $L(\lambda)$. 
Similarly we define the $\U$-module $^\omega L(\lambda)$ of lowest weight $-\la$. We have $^\omega L(-w_0 \lambda) =  L(\lambda)$.
For $\la \in X^+$, we let $_\mA L(\lambda) ={_\mA\U^-}  \eta$ and $^\omega _\mA L(\lambda) ={_\mA\U^+ } \xi$ 
be the $\mA$-submodules of $L(\lambda)$ and $^\omega L(\lambda) $, respectively. 

In \cite{Lu90} and \cite{Ka91}, the canonical basis $\bold{B}$ of ${}_\mA\f$ is constructed. 
Recall that we can identify $\f$ with both $\U^-$ and $\U^+$. 
For any element $b \in \B$, when considered as an element in $\U^-$ or $\U^+$ under such identifications, 
we shall denote it by $b^-$ or $b^+$,  respectively. 
Subsets $\B(\lambda)$ of $\B$ are also constructed for each $\lambda \in X^+$
such that $b \mapsto b^-\eta_{\lambda}$ is a bijection from $\B(\lambda)$ to the set of canonical basis of $_\mA L(\lambda)$; 
similarly $\{b^+ \xi_{-\lambda} \mid b \in \B(\lambda)\}$ gives the canonical basis of ${^{\omega}L(\lambda)}$.  
We denote by $\mc{L}(\lambda)$ the $\Z[q^{-1}]$-submodule of $L(\lambda)$ spanned by $\{b^-\eta_{\lambda} \mid b \in \B(\lambda)\}$. 
Similarly we denote by ${}^\omega \mc{L}(\lambda)$ the $\Z[q^{-1}]$-submodule of ${}^\omega L(\lambda)$ spanned by $\{b^+ \xi_{-\lambda} \mid b \in \B(\lambda)\}$.

%
%
Recall \cite[\S4.1]{Lu94} the quasi-$\mc{R}$-matrix $\Theta : = \sum_{\mu \in \N[\I]} \Theta_{\mu}$ is defined in 
   a suitable completion of $\U^- \otimes \U^+$. For any finite-dimensional $\U$-modules $M$ and $M'$, $\Theta$ is a well-defined operator on $M \otimes M'$, such that 
    \begin{equation}\label{eq:propThetaA}
    \Delta(u) \Theta (m \otimes m') 
    = \Theta \ov{\Delta ( \ov {u})} (m \otimes m') \quad  \text{ and } \quad \Theta \ov{\Theta} (m \otimes m') = m \otimes m',
    \end{equation}
    for all $m \in M$, $m' \in M'$, and $u \in \U$.

   Lusztig developed the theory of based $\U$-modules in \cite[Chapter~27]{Lu94}. The finite-dimensional simple $\U$-modules $L(\lambda)$ and ${}^\omega L(\lambda)$ are both based $\U$-modules.  Given any based $\U$-modules $M$ and $M'$, their tensor product is also a based $\U$-module with the bar involution $\psi = \Theta \circ (\psi \otimes \psi)$. 
In particular, the tensor product $L(\lambda) \otimes L(\mu) = {}^\omega L(-w_0\lambda) \otimes L(\mu)$ is a based $\U$-module for $\lambda, \mu \in X^+$ with basis $\B(\lambda,\mu)$. Elements in $\B(\lambda,\mu)$ are $\psi$-invariant and of the form 
\[
b^-_1 \eta_{\lambda} \diamondsuit b^-_{2} \eta_{\mu} \in b^-_1 \eta_{\lambda} \otimes b^-_{2} \eta_{\mu} 
+\!\!\!\!\!\! \sum_{|b'_1| \ge |b_1|, |b'_2| \le |b_2|} \!\!\! q^{-1}\Z[q^{-1}] b'^-_1 \eta_{\lambda} \otimes b'^-_2 \eta_{\mu}, b_1 \in \B(\lambda), b_2 \in \B(\mu).
\]
For convenience, we shall declare that $b^-_1 \eta_{\lambda} \diamondsuit b^-_{2} \eta_{\mu} = 0 $ whenever either $b_1 \not\in \B(\lambda)$ or $b_2 \not\in \B(\mu)$, 
or equivalently whenever either $b^-_1 \eta_{\lambda}=0$ or $b^-_{2} \eta_{\mu} = 0$ for $(b_1, b_2 )\in \B\times \B$.

For $\lambda, \mu \in X^+$, we denote by  $\mc{L}_{\lambda,\mu}$ the $\Z[q^{-1}]$-submodule of $L(\lambda) \otimes L(\mu) $ 
spanned by $b^-_1\eta_{\lambda} \otimes b^-_2 \eta_{\mu}$ for all $(b_1, b_2) \in \B(\lambda) \times \B(\mu)$. 
Let ${}_\mA \mc{L}_{\lambda,\mu} = \mA \otimes_{\Z[q^{-1}]} \mc{L}_{\lambda,\mu}$.
%
%

%
%
%
%
Recall \cite[Chapter~ 23]{Lu94} that the modified (or idempotented) form
\[
\Udot = \bigoplus_{\lambda', \lambda'' \in X} {{_{\lambda'}}\U_{\lambda''}}
\]
is naturally an associative algebra (without unit) where $\one_{\lambda} \one_{\lambda'} =\delta_{\la, \la'} \one_{\lambda}$. 
The algebra $\Udot$ admits a $(\U, \U)$-bimodule structure as well. 
Moreover, any weight $\U$-module (i.e., a $\U$-module with a direct sum decomposition into weight spaces, cf. \cite[\S3.4.1]{Lu94}) can naturally be regarded as a $\Udot$-module by \cite[\S23.1.4]{Lu94}. 
Denote by ${_{\mA}\Udot}$ the $\mA$-subalgebra of $\Udot$ generated by ${_\mA\U^{-}} \one_{\lambda} {_\mA\U^{+}}$ for various $\lambda \in X$.

For any $w \in W$ and $\lambda \in X^+$, we denote by 
$$
\eta_{w\lambda} = \xi_{w w_{0} (w_0 \lambda)} \in L(\lambda) = {}^\omega L(-w_0 \lambda),
$$ 
the unique canonical basis element of weight $w \lambda$.  We shall use later (in Proposition~\ref{prop:def:contraction}) the following slight 
generalization of \cite[27.1.7]{Lu94}. 

\begin{lem}\label{lem:chi}
 There exists a unique homomorphism of $\U$-modules 
\[
\chi: L(\lambda+\mu)  \longrightarrow L(\lambda) \otimes L(\mu)
\]
such that $\chi (\eta_{w (\lambda+\mu)} ) = \eta_{w \lambda} \otimes \eta_{w \mu}$ for any (or all) $w \in W$.
\end{lem}

\begin{proof}
We define $\chi : L(\lambda+\mu)  \longrightarrow L(\lambda) \otimes L(\mu)$ to be the 
homomorphism of $\U$-modules such that $\chi (\eta_{\lambda+\mu}) =  \eta_\lambda \otimes \eta_\mu$. 
Thanks to \cite[27.1.7]{Lu94}, $\chi$ is an homomorphism of based modules. 
For any $w \in W$,  $\eta_{w\lambda} \otimes \eta_{w\mu}$ is the unique canonical basis element 
in $L(\lambda) \otimes L(\mu)$ of weight $w (\lambda + \mu)$, since the weight subspace $\big( L(\lambda) \otimes L(\mu) \big)_{w (\lambda + \mu)}$ is one-dimensional. 

The uniqueness follows from the fact that $\eta_{w (\lambda+\mu)}$ is a cyclic vector of the $\U$-module $L(\lambda+\mu)$.
\end{proof}

We collect the following results from \cite[Chapter~25]{Lu94} in Propositions~\ref{prop:Lucontraction}-\ref{prop:CBUdot} below. 
A slight improvement here is that we can use the same $\chi$ uniformly thanks to Lemma~\ref{lem:chi}.

\begin{prop}
 \label{prop:Lucontraction}
Let $\lambda, \mu \in X^+$. 
The $\U$-module homomorphism $\chi : L(\lambda+\mu) \rightarrow L(\lambda) \otimes L(\mu)$, 
$\eta_{\lambda+\mu} \mapsto \eta_{\lambda} \otimes \eta_\mu$, satisfies the following properties: 
\begin{enumerate} 
	\item	
	Let $b \in \B(\lambda+\mu)$. We have 
$\chi (b^- \eta_{\lambda +\mu}) = \sum_{b_1, b_2} f(b; b_1, b_2) b_1^- \eta_{\lambda} \otimes b_2^- \eta_{\mu}$, 
summed over $b_1 \in \B(\lambda)$, $b_2 \in \B(\mu)$, with $f(b;b_1,b_2) \in \Z[q^{-1}]$. 
If $b^- \eta_{\mu} \neq 0$, then $f(b; 1,b) =1$ and $f(b;1,b_2) =0$ for any $b_2 \neq b$. If $b^- \eta_{\mu} = 0$, then $f(b;1,b_2) =0$ for any $b_2$;

	\item	
	Let $b \in \B(-w_0 (\lambda+\mu))$. We have 
$\chi (b^+ \xi_{w_0(\lambda +\mu)}) = \sum_{b_1, b_2} f(b; b_1, b_2) b_2^+ \xi_{w_0\lambda} \otimes b_1^+ \xi_{w_0\mu}$, 
summed over $b_1 \in \B(-w_0\lambda)$, $b_2 \in \B(-w_0\mu)$, with $f(b;b_1,b_2) \in \Z[q^{-1}]$. If $b^+ \xi_{w_0\lambda} \neq 0 $, 
then $f(b; 1, b) = 1$ and $f(b; 1, b_2) =0$ for any $b_2 \neq b$.   If $b^+ \xi_{w_0\lambda} = 0 $, then $f(b; 1, b_2) =0$ for any $b_2$.   
\end{enumerate} 
\end{prop}




\begin{prop}
 \label{prop:CBUdot}
Let $\zeta \in X$ and $b_1, b_2 \in \B$. 
	\begin{enumerate}
		\item	There exists a unique element $b_1 \diamondsuit_\zeta b_2  \in {}_\mA \Udot \one_\zeta$ such that 
		\[
			b_1 \diamondsuit_\zeta b_2 (\xi_{w_0\lambda} \otimes \eta_\mu) = (b_1^+ \xi_{w_0\lambda} \diamondsuit b_2^- \eta_{\mu}) 
		\]
		for any $\lambda, \mu \in X^+$ such that $b_1 \in \B(-w_0\lambda)$, $b_2 \in \B(\mu)$ and $w_0\lambda + \mu =\zeta$.
		
		\item	
		We have $\overline{b_1 \diamondsuit_\zeta b_2 } = b_1 \diamondsuit_\zeta b_2 $.
		
		\item	
		The set $\Bdot = \{b_1 \diamondsuit_\zeta b_2  \vert \zeta \in X, (b_1, b_2) \in \B \times \B\}$ forms a (canonical) $\Qq$-basis of $\Udot$ and a $\mA$-basis of ${}_\mA \Udot$.
	\end{enumerate}
\end{prop}

\subsection{A based submodule}
  \label{subset:parabolic}
  
 The following  is a generalization of Kashiwara's result in \cite[Lemma~8.2.1]{Ka94}.
  
\begin{thm}
  \label{thm:Betabullet}
Let $w \in W$, and $\mu, \lambda \in X^+$. For any $b \in \Bdot$, we have 
\[
b (\eta_{w\lambda} \otimes \eta_\mu)  \in \B( \lambda,\mu) \cup \{ 0\}.
\]
\end{thm}

\begin{proof}
We prove by induction on $k_w = \ell(w_0)-\ell(w) =\ell(ww_0)$.
When $k_w=0$, we have $w=w_0$, this is Lusztig's theorem (see Proposition~\ref{prop:CBUdot}). 

Assume $k_w=k>0$, and let $ww_0=s_{i_1} s_{i_2}\cdots s_{i_k}$ be a reduced expression.
We have 
\[
\eta_{w\lambda} \otimes \eta_\mu = E^{(a_1)}_{i_1} E^{(a_2)}_{i_2} \cdots E^{(a_k)}_{i_k} (\eta_{w_0\lambda} \otimes \eta_\mu),
\]
for some $a_i\ge 0$ (uniquely determined by $\la$ and $w$); moreover, we have $E_{i_1} (\eta_{w\lambda} \otimes \eta_\mu )=0$.

We define
\[
I_N = \sum_{i \in \I} \Udot E_i^N + \sum_{i \in \I} \Udot F_i^N, \quad \text{ for } N \in \N.
\]
Note that for $ x \in L_{\wb}(\lambda, \mu)$, we have $I_{N} (x) =0$ for $N \gg 0$.
Then  adapting \cite[Lemma~8.2.1]{Ka94} to our setting, we have the following two possibilities (depending on $\varepsilon^*_{i_1}(b)$ therein)
\[
b \in \Udot E_{i_1} \quad \text{ or } \quad b E^{(a_1)}_{i_1} \in b' +  \Udot E^{(a_1+1)}_{i_1} + I_N, \text{ for any } N \text{ and some } b' \in \Bdot.
\]

If $b \in \Udot E_{i_1} $, we clearly have $b (\eta_{w\lambda} \otimes \eta_\mu) =0$. 
If $b E^{(a_1)}_{i_1} \in b' +  \Udot E^{(a_1+1)}_{i_1} + I_N$, then we have (by taking $N \gg 0$)
\[
b (\eta_{w\lambda} \otimes \eta_\mu) = b' E^{(a_2)}_{i_2} \cdots E^{(a_k)}_{i_k} (\eta_{w_0\lambda} \otimes \eta_\mu)
=b'  (\eta_{s_{i_1}w} \otimes \eta_\mu).
\]
Since $\ell(s_{i_1}w w_0) =k-1$, the inductive assumption gives us $b'  (\eta_{s_{i_1}w} \otimes \eta_\mu)  \in \B( \lambda,\mu) \cup \{ 0\}.$
This proves the theorem. 
\end{proof}


Let $W_{\Iblack} = \langle s_i \in W \vert i \in \Iblack \rangle $ be the  
Weyl group associated with a subset $\Iblack \subset \I$. Let $\wb$ be the longest element in $W_{\Iblack}$. For $\lambda, \mu \in X^+$, 
we introduce the following $\U$-submodule of $L(\lambda) \otimes L(\mu)$: 
\begin{equation}
  \label{eq:Lbullet}
L^\imath(\lambda, \mu) = \U (\eta_{\wb\lambda} \otimes \eta_{\mu}).
\end{equation}
In case when $\Iblack = \I$, we have $L^\imath(\lambda, \mu) = {}^\omega L(\lambda) \otimes L(\mu)$.

\begin{cor}
 \label{cor:based}
Let $\Iblack \subset \I$ and $\lambda, \mu \in X^+$. Then the $\U$-submodule $L^\imath(\lambda, \mu)$   
is a based submodule of $L(\lambda) \otimes L(\mu)$ with canonical basis $\B(\lambda,\mu) \cap L^\imath(\lambda,\mu)$.
\end{cor}

\begin{rem}
The fact that $L^\imath(\lambda, \mu)$  
is a based submodule of $L(\lambda) \otimes L(\mu)$  can also be proved by observing that 
$$
L^\imath(\lambda,\mu) = \big( L(\lambda) \otimes L(\mu) \big) [\ge w_{\bullet}(\lambda+\mu)]
$$ 
in the spirit of \cite[\S27.1.2]{Lu94}.
\end{rem}

We denote by $\mc{L}^\imath(\lambda,\mu)$ the $\Z[q^{-1}]$-lattice of $L^\imath(\lambda,\mu)$ spanned by $\B(\lambda,\mu) \cap L^\imath(\lambda,\mu)$.

\subsection{The parabolic subalgebra $\Pa$}

For any $\Iblack \subset \I$, let $\U_{\Iblack}$ be the $\Qq$-subalgebra of $\U$ generated by $F_{i} (i \in \Iblack)$, $E_{i} (i \in \Iblack)$ and $K_{i} ( i \in \Iblack)$.  
Let $\B_{\I_\bullet}$ be the canonical basis of $\f_{\Iblack}$ (here $\f_{\Iblack}$ is simply a version of $\f$ associated to $\Iblack$). 
Then we have natural identifications $\U^{-}_{\I_\bullet} \cong \U^{+}_{\I_\bullet} \cong \f_{\Iblack}$. 
As usual, for $b \in \B_{\I_\bullet}$, we shall denote by $b^-$ the corresponding element in $\U^{-}_{\I_\bullet}$  
and denote by $b^+$ the corresponding element in $\U^{+}_{\I_\bullet}$ under such  identifications. 

Let $\Pa= \Pa_{\I_\bullet}$ be the $\Qq$-subalgebra of $\U$ generated by $\U_{\Iblack}$ and $\U^{-}$. For $\lambda \in X^+$, 
we denote by  ${}^\omega L^\bullet(\lambda)$ the $\Pa$-submodule of ${}^\omega L(\lambda)$ generated by $\xi_{-\lambda}$. 
Clearly ${}^\omega L^\bullet(\lambda)$ restricts to a simple $\U_{\Iblack}$-module with lowest weight $-\lambda$, 
and ${}^\omega L^\bullet(\lambda)$ admits a canonical basis 
$\B_{\I_\bullet}(\lambda) = \{b \in \B(\lambda) \vert b^+ \xi_{-\lambda} \neq 0\} = \B_{\I_\bullet} \cap \B(\lambda)$.

We introduce the following subalgebra of $\Udot$:
$$\Padot = \bigoplus_{\lambda \in X} \Pa \one_{\lambda}.
$$
We further set ${}_\mA \Padot = \Padot \cap {}_\mA \Udot$.

Recall the canonical basis element $b_1 \diamondsuit_{\zeta} b_2 \in \dot{\B}$ from Proposition~\ref{prop:CBUdot}. 
\begin{prop}
  \label{prop:PCB}
\cite[Theorem~3.2.1]{Ka94}
The set $\Padot \cap \dot{\B}$ forms a $\Qq$-basis of $\Padot$ and an $\mA$-basis of ${}_\mA \Padot$. 
Moreover, we have $\Padot \cap \dot{\B} =\{  b_1 \diamondsuit_{\zeta} b_2  \vert (b_1, b_2) \in \B_{\I_\bullet} \times \B, \zeta \in X \}$.
\end{prop}
We shall denote 
$$
\dot{\B}_{\Pa} =\Padot \cap \dot{\B} = \big \{  b_1 \diamondsuit_{\zeta} b_2  \vert (b_1, b_2) \in \B_{\I_\bullet} \times \B, \zeta \in X  \big \} 
$$ 
and refer to it as the canonical basis of $\Padot$. It follows by construction that
$$
\dot{\B}_{\Pa \one_\zeta} := 
\dot{\B}_{\Pa} \one_\zeta = 
\{  b_1 \diamondsuit_{\zeta} b_2  \vert (b_1, b_2) \in \B_{\I_\bullet} \times \B \}
$$ 
forms a (canonical) basis of $\Pa\one_\zeta$, for $\zeta \in X$. 

Let $\zeta', \zeta \in X$ be such that $\langle i , \zeta \rangle = \langle i, \zeta' \rangle$ for all $i \in \Iblack$. We have the following isomorphism of (left) $\Pa$-modules:
\begin{equation}
  \label{eq:pzz}
p=p_{\zeta, \zeta'}: \Pa \one_{\zeta} \longrightarrow \Pa \one_{\zeta'}
\end{equation}
such that $p ( \one_{\zeta}) =  \one_{\zeta'}$. The following proposition will be used later on. 

\begin{prop}
    \label{prop:PaCB}
Let $\zeta', \zeta \in X$ such that $\langle i , \zeta \rangle = \langle i, \zeta' \rangle$ for all $i \in \Iblack$. Then the isomorphism $p: \Pa \one_{\zeta} \rightarrow \Pa \one_{\zeta'}$ in \eqref{eq:pzz}
preserves the canonical bases. More precisely, for any $(b_1, b_2) \in \B_{\I_\bullet} \times \B$, we have 
$
 p (b_1 \diamondsuit_{\zeta} b_2) = b_1 \diamondsuit_{\zeta'} b_2.
$
\end{prop}

\begin{proof}
Let $\mu, \lambda \in X^+$ be such that $\mu -\lambda = \zeta$. Let $\lambda'=\lambda + \zeta -\zeta'$. 
Clearly by taking $\lambda \gg 0$, we can have $\lambda' \in X^+$. We shall assume $\lambda \gg 0$ and $\lambda' \in X^+$ below. 

We have a natural isomorphism of $\Pa$-modules $\gamma: {}^\omega L^\bullet(\lambda) \stackrel{\cong}{\longrightarrow}  {}^{\omega} L^\bullet(\lambda')$, 
which maps the canonical basis of ${}^\omega L^\bullet(\lambda)$ 
to the canonical basis of ${}^{\omega} L^\bullet(\lambda')$. Let us consider the induced isomorphism $\gamma \otimes \id$ on the tensor product: 
\[
\xymatrix{
{}^\omega L^\bullet(\lambda) \otimes L(\mu) \ar@{^{(}->}[d] \ar[r]^-{\gamma \otimes \id} &{}^\omega L^\bullet(\lambda') \otimes L(\mu) \ar@{^{(}->}[d] \\
{}^\omega L(\lambda) \otimes L(\mu)  & {}^\omega L(\lambda') \otimes L(\mu)
}
\]
Notice that the bar involution, $\widetilde \psi$, on ${}^\omega L(\lambda) \otimes L(\mu)  $ is defined as $\widetilde \psi = \Theta \circ (\psi \otimes \psi)$, 
where $\Theta =\sum_{\mu \in \N [\I]}\Theta^-_{\mu} \otimes \Theta^+_{\mu}$. The subspace ${}^\omega L^\bullet(\lambda) \otimes L(\mu)$ is stable 
under the bar involution $\widetilde \psi$ and clearly admits a canonical basis $\B(-w_0\lambda,\mu) \cap ({}^\omega L^\bullet(\lambda) \otimes L(\mu))$. 
Similarly  ${}^\omega L^{\bullet} (\lambda') \otimes L(\mu)  $ admits a canonical basis $\B(-w_0\lambda',\mu) \cap ({}^\omega L^\bullet(\lambda') \otimes L(\mu))$.

The actions of $\widetilde \psi$ on ${}^\omega L^\bullet(\lambda) \otimes L(\mu)$ and ${}^\omega L^\bullet(\lambda') \otimes L(\mu)$ are given by the composition
$ (\sum_{\mu \in \N [\Iblack]}\Theta^-_{\mu} \otimes \Theta^+_{\mu}) \circ (\psi \otimes \psi )$. 
Note that $\gamma \otimes \id$ intertwines the actions of $\psi  \otimes \psi $ on ${}^\omega L^\bullet(\lambda) \otimes L(\mu)$ 
and  ${}^\omega L^\bullet(\lambda') \otimes L(\mu)$. 
The map $\gamma \otimes \id$ also commutes with the operator $\sum_{\mu \in \N [\Iblack]}\Theta^-_{\mu} \otimes \Theta^+_{\mu}$, 
since we have $\Theta^{\pm}_{\mu} \in \Pa$ for $\mu \in \N[\Iblack]$. Hence the map $\gamma \otimes \id $  intertwines with the bar involutions
$\widetilde\psi$ on ${}^\omega L^\bullet(\lambda) \otimes L(\mu)$ and ${}^\omega L^\bullet(\lambda') \otimes L(\mu)$. 
It follows by the uniqueness of canonical basis that 
$\gamma \otimes \id$ maps the canonical basis on ${}^\omega L^\bullet(\lambda) \otimes L(\mu)$ to the canonical basis on ${}^\omega L^\bullet(\lambda') \otimes L(\mu)$.

The construction above works for all $\lambda, \mu \gg 0$ with $\mu -\lambda = \zeta$, and hence the proposition follows from Lusztig's construction of 
canonical basis on $\Udot$ in \cite[Theorem~25.2.1]{Lu94} as well as Proposition~\ref{prop:PCB}. 
\end{proof}


\section{Quantum symmetric pairs: definitions and first properties}
 \label{sec:QSP}

In this section we formulate quantum symmetric pairs $(\U, \Ui)$ over $\Q(q)$ and a modified form of the algebra $\Ui$. 
The quantum symmetric pairs of finite type are constructed in terms of Satake diagrams,  which specify a subset $\Iblack \subset \I$ of black nodes and a diagram involution. 
We introduce admissible subdiagrams of Satake diagrams (of real rank one) and the corresponding Levi subalgebras of $\Ui$ (of real rank one).
We formulate a connection between $\Ui$ and the parabolic subalgebra of $\U$ associated to $\Iblack$.

\subsection{The $\imath$-root datum}\label{subsec:irootdatum}

Let $(Y,X, \langle\cdot,\cdot\rangle, \cdots)$ be a root datum of type $(\I, \cdot)$.  We call a permutation $\tau$ of the set $\I$ an involution 
of the Cartan datum $(\I, \cdot)$ if $\tau^2 =\id$ and $\tau(i) \cdot \tau(j) = i \cdot j$ for $i$, $j \in \I$.  Note we allow $\tau =1$.
 We further assume that $\tau$ extends to  an involution on $X$ and an involution on $Y$, respectively, such that the perfect bilinear pairing is invariant under the involution $\tau$. 
Such involutions $\tau$ on $X$ and $Y$ exist and are unique for the simply connected or adjoint simple root datum.

For a subset $\I_{\bullet} \subset \I$, let $W_{\I_\bullet}$ be the parabolic subgroup of $W$ generated by simple reflections $s_i$ with $i \in \I_{\bullet}$. 
Let $w_{\bullet}$ be the longest element in $W_{\I_\bullet}$. 
Let $R^\vee_\bullet$ denote the set of coroots associated to the simple coroots $\I_{\bullet} \hookrightarrow Y$, and
let $R_\bullet$ denote the set of roots associated to the simple roots $\I_{\bullet} \hookrightarrow X$. 
Let $\rho^\vee_{\bullet}$ be the half sum of all positive coroots in the set $R^{\vee}_{\bullet}$, 
and let $\rho_{\bullet}$ be the half sum of all positive coroots in the set $R _{\bullet}$. 
We shall write 
\begin{equation}
\label{eq:white}
 \I_{\circ} = \I \backslash \I_{\bullet}.
\end{equation}

We recall the following definition of an admissible pair  $(\I_{\bullet}, \tau)$ (cf. \cite[Definition~2.3]{Ko14}).
\begin{definition}
\label{def:AP}
A pair $(\I_{\bullet}, \tau)$ consisting of a subset $\I_{\bullet} \subset \I$ and an involution $\tau$ of the Cartan datum  $(\I, \cdot)$
is called admissible if the following conditions are satisfied:
\begin{itemize}
	\item	[(1)]	$\tau (\I_{\bullet}) = \I_{\bullet}$; 
	\item	[(2)]	The action of $\tau$ on $\I_{\bullet}$ coincides with the action of $-w_{\bullet}$; 
	\item	[(3)]	If $j \in \I_{\circ}$ and $\tau(j) = j$, then $\langle \rho^\vee_{\bullet}, j' \rangle \in \Z$.
\end{itemize}
\end{definition}

Note that
\begin{equation}
  \label{eq:th}
 \inv =  -w_{\bullet} \circ \tau
\end{equation}
is an involution of $X$, as well as an involution of $Y$, thanks to $\tau\circ w_{\bullet} \circ \tau = w_{\bullet}$. 
(Note our convention on $\theta$ in \cite{BW13} differs by a sign from here.)

We introduce
\begin{align}
  \label{XY}
 \begin{split}
X_{{\imath}} = X /  \breve{X}, & \quad \text{ where } \; \breve{X}  = \{ \la - \inv(\la) \vert \la \in X\},
 \\
Y^{\imath} &= \{\mu \in Y \big \vert \inv(\mu) =\mu \}.
\end{split}
\end{align}

We shall call $X_{\imath}$ the $\imath$-weight lattice (even though $X_{\imath}$ is {\em not} always a lattice), 
and call $Y^{\imath}$ the $\imath$-root lattice, respectively.  For any $\la \in X$ denote its image in $X_{\imath}$ by $\overline{\la}$.
There is a well-defined bilinear pairing (denoted by $\langle \cdot, \cdot \rangle$ again, by abuse of notations)
$$ 
\langle \cdot, \cdot \rangle: 
Y^{\imath}  \times X_{\imath} \longrightarrow \Z$$
defined by $\langle \mu, \overline{\la} \rangle  : = \langle \mu, \la \rangle$, 
where $\la \in X$ is any preimage of $\overline{\la}$ and $\mu \in Y^{\imath}$.

\subsection{Satake subdiagrams of real rank one}\label{subset:rankone}

According to Kolb \cite{Ko14}, the admissible pairs of finite type (excluding the trivial case when $\I =\Iblack$) are in bijection with the 
Satake diagrams \cite{Araki} arising from classification of real simple Lie algebras. 
(Beware that there is a hidden involution on the black dots when the number of black nodes is odd for type DI/DII.)
These Satake diagrams,  consist of black and white nodes with arrows;
the set of black nodes corresponds to $\I_{\bullet}$ and the involution $\tau$ is expressed in terms of $2$-sided arrows on white nodes. 
We reproduce from {\em loc. cit.} the Satake diagrams in Table~\ref{table:Satake} 
at the end of this paper  for the reader's convenience.  
For the rest of this section, we consider root data $(Y,X, \langle\cdot,\cdot\rangle, \ldots)$ of finite type $(\I, \cdot)$ and an admissible pair $(\I_{\bullet}, \tau)$. 
The number of $\tau$-orbits of white nodes in a Satake diagram is called its {\em real rank}. 

\begin{definition} 
 \label{def:rank1}
Let $D$ be a Satake diagram with a set $\I_\circ$ of white nodes. 
Given a $\langle\tau\rangle$-orbit $\mathbf o$ of white nodes
in $D$, the removal of  all white nodes in $\I_\circ \backslash \mathbf o$ and their adjacent edges in $D$ produces a diagram $D_{\mathbf o}$.
The connected subdiagram of $D_{\mathbf o}$ containing $\mathbf o$ is call a subdiagram of {\em real rank one} (associated to $\mathbf o$). 
\end{definition}
By definition, subdiagrams of real rank one of a Satake diagram $D$ are parametrized by the $\langle\tau\rangle$-orbits of white nodes of $D$.

\begin{example}
The Satake diagram of type $EIII$
$
\begin{tikzpicture}[baseline = 8, scale =1.0]
		\node at (-1,0.5) {$\circ$};
		\draw (-0.9,0.5) to (-0.6,0.5);
		\node at (-0.5,0.5) {$\bullet$};
		\draw (-0.4,0.5) to (-0.1,0.5);
		\node at (0,0.5) {$\bullet$};
		\draw (0.1,0.5) to (0.4,0.5);
		\node at (0.5,0.5) {$\bullet$};
		\draw (0.6,0.5) to (0.9,0.5);
		\node at (1,0.5) {$\circ$};
		\draw (0,0.4) to (0,0.1);
		\node at (0,-0) {$\circ$};
		\draw[<->] (-0.9, 0.6) to[out=45, in=180] (-0.5, 0.8) to (0.5, 0.8) to[out=0, in=135] (0.9, 0.6);
	\end{tikzpicture}
$
has 2 subdiagrams of real rank one:
$
\begin{tikzpicture}[baseline = 8, scale =0.8]
		\node at (-1,0.4) {$\circ$};
		\draw (-0.9,0.4) to (-0.6,0.4);
		\node at (-0.5,0.4) {$\bullet$};
		\draw (-0.4,0.4) to (-0.1,0.4);
		\node at (0,0.4) {$\bullet$};
		\draw (0.1,0.4) to (0.4,0.4);
		\node at (0.5,0.4) {$\bullet$};
		\draw (0.6,0.4) to (0.9,0.4);
		\node at (1,0.4) {$\circ$};
		\draw[<->] (-0.9, 0.5) to[out=45, in=180] (-0.5, 0.7) to (0.5, 0.7) to[out=0, in=135] (0.9, 0.5);
	\end{tikzpicture}
$
of type AIII,
and 
$\begin{tikzpicture}[baseline = 6, scale =0.8]
		\node at (-0.5,0.5) {$\bullet$};
		\draw (-0.4,0.5) to (-0.1,0.5);
		\node at (0,0.5) {$\bullet$};
		\draw (0.1,0.5) to (0.4,0.5);
		\node at (0.5,0.5) {$\bullet$};
		\draw (0,0.4) to (0,0.1);
		\node at (0, 0) {$\circ$};
	\end{tikzpicture}
$
of type DI.
\end{example}

By inspection, there are eight types of local configurations of Satake diagrams of real rank one as listed in Table~\ref{table:localSatake}. 
Note there are no black nodes and the two white nodes are connected in the type AIV for $n=2$, and so it differs from type AIII$_{11}$.

In Table~\ref{table:decomposition} (where SP stands for symmetric pairs), 
for each Satake diagram we list the possible Satake subdiargrams of real rank one in the sense of Definition~\ref{def:rank1}.
For several crucial arguments in this paper we shall reduce to the types of real rank one and do case-by-case analysis.

\begin{table}[h]
\caption{Satake diagrams of  
 symmetric pairs of real rank one}
\label{table:localSatake}
\begin{tabular}{| c | c || c | c |}
\hline
\begin{tikzpicture}[baseline=0]
\node at (0, -0.15) {AI$_1$};
\end{tikzpicture} 
& 
$	\begin{tikzpicture}[baseline=0]
		\node  at (0,0) {$\circ$};
		\node  at (0,-.3) {\small 1};
	\end{tikzpicture}
$ 
&
\begin{tikzpicture}[baseline=0]
\node at (0, -0.15) {AII$_3$};
\end{tikzpicture}
&
$	\begin{tikzpicture}[baseline=0]
		\node at (0,0) {$\bullet$};
		\draw (0.1, 0) to (0.4,0);
		\node  at (0.5,0) {$\circ$};
		\draw (0.6, 0) to (0.9,0);
		\node at (1,0) {$\bullet$};
		\node at (0,-.3) {\small 1};
		\node  at (0.5,-.3) {\small 2};
		\node at (1,-.3) {\small 3};
			\end{tikzpicture}	
$

\\
\hline
\begin{tikzpicture}[baseline=0]
\node at (0, -0.2) {AIII$_{11}$};
\end{tikzpicture}
 &
\begin{tikzpicture}[baseline = 6] 
		\node at (-0.5,0) {$\circ$};
		\node at (0.5,0) {$\circ$};
		\draw[<->] (-0.5, 0.2) to[out=45, in=180] (-0.15, 0.35) to (0.15, 0.35) to[out=0, in=135] (0.5, 0.2);
		\node at (-0.5,-0.3) {\small 1};
		\node at (0.5,-0.3) {\small 2};
	\end{tikzpicture}
	&
\begin{tikzpicture}[baseline=0]
\node at (0, -0.2) {AIV, n$\ge$2};
\end{tikzpicture} &
\begin{tikzpicture}	[baseline=6]
		\node at (-0.5,0) {$\circ$};
		\draw[-] (-0.4,0) to (-0.1, 0);
		\node  at (0,0) {$\bullet$};
		\node at (2,0) {$\bullet$};
		\node at (2.5,0) {$\circ$};
		\draw[-] (0.1, 0) to (0.5,0);
		\draw[dashed] (0.5,0) to (1.4,0);
		\draw[-] (1.6,0)  to (1.9,0);
		\draw[-] (2.1,0) to (2.4,0);
		\draw[<->] (-0.5, 0.2) to[out=45, in=180] (0, 0.35) to (2, 0.35) to[out=0, in=135] (2.5, 0.2);
		\node at (-0.5,-.3) {\small 1};
		\node  at (0,-.3) {\small 2};
		\node at (2.5,-.3) {\small n};
	\end{tikzpicture}
\\
\hline
\begin{tikzpicture}[baseline=0]
\node at (0, -0.2) {BII, n$\ge$ 2};
\end{tikzpicture} & 
    	\begin{tikzpicture}[baseline=0, scale=1.5]
		\node at (1.05,0) {$\circ$};
		\node at (1.5,0) {$\bullet$};
		\draw[-] (1.1,0)  to (1.4,0);
		\draw[-] (1.4,0) to (1.9, 0);
		\draw[dashed] (1.9,0) to (2.7,0);
		\draw[-] (2.7,0) to (2.9, 0);
		\node at (3,0) {$\bullet$};
		\draw[-implies, double equal sign distance]  (3.1, 0) to (3.7, 0);
		\node at (3.8,0) {$\bullet$};
		\node at (1,-.2) {\small 1};
		\node at (1.5,-.2) {\small 2};
		\node at (3.8,-.2) {\small n};
	\end{tikzpicture}	
&
\begin{tikzpicture}[baseline=0]
\node at (0, -0.15) {CII, n$\ge$3};
\end{tikzpicture}  
& 
		\begin{tikzpicture}[baseline=6]
		\draw (0.6, 0.15) to (0.9, 0.15);
		\node  at (0.5,0.15) {$\bullet$};
		\node at (1,0.15) {$\circ$};
		\node at (1.5,0.15) {$\bullet$};
		\draw[-] (1.1,0.15)  to (1.4,0.15);
		\draw[-] (1.4,0.15) to (1.9, 0.15);
		\draw (1.9, 0.15) to (2.1, 0.15);
		\draw[dashed] (2.1,0.15) to (2.7,0.15);
		\draw[-] (2.7,0.15) to (2.9, 0.15);
		\node at (3,0.15) {$\bullet$};
		\draw[implies-, double equal sign distance]  (3.1, 0.15) to (3.7, 0.15);
		\node at (3.8,0.15) {$\bullet$};
		\node  at (0.5,-0.15) {\small 1};
		\node at (1,-0.15) {\small 2};
		\node at (3.8,-0.15) {\small n};
	\end{tikzpicture}		
\\
\hline
\begin{tikzpicture}[baseline=0]
\node at (0, -0.05) {DII, n$\ge$4};
\end{tikzpicture}&
	\begin{tikzpicture}[baseline=0]
		\node at (1,0) {$\circ$};
		\node at (1.5,0) {$\bullet$};
		\draw[-] (1.1,0)  to (1.4,0);
		\draw[-] (1.4,0) to (1.9, 0);
		\draw[dashed] (1.9,0) to (2.7,0);
		\draw[-] (2.7,0) to (2.9, 0);
		\node at (3,0) {$\bullet$};
		\node at (3.5, 0.4) {$\bullet$};
		\node at (3.5, -0.4) {$\bullet$};
		\draw (3.1, 0.1) to (3.4, 0.35);
		\draw (3.1, -0.1) to (3.4, -0.35);
		\node at (1,-.3) {\small 1};
		\node at (1.5,-.3) {\small 2};
		\node at (3.5, 0.7) {\small n-1};
		\node at (3.5, -0.6) {\small n};
	\end{tikzpicture}		
&
\begin{tikzpicture}[baseline=0]
\node at (0, -0.2) {FII};
\end{tikzpicture}&
\begin{tikzpicture}[baseline=0][scale=1.5]
	\node at (0,0) {$\bullet$};
	\draw (0.1, 0) to (0.4,0);
	\node at (0.5,0) {$\bullet$};
	\draw[-implies, double equal sign distance]  (0.6, 0) to (1.2,0);
	\node at (1.3,0) {$\bullet$};
	\draw (1.4, 0) to (1.7,0);
	\node at (1.8,0) {$\circ$};
	\node at (0,-.3) {\small 1};
	\node at (0.5,-.3) {\small 2};
	\node at (1.3,-.3) {\small 3};
	\node at (1.8,-.3) {\small 4};
\end{tikzpicture}
\\
\hline
\end{tabular}
\newline
\smallskip
\end{table}

\vspace{.6cm}

\begin{table}[h]
\caption{Subdiagrams of real rank one in Satake diagrams}
\label{table:decomposition}
\begin{tabular}{| c || c | c | c | c |c | c | }
\hline
\text{\blue{SP Type}} & \blue{AI}  & \blue{AII} & \blue{AIII} & \blue{AIV} & \blue{BI} & \blue{BII}  
\\
\hline
\text{Local} & AI$_1$  & AII$_3$ & AI$_1$, AIII$_{11}$, AIV  & AIV &AI$_1$, BII  & BII  
\\
\hline\hline
\text{\blue{SP Type}} & \blue{CI} & \blue{CII} & \blue{DI} & \blue{DII} & \blue{DIII} & \blue{EI}  
\\
\hline
\text{Local} & AI$_1$ & AII$_3$,  BII 
, CII &AI$_1$, AIII$_{11}$, DII  & DII & AI$_1$, AII$_3$ &AI$_1$   
\\
\hline\hline
\text{\blue{SP Type}} & \blue{EII} & \blue{EIII} & \blue{EIV} & \blue{EV} & \blue{EVI} & \blue{EVII}  

\\
\hline
 \text{Local}& AI$_1$, AIII$_{11}$  &AI$_1$, AIV & DII 
  &AI$_1$  & AI$_1$,  AII$_3$ &  AI$_1$, DII 
  \\
\hline\hline
\text{\blue{SP Type}}  & \blue{EVIII} & \blue{EIX} & \blue{FI} & \blue{FII} & \blue{G} &
\\
\hline
 \text{Local}   &  AI$_1$ &  AI$_1$, DII 
 & AI$_1$ & FII & AI$_1$ &

\\
\hline

\end{tabular}
\end{table}

In analogy with subdiagrams of real rank one, we call a single black node of a Satake diagram a subdiagram of {\em compact rank one}.

\begin{definition}  
 \label{def:AdmSub}
An {\em admissible subdiagram} of a Satake diagram $D$ is a full subdiagram whose vertex set 
is the union of subdiagrams of compact rank one and subdiagrams of real rank one of $D$. 

(Hence there are two kinds of {\em minimal} admissible subdiagrams of a Satake diagram: (i)~ subdiagrams of real rank one,  and (ii)~ subdiagrams of compact rank one.)
\end{definition}


\subsection{The $\imath$quantum group $\Ui$ over $\Qq$}
The permutation $\tau$ of $\I$ induces an isomorphism of $\U$, denoted also by $\tau$,
which sends $E_i \mapsto E_{\tau i}, F_i \mapsto F_{\tau i}$, and $K_\mu \mapsto K_{\tau \mu}$. 
Let 
$$
\theta :=  \T_{w_{\bullet}} \circ \tau \circ \omega
$$ 
be an automorphism of $\U$. 
(As it will not cause confusion, here we  abuse notation to use the same $\theta$ as in \eqref{eq:th}, which is actually a shadow of the current involution on the (co)weight level.)

\begin{definition}\label{def:Ui}
The algebra $\Ui$, with parameters 
\begin{equation}
  \label{parameters}
\vs_{i} \in \pm q^\Z, \quad \kappa_i \in \Z[q,q^{-1}], \qquad \text{ for }  i \in \I_{\circ},
\end{equation}
is the $\Qq$-subalgebra of $\U$ generated by the following elements:
\begin{align*}
F_{i}  &+ \vs_i \T_{w_{\bullet}} (E_{\tau i}) \widetilde{K}^{-1}_i 
+ \kappa_i \widetilde{K}^{-1}_{i} \, (i \in \I_{\circ}), 
 \\
& \quad K_{\mu} \,(\mu \in Y^{\imath}), \quad F_i \,(i \in \I_{\bullet}), \quad E_{i} \,(i \in \I_{\bullet}).
\end{align*}
The parameters are required to satisfy Conditions \eqref{kappa}-\eqref{vs2}: 
\begin{align}
 \label{kappa}
 \begin{split}
\kappa_i &=0 \; \text{ unless } \tau(i) =i, \langle i, j' \rangle = 0 \; \forall j \in \Iblack,
\\
&  
\qquad\quad
\text{ and } \langle k,i' \rangle \in 2\Z \; \forall k = \tau(k) \in \Iwhite \text{ such that } \langle k, j' \rangle = 0 \; \text{ for all } j \in \Iblack;
\end{split}
\\
\overline{\kappa_i} &= \kappa_i;   \label{kappa2}   
\\
\vs_{i} & =\vs_{{\tau i}} \text{ if }    i \cdot \theta (i) =0;
\label{vs=}
\\
\vs_{i} \cdot  \vs_{{\tau i}} &= (-1)^{ \langle 2\rho^\vee_{\bullet},  i' \rangle } q_i^{-\langle i, 2\rho_{\bullet}+\wb\tau i ' \rangle}. 
  \label{vs2}
\end{align}
\end{definition}

By definition, $\Ui$ contains $\U_{\I_\bullet}$ as a subalgebra. We have 
$$
\Delta: \Ui \longrightarrow \Ui \otimes \U,
$$
that is, $\Ui$ is a (right) coideal subalgebra of $\U$. 
The pair $(\U, \Ui)$ is called a {\em quantum symmetric pair}, as its $q\mapsto 1$ limit is the classical symmetric pair (cf. \cite{Araki, OV} and references therein). 
The algebra $\Ui$ on its own will be also referred to as the {\em $\imath$quantum group}.

\begin{rem}
The foundation of quantum symmetric pairs was established by G.~Letzter \cite{Le02, Le03} and Kolb  \cite{Ko14} (also see \cite{BK15}). 
We refer to these papers and the references therein for more original motivations and historical remarks. 
In the literature, the $\imath$quantum group $\Ui$ was defined over some field $\mathbb K(q^{\frac1d})$ with 
$d>1$ and a field $\mathbb K \supset \mathbb Q$ containing some roots of 1 
of characteristic zero (see \cite[Remark~2.3]{BK15}). To develop a theory of canonical basis, 
it is natural to formulate the algebra $\Ui$ over the field $\Q(q)$ 
as we did in \cite{BW13}; this is made possible by \cite[Remark~5.2 and its preceding paragraph, \S5.4]{BK15}. 
\end{rem}

\begin{rem}
 \label{rem:BKvsBW}
 The precise relations between constraints for parameters for $\Ui$ in this paper and in \cite{BK15} are as follows. 
Our notations correspond to those in \cite[\S5]{BK15} in the following way: $\kappa_i \leftrightarrow s_i$,  $\vs_i \leftrightarrow - c_i s(\tau(i))$. 
Our Condition~\eqref{kappa} is \cite[(5.7)]{BK15}, 
Condition~\eqref{kappa2} is \cite[(5.15)]{BK15}, and Condition~\eqref{vs=} is \cite[(5.6)]{BK15} where we use \cite[(5.2)]{BK15}. 
In particular, our formulation does not use their parameters $c_i, s(\tau(i))$ separately. 

Condition~\eqref{vs2} (in the presence of \eqref{parameters}) implies  (but is {\em inequivalent} to) \cite[(5.16)]{BK15}, if we take into account \cite[(5.1)-(5.2), Remark~5.2]{BK15}. 
More precisely, our condition follows from theirs by imposing the additional condition that their $c_i$ (or our $\vs_i$) is a monomial in $q$ and hence $\overline{c_i} =c_i^{-1}$. 
Our stronger Condition~\eqref{parameters} is needed for the integral form of $\Uidot$ and the validity of Proposition~\ref{prop:rho} below (Remark~\ref{rem:characterization}). 
Lemma~\ref{lem:parameter vs} below 
computes the values 
of the parameters $\vs_i$ for $i\in \Iwhite$ satisfying \eqref{parameters} and \eqref{vs=}--\eqref{vs2}. 
\end{rem}

We shall write 
\begin{equation}\label{eq:def:ff}
\ff_i = 
\begin{cases}
F_{i} + \vs_i \T_{w_{\bullet}} (E_{\tau i}) \widetilde{K}^{-1}_i + \kappa_i \widetilde{K}^{-1}_{i} & \text{ if } i \in \I_{\circ};
\\
 F_{i} & \text{ if } i \in \I_{\bullet}.
 \end{cases}
\end{equation}

As can be found in \cite[(7.3)]{Ko14},
it follows by Definition~\ref{def:Ui} that, for $i \in \Iblack$ and $j \in \I$, 
\begin{equation}
 \label{eq;Eifj}
E_i \ff_j  - \ff_j E_i = \delta_{i,j} \frac{\tK_{i} - \tK_{-i}}{q_i -q^{-1}_i}.
\end{equation}

\begin{rem}
A presentation of the algebra $\Ui$ with generators $\ff_i (i \in \I)$, $K_{\mu} \,(\mu \in Y^{\imath})$ and $E_{i} \,(i \in \I_{\bullet})$ 
has been obtained in \cite{Le02, Ko14} (see \cite[Section~3.2]{BK14}).  
\end{rem}

Let $\Ui^{-}$ be the subalgebra of $\Ui$ generated by $\ff_i$ for $i \in \I$. Let $\Ui^0$ be the subalgebra of $\Ui$ generated by $K_{\mu} \,(\mu \in Y^{\imath})$. Let $\Ui^+$ be the subalgebra of $\Ui$ generated by $E_i$ for $i \in \I_{\bullet}$. 

\begin{rem}
The following multiplication map is only surjective, but not an isomorphism: 
\[
m : \Ui^- \otimes \Ui^0 \otimes \Ui^+ \longrightarrow \Ui.
\]
(Note that $\Ui^+ = \U_{\I_\bullet}^+$.)
It is possible to replace $\Ui^-$ by certain subspace of $\Ui^-$, such that the above map becomes an isomorphism following \cite[Proposition~6.1\&6.2]{Ko14}.
 \end{rem}
 
 We call a $\Ui$-module $M$ a weight $\Ui$-module, if $M$ admits a direct sum decomposition $M = \oplus_{\lambda \in X_\imath} M_{\lambda}$ such that, for any $\mu \in Y^\imath$, $\lambda \in X_\imath$, $m \in M_{\lambda}$, we have $K_{\mu} m = q^{\langle \mu, \lambda \rangle} m$. We shall only consider weight $\Ui$-modules in this paper.
 
\subsection{Parameters}

By Remark~\ref{rem:BKvsBW}, our parameters satisfy stronger constraints than those in \cite{BK15}, and in next lemma we ensure  
the existence of solutions of $\vs_{i} \in \pm q^{\Z}$ ($i \in \I_{\circ}$) which satisfy  Conditions \eqref{vs=}--\eqref{vs2}. As these conditions are local,
it suffices to consider the Satake (sub)diagrams of real rank one  in Table~\ref{table:values}.

\begin{lem}
 \label{lem:parameter vs}
The values of $\vs_i$ for  quantum symmetric pairs of real rank one are  given in Table~\ref{table:values}.
\begin{table}[h]
\caption{Values of $\vs_i$ $(i \in \Iwhite)$ for quantum symmetric pairs of  real rank one}
\label{table:values}
\begin{tabular}{| c | c | c | c | c | c | c | c |}
\hline
\begin{tikzpicture}[baseline=0]
\node at (0, 0.2) {AI$_1$};
\end{tikzpicture} 
&
\begin{tikzpicture}[baseline=0]
\node at (0, 0.2) {AII$_3$};
\end{tikzpicture}
&
\begin{tikzpicture}[baseline=0]
\node at (0, 0.2) {AIII$_{11}$};
\end{tikzpicture}
	&
\begin{tikzpicture}[baseline=0]
\node at (0, 0.2) {AIV, n$\ge$2};
\end{tikzpicture} 
&
\begin{tikzpicture}[baseline=0]
\node at (0, 0.2) {BII, n$\ge$ 2};
\end{tikzpicture} 
& 
\begin{tikzpicture}[baseline=0]
\node at (0, 0.2) {CII, n$\ge$3};
\end{tikzpicture}  
& 
\begin{tikzpicture}[baseline=0]
\node at (0, 0.2) {DII, n$\ge$4};
\end{tikzpicture}
&
\begin{tikzpicture}[baseline=0]
\node at (0, 0.2) {FII};
\end{tikzpicture}
\\
\hline
$\pm q_1^{-1}$
&
$\pm q$
 &
\eqref{eq:vs12}
 &
 \eqref{eq:vs1n}
 & 
    	$\pm q^{2n-3}$
&
$\pm q^{n-1}$
 &
$\pm q^{n-2}$	
&
$\pm q^5$
\\
\hline
\end{tabular}
\newline
\end{table}
\end{lem}

\begin{proof}
We shall compute $(-1)^{ \langle 2\rho^\vee_{\bullet},  i \rangle } q_i^{-\langle i, 2\rho_{\bullet}+\wb\tau i \rangle}$ in \eqref{vs2}
case by case following Table~\ref{table:localSatake} and the labeling therein. 
For $i \in \I$, we sometimes use the notation $\alpha_i$ (instead of $i'$)  for the corresponding element  in $\I \subset X$ 
and use the notation $\alpha^\vee_i$ (instead of $i$)  for the corresponding element in $\I \subset Y$. 
\begin{enumerate}
  \item[(AI$_1$)]		
  We have $(-1)^{ \langle 2\rho^\vee_{\bullet},  i \rangle } q_i^{-\langle i, 2\rho_{\bullet}+\wb\tau i \rangle} = q_1^{-2}$. 
  Then we can clearly take $\vs_1 = \pm q_1^{-1}$.
  
 \item[(AII$_3$)]	
 We have $\rho^\vee_\bullet = \hf \alpha^\vee_1 + \hf \alpha^\vee_3$, $\rho_\bullet = \hf \alpha_1 + \hf \alpha_3$ 
 and $\wb\tau (\alpha_2) = \alpha_1 +\alpha_2+ \alpha_3$. Therefore we have $\langle 2\rho^\vee_{\bullet},  \alpha_2 \rangle = -2$ 
 and $-\langle \alpha^\vee_2, 2\rho_{\bullet}+\wb\tau  \alpha_2 \rangle = 2$. Hence \eqref{vs2} becomes $\vs_2^2 =q_2^2$
 and we can take $\vs_2 = \pm q$ (by noting $q_2=q$).
 
 \item[(AIII$_{11}$)]  
 Note $\tau (\alpha_1 ) = \alpha_2$, $q_1=q$. Condition~ \eqref{vs=} applies in this case and gives us $\vs_1 =\vs_2$. Also it follows by \eqref{vs2} that 
$ \vs_1 \vs_2 = 1. 
$ 
Hence \begin{equation}
 \label{eq:vs12}
 \vs_1 = \vs_2 =\pm 1.
 \end{equation} 
 
  \item[(AIV)]  
 Note $\tau (\alpha_1 ) = \alpha_n$,  $q_1=q$.  We have $\langle 2\rho^\vee_{\bullet},  \alpha_1 \rangle = \langle \alpha_1^\vee, 2\rho_{\bullet} \rangle =2-n$, and
 $\langle \alpha_1^\vee, \wb\alpha_n \rangle =-1$.  Condition~ \eqref{vs=} does not apply. It  follows by \eqref{vs2} that 
$
  \vs_1 \vs_n  
= (-1)^{n} q^{n-1}. 
$ 
Hence we can choose 
\begin{equation}\label{eq:vs1n} 
\vs_1 = (-1)^{a} q^{b}, \quad \vs_n = (-1)^{n-a}q^{n-1-b}, \quad \text{ for any }a, b \in \Z.
\end{equation} 
  
 \item[(BII)]		
 Let $n\ge 2$. We have
		\begin{align*}
			\rho^\vee_{\bullet} &=  \sum^{n-1}_{i=2} \frac{(2n-i)(i-1)}{2}\alpha^\vee_{i}+  \frac{n(n-1)}{4} \alpha^\vee_n,\\
			\rho _{\bullet} & = \sum^{n-1}_{i=2} \frac{(2n-i-1)(i-1)}{2}\alpha_i + \frac{(n-1)^2}{2} \alpha_n,\\
			\wb\tau(\alpha_1) &= \alpha_1 +  2\alpha_2  + \cdots +  2\alpha_{n-1} + 2\alpha_n.
		\end{align*}
     Therefore we have $\langle 2\rho^\vee_{\bullet},  \alpha_1 \rangle = - (2n-2)$, $-\langle \alpha^\vee_1, 2\rho_{\bullet}+\wb\tau  \alpha_1 \rangle =  2n-3$, 
     and \eqref{vs2} becomes $\vs_1^2 =q_1^{2n-3}$. 
     Since $q_1 = q^2$, we can take $\vs_1 =  \pm q^{2n-3}$.
	
  \item[(CII)] 		
  Let $n\ge 3$. We have
	              %
				\begin{align*}
					\rho^\vee_{\bullet} &=\hf \alpha^\vee_1 + \sum^{n-1}_{i=3} \frac{(2n-i-2)(i-2)}{2}\alpha^\vee_i + \frac{(n-2)^2}{2} \alpha^\vee_n,\\
					\rho _{\bullet} & = \hf \alpha_1 +  \sum^{n-1}_{i=3} \frac{(2n-i-1)(i-2)}{2}\alpha_{i}+  \frac{(n-1)(n-2)}{4} \alpha_n,,\\
					\wb\tau(\alpha_2) &=  \alpha_1 +\alpha_2+  2\alpha_3  + \cdots +  2\alpha_{n-1} + 2\alpha_n.
				\end{align*}
		 Therefore we have $\langle 2\rho^\vee_{\bullet},  \alpha_2 \rangle = -2n+4$ and $-\langle \alpha^\vee_2, 2\rho_{\bullet}+\wb\tau  \alpha_2 \rangle =  2n-2$.    
		 So \eqref{vs2} becomes $\vs_2^2 =q_2^{2n-2} =q^{2n-2}$, and we can take $\vs_2 = \pm q^{n-1}$.
		
  \item[(DII)]		
  Let $n\ge 4$. We have 
		\begin{align*}
			\rho^\vee_\bullet &= (n-2) \alpha^\vee_2  + \cdots + \frac{(n-1)(n-2)}{2} 
			\alpha^\vee_{n-2} + \frac{n(n-1)}{4} \alpha^\vee_{n-1} + \frac{n(n-1)}{4} \alpha^\vee_{n},\\
			\rho_\bullet &= (n-2) \alpha_2  + \cdots + \frac{(n-1)(n-2)}{2} \alpha_{n-2} + \frac{n(n-1)}{4} \alpha_{n-1} + \frac{n(n-1)}{4} \alpha_{n},\\
			\wb\tau (\alpha_1) &= \alpha_1 + 2\alpha_2 + \cdots 2 \alpha_{n-2} + \alpha_{n-1} +\alpha_{n}.
		\end{align*}
		Therefore we have $\langle 2\rho^\vee_\bullet, \alpha_1 \rangle = -(n-2)$ and $-\langle \alpha^\vee_1, 2\rho_{\bullet}+\wb\tau  \alpha_1 \rangle = 2n-4 $. 
		So \eqref{vs2} becomes $\vs_1^2 =q^{2n-4}$, and we can take $\vs_1 = \pm q^{n-2} $.
		
\item[(FII)] We have 
		\begin{align*}
			\rho^\vee_{\bullet} &= 3 \alpha^\vee_1 + 5\alpha^\vee_2    +3\alpha^\vee_3,\\
			\rho _{\bullet} & = \frac{5}{2} \alpha_1 + 4 \alpha_2 + \frac{9}{2} \alpha_3,\\
			\wb\tau(\alpha_4) &= \alpha_1 +  2\alpha_2  + 3\alpha_{3} + \alpha_4.
		\end{align*}
Therefore we have $\langle 2\rho^\vee_\bullet, \alpha_4 \rangle = -6$ and $-\langle \alpha^\vee_4, 2\rho_{\bullet}+\wb\tau  \alpha_1 \rangle = 10$. 
So \eqref{vs2} becomes $\vs_4^2 =q_4^{10} =q^{10}$, and we can take $\vs_4 =\pm q^{5} $.
		\end{enumerate}
		This finishes the proof.
\end{proof}

\begin{rem}
  \label{lem:nonzero}
For $\Ui$ of finite type associated to the Satake diagrams in Table~\ref{table:Satake}, we have $\kappa_i=0$ $(i \in \Iwhite)$ except in four cases (cf. \cite{Le03}):
(1) type AIII with $\Iblack =\emptyset$ and with $i$ being the middle node fixed by $\tau$ (we regard AI$_1$ as a special case here); 
(2) type CI with $i$ being the long simple root; (3) type DIII of even rank ($\tau = \id$) with $i$ being the rightmost white node; (4) type EVII with $i$ being the leftmost node.
\end{rem}

\begin{rem}
Lemma~\ref{lem:parameter vs} also follows from \cite[Remark~3.14]{BK14}, though they do not provide the precise values as in Table~\ref{table:values}.
Their approach has the advantage of being applicable in Kac-Moody case (under the assumption ``$\nu_i=1$ for all $i \in \Iwhite$"). 
\end{rem}

\subsection{Levi subalgebras}

Note an admissible subdiagram of a Satake diagram (see Definition~\ref{def:AdmSub}) is a Satake diagram itself. 
We sometimes denote the quantum symmetric pair $(\U, \Ui)$ associated to a Satake diagram $D$ with a root datum $\I$
by $(\U_{\I}, \U^\imath_{\I})$.
Let $A$ be an admissible subdiagram (with root datum $\mathbb J$) of the Satake diagram $D$. 
Then we have a quantum symmetric pair $(\U_{\mathbb J},  \U^\imath_{\mathbb J})$. 

\begin{lem}  
  \label{lem:Levi}
The coideal subalgebra $\U^\imath_{\mathbb J}$ of $\U_{\mathbb J}$ (whose parameter set is the restriction  from the parameter set of $\Ui$)
is naturally a subalgebra of $\Ui$.
\end{lem}

\begin{proof}
Let us use notation $w_\bullet^{\mathbb J}$ to indicate we are talking about the algebra $\U^\imath_{\mathbb J}$ 
associated to the root datum $\mathbb J$.
Let $i \in \I_{\circ} \cap \mathbb J$. Recall from \eqref{eq:def:ff} that 
$\ff_i = F_{i} + \vs_i \T_{w_{\bullet}} (E_{\tau i}) \widetilde{K}^{-1}_i  + \kappa_i \widetilde{K}^{-1}_{i}$ in $\Ui$, 
and 
$\ff_i = F_{i} + \vs_i \T_{w_{\bullet}^{\mathbb J}} (E_{\tau i}) \widetilde{K}^{-1}_i  + \kappa_i \widetilde{K}^{-1}_{i}$ in $\U^\imath_{\mathbb J}$. 
A simple key observation here is that $\T_{w_{\bullet}} (E_{\tau i}) = \T_{w_{\bullet}^{\mathbb J}}(E_{\tau i})$, which 
is a direct consequence of the definitions of subdiagrams of real rank one and admissible subdiagrams.
The lemma follows. 
\end{proof}

In light of Lemma~\ref{lem:Levi}, we make the following definition.
\begin{definition}  
  \label{def:Levi}
 A subalgebra of $\Ui$ of the form $\U^\imath_{\mathbb J}$ associated to some admissible subdiagram of $D$  is called a {\em Levi subalgebra}.
 (Some readers might prefer to call the subalgebra $\U^\imath_{\mathbb J} \Ui^0 \subseteq \Ui$ a Levi subalgebra of $\Ui$.)
Associated to a subdiagram of  
 $D$ of real (respectively, compact) rank one, $\U^\imath_{\mathbb J}$ is called a Levi subalgebra of $\Ui$ of {\em real (respectively, compact) rank one}.
 \end{definition}
 A Levi subalgebra of $\Ui$ of  compact rank one is very simple as it is always isomorphic to $\U_q(\mathfrak{sl}_2)$; it is a basic building block here as for quantum groups. 
Levi subalgebras of $\Ui$ of  real rank one, or $\imath$quantum groups of real rank one, are new (rich and 
sophisticated) basic building blocks for the theory of quantum symmetric pairs.

\subsection{The bar involution}

The following (or rather its variant on the modified $\imath$quantum group below) plays a fundamental role in the theory of $\imath$-canonical basis.

\begin{lem} \cite{BK14}
\label{lem:bar}
There is a unique anti-linear bar involution of the $\Q$-algebra $\Ui$, denoted by $\bar{\,}$ or $\psi_{\imath}$, such that 
\begin{align*}
\psi_{\imath} (q) =q^{-1}, \quad
\psi_{\imath} (\ff_i) = \ff_i \; (i \in \I), \quad \ipsi(E_i) = E_i  \;(i \in \I_{\bullet}), \quad \ipsi (K_\mu) = K_{-\mu} \; (\mu \in Y^{\imath}).
\end{align*}
\end{lem}

\begin{proof}
A complete proof for Lemma~\ref{lem:bar} was presented in \cite{BK14} over $\mathbb K (q^{\frac1{d}})$ for certain field $\mathbb{K}$ containing roots of 1, 
where they determined the precise constraints on the parameters. Now the existence of the bar involution of the $\Q(q)$-algebra $\Ui$ follows 
 from  our further restrictions on the parameters in Definition~\ref{def:Ui} and Lemma~\ref{lem:parameter vs}.
\end{proof}

\begin{rem}
The  bar involution for $\Ui$ in the special case of type AIII/AIV (with $\Iblack =\emptyset$) was proved in \cite{BW13}
and \cite{ES13} independently.  
The existence of the bar involutions on a general $\imath$quantum group $ \U^\imath$ 
was stated in \cite[\S0.5]{BW13} and verified by the authors for numerous examples, 
as it is a prerequisite for the theory of the $\imath$canonical bases announced therein. 

\end{rem}

\subsection{The modified $\imath$quantum group}
  \label{sec:modified-i}
  
Following the by now standard construction in quantum groups \cite[IV]{Lu94}, 
we can define a modified version of the $\imath$quantum groups (this was first considered in \cite{BKLW} in a special case of type AIII/AIV). 
Let $\lambda'$, $\lambda'' \in X_{\imath}$, we set 
\[
{{_{\lambda'}}\Ui_{\lambda''}} 
= \Ui \Big/ \Big(\sum_{\mu \in Y^{\imath}}(K_{\mu} - q^{\langle \mu, \lambda' \rangle}) \Ui 
 + \sum_{\mu \in Y^{\imath}} \Ui (K_{\mu} - q^{\langle \mu, \lambda'' \rangle}) \Big).
\] 
Let $\pi_{\lambda',\lambda''}: \Ui \rightarrow {{_{\lambda'}}\Ui_{\lambda''}}$ be the canonical projection. 
Write $\one_{\lambda'} = \pi_{\lambda',\lambda'}(1)$. Let 
\[
\Uidot = \bigoplus_{\lambda', \lambda'' \in X_{\imath}} {{_{\lambda'}}\Ui_{\lambda''}}.
\]
Then $\Uidot$ is naturally an associative algebra (without unit). 
The algebra $\Uidot$ admits a $(\Ui, \Ui)$-bimodule structure as well. Moreover, any weight 
(left/right) $\Ui$-module can naturally be regarded as a (left/right) $\Uidot$-module. 
In particular, the modified algebra $\Udot$ is a $(\Ui, \Ui)$-bimodule, where the bimodule structure is induced 
by the natural  embedding $\imath: \Ui \rightarrow \U$ and the quotient map $X \rightarrow X_\imath$. 
For any $\one_{\lambda} \in \Udot$ and $u \in \Uidot$ (or $\Ui$), we shall denote by $u \one_\lambda \in \Udot$ the action of $u$ on $\one_\lambda$. 
The first part of the following proposition follows from \eqref{eq;Eifj} , and the second part follows from Lemma~\ref{lem:bar}.

\begin{prop}
The following identities hold:
$$
\Uidot =\bigoplus_{\zeta \in X_\imath} \Ui^-\Ui^+ \one_\zeta 
= \bigoplus_{\zeta \in X_\imath} \Ui^-  \one_\zeta \Ui^+
= \bigoplus_{\zeta \in X_\imath} \Ui^+\Ui^- \one_\zeta. 
$$
There is a bar involution $\ipsi$ on the $\Q$-algebra $\Uidot$ such that $\psi_{\imath} (q) =q^{-1}$ and
\begin{align*}
\psi_{\imath} (\ff_i \one_\zeta) = \ff_i \one_\zeta \; (i \in \I), \quad \ipsi(E_i \one_\zeta) 
= E_i \one_\zeta \;(i \in \I_{\bullet}), \quad \ipsi (\one_\zeta) = \one_{\zeta} \; (\zeta \in X_{\imath}).
\end{align*}

\end{prop}  

\begin{rem}
It is possible to consider $\Uidot$ as a subalgebra of a certain completion of $\Udot$. 
But since we only consider weight $\U$-modules (i.e., unital modules in the sense of \cite[\S23.1.4]{Lu94}) as $\Ui$-modules, 
we prefer to regard $\Udot$ as a $(\Uidot, \Uidot)$-module.
\end{rem}

\begin{definition}
  \label{def:mAUidot}
 We define $_{\mA} \Uidot$ to be the set of elements $u \in \Uidot$, such that $ u \cdot m \in {}_\mA\Udot$ for all $m \in {}_\mA\Udot$. 
 Then $_{\mA} \Uidot$  is clearly a $\mA$-subalgebra of $\Uidot$ which contains all the idempotents $\one_\zeta$ $(\zeta \in X_\imath)$,
 and $_{\mA} \Uidot = \bigoplus_{\zeta \in X_\imath}\,  {}_{\mA} \Uidot \one_{\zeta}$.  
 \end{definition}
 Later  we shall show that ${}_\mA\Uidot$ is a free $\mA$-module such that $\Uidot \cong \Q(q) \otimes_{\mA} {}_\mA\Uidot$; 
 see Theorem~\ref{thm:iCBUi}(3). 

\begin{lem}\label{lem:Uiintegral}
Let $u \in \Uidot$. Then we have $u \in {}_{\mA}\Uidot$ if and only if $u \cdot \one_{\lambda} \in {}_\mA \Udot $ for each $\lambda \in X$.
\end{lem}

\begin{proof}
If remains to prove the ``if" direction.  
Take $m \in \one_\lambda ({}_\mA\Udot)$, for some $\lambda \in X$. By assumption, we have $u \cdot \one_{\lambda} \in {}_\mA \Udot $. 
Thus we have $u\cdot m =u \cdot (\one_\lambda m) =(u \cdot \one_\lambda) m  \in {}_{\mA}\Udot$, and so by definition, $u \in {}_{\mA} \Uidot$.
\end{proof}

\begin{cor}
 \label{cor:AA}
Let $u \in \Uidot$. Then we have $u \in {}_\mA \Uidot$ if and only if $ u \big({}_\mA L(\lambda) \big) \subset {}_\mA L(\lambda) $ for all $\lambda \in X^+$.
\end{cor}

\subsection{Relation with a parabolic subalgebra }

Let $w =s_{i_1} \cdots s_{i_l}$ be a reduced expression of an element  $w \in W$. 
Then the following elements (for various $c_{i_j} \in  \N$) 
\begin{equation}
  \label{PBW}
E^{(c_{i_1})}_{i_1} \cdot \T_{i_1}(E^{(c_{i_2})}_{i_2})\cdots \T_{i_1} \cdots \T_{i_{l-1}} (E^{(c_{i_l})}_{i_l})
\end{equation}
form a $\Qq$-basis of a subspace $\U^+(w)$ of $\U^+$ and an $\mA$-basis of an $\mA$-submodule ${}_\mA \U^+(w)$ of ${}_\mA \U^+$. 
The sets $\U^+(w)$ and ${}_\mA \U^+(w)$ depend only on $w$ but not on the choices of 
reduced expressions of $w$; our subspace $\U^+(w)$ here is denoted by $\U^+(w,1)$ in \cite[40.2]{Lu94}.
In particular, we have $\U^+(w_0) = \U^+$ and ${}_\mA \U^+(w_0) = {}_\mA \U^+$.  

Let $w^\bullet = w_0 \wb$. 
We can identify $\Padot$ with the quotient as $\Qq$-spaces
\[
\Padot \cong  \Udot \big/ \Udot \U^+(w^\bullet)_{>}
\]
where
$\U^+(w^\bullet)_{>} = \oplus_{\mu \in \N [\I] \backslash \{0\}}  \U^+(w^\bullet)_{\mu}$; 
Note $\Udot \big/ \Udot \U^+(w^\bullet)_{>} = \Udot \Big/ \big(\sum_{x, \lambda \in X} \Udot x \one_\lambda \big)$, where the sum is taken over all homogeneous $x \in \U^+$ whose weights are of the form 
$|x| = \sum_{\i \in \I} a_i i$ with $a_i \neq 0$ for some $i \in \Iwhite$. Thus we can define a left $\Udot$ action, 
which induces a left $\Uidot$ action on $\Padot$. 
For $\lambda \in X$, denote by $p_\imath= p_{\imath,\lambda}$ the composition map
\begin{equation}
  \label{eq:comp3}
\Uidot \one_{\overline{\lambda}} \longrightarrow \Udot \one_{\lambda} 
 \longrightarrow  \Udot \one_\la\big/ \Udot   \U^+(w^\bullet)_{>} \one_\la \longrightarrow \Padot \one_{\lambda}.
\end{equation}

\begin{lem}
  \label{lem:pimath}
Let $\lambda \in X$. The map $p_\imath= p_{\imath,\lambda}: \Uidot 1_{\overline{\lambda}} \rightarrow \Padot 1_{\lambda}$ is an isomorphism of left $\Uidot$-modules. Moreover $p_{\imath}$ maps ${}_\mA \Uidot \one_{\overline{\lambda}}$ injectively to ${}_\mA\Padot \one_{\lambda}$.
\end{lem}

Later in Corollary~\ref{cor:pimath}
we shall see that $p_{\imath}: {}_\mA \Uidot \one_{\overline{\lambda}} \longrightarrow {}_\mA\Padot \one_{\lambda}$ is an isomorphism. 

\begin{proof}
It is clear by definition that $p_\imath$ is a homomorphism of $\Uidot$-modules
and that $p_{\imath}$ maps ${}_\mA \Uidot \one_{\overline{\lambda}}$ to ${}_\mA\Padot \one_{\lambda}$ through the composition
${}_\mA \Uidot \one_{\overline{\lambda}} \rightarrow {}_\mA \Udot \one_{\lambda} \rightarrow {}_\mA\Padot \one_{\lambda}.$

It remains to show that $p_\imath : \Uidot 1_{\overline{\lambda}} \rightarrow \Padot 1_{\lambda}$ is both surjective and injective.
Let us first prove the surjectivity. We know that $\Padot \one_\lambda$ is spanned by elements of the form 
$F^{a_1}_{i_1} F^{a_2}_{i_2} \cdots F^{a_s}_{i_s} b^+ \one_{\lambda}$ with $b \in \B_{\Iblack}$. 
We shall proceed by induction on the sum $a = \sum a_i$ to prove that all such elements are in the image of $p_\imath$. 
The base case $a=0$ follows from the fact that $\Ui^+ = \U_{\Iblack}^+$. 
To show any $F^{a_1}_{i_1} F^{a_2}_{i_2} \cdots F^{a_s}_{i_s} b^+ \one_{\lambda}$ is in the image of $p_\imath$, we consider  
\begin{equation}\label{eq:pimath}
p_\imath ( \ff^{a_1}_{i_1} \ff^{a_2}_{i_2} \cdots \ff^{a_s}_{i_s} b^+ \one_{\overline{\lambda}} ) = F^{a_1}_{i_1} F^{a_2}_{i_2} \cdots F^{a_s}_{i_s} b^+ \one_{\lambda} + \text{ lower terms}.
\end{equation}
Now by definition \eqref{eq:def:ff}, we see that the ``lower terms" are linear combinations of  elements of the form 
$F^{a'_1}_{i'_1} F^{a'_2}_{i'_2} \cdots F^{a'_{s'}}_{i'_{s'}} b'^+ \one_{\lambda},$ 
where $\sum a'_1 < a$ and $b' \in \B_{\Iblack}$. The surjectivity follows by induction. 

Now we prove the injectivity of $p_\imath$. Recall the following multiplication maps
\[\Ui^- \otimes \Ui^0 \otimes \Ui^+ \longrightarrow \Ui \quad \text{ and } \quad  \U^- \otimes \U_{\Iblack}^0 \otimes \U^+_{\Iblack} \longrightarrow \Pa.
\]
We can find a subset of $\mc{J} \subset \cup^\infty_{n=0} \I^n$ such that the set $M= \{ F_{i_1}  F_{i_2} \cdots   \vert (i_1, i_2, \dots) \in \mc{J}\}$ form $\Qq$-bases of $\U^-$. Let $M^\imath = \{ \ff_{i_1}  \ff_{i_2} \cdots   \vert (i_1, i_2, \dots) \in \mc{J}\}$ in $\Ui^-$.
Thanks to \cite[Propositions~6.1, 6.2]{Ko14}, the set $\{y b^+ \one_{\overline{\lambda}} \vert y \in M^\imath, b \in \B_{\Iblack}\}$ 
forms a $\Qq$-basis of $\Ui\one_{\overline{\lambda}}$. Moreover by examining the leading terms as in \eqref{eq:pimath}, 
we see that the set $\{p_\imath  (yb^+ \one_{\overline{\lambda}}) \vert y \in M^\imath, b \in \B_{\Iblack}\}$ forms a $\Qq$-basis of  $\Padot 1_{\lambda}$,
whence the injectivity of $p_\imath$. 
\end{proof}

\begin{rem}
There are further intimate connections between the parabolic subalgebra $\Pa$ and the algebra $\Ui$.
Lemma~\ref{lem:Uietabotimeseta} below is another such example. 
Moreover, the $\imath$-canonical basis for $\Uidot \one_{\overline{\la}}$ is parametrized by the canonical basis for $\Padot \one_{\la}$ (see Theorem~\ref{thm:iCBUi}).
\end{rem}

\section{Symmetries of quantum symmetric pairs}
  \label{sec:braid}

In this section, we show that Lusztig's braid group operators $\T'_{i,e}$ and $\T''_{i,e}$, for $i \in \Iblack$, restrict to automorphisms of $\Ui$,
and the anti-involution $\wp$ on $\U$ restricts to an anti-involution of $\Ui$. 
We then prove that the intertwiner $\Upsilon$ is fixed by the actions of
$\T'_{i,e}$ and $\T''_{i,e}$, for $i \in \Iblack$, and this further implies that $\Upsilon_\mu \in \U^+(w_{\bullet}w_0)$.
We formulate an $\Ui$-module isomorphism $\mc T$ of any finite-dimensional $\U$-module over $\Qq$, following \cite{BK15} and generalizing \cite{BW13}.

\subsection{Braid group actions on $\Ui$}
Recall we have $-\wb \circ \tau= \id$ as permutations of the set $\Iblack$.
Recall the braid operators $\T_i$ and $\T_w$ (for $i \in \I$, $w\in W$) on the algebra $\U$. 

\begin{lem}
\label{lem:braid}
We have $$\T_i \T_{\wb} =\T_{\wb} \T_{\tau i},  \quad \forall i \in \Iblack.$$
In particular $\T_{\wb}^2$ commutes with $\T_i$ for any $i\in\Iblack$. 
\end{lem}
 
\begin{proof}
Since $s_i\wb =\wb s_{\tau i}  \in W_\bullet$ has length $\ell(\wb)-1$,
we have $\T_i \T_{s_i\wb} =\T_{\wb} =\T_{s_i\wb} \T_{\tau i}$, for all $i \in \Iblack$.
Hence $\T_i \T_{\wb}  =  \T_i \T_{s_i\wb} \T_{\tau i}=\T_{\wb} \T_{\tau i}$. 
Also we have 
$$
\T_i \T_{\wb}^2 =\T_{\wb} \T_{\tau i} \T_{\wb} = \T_{\wb}^2 \T_{\tau^2 i} =\T_{\wb}^2 \T_{i}.
$$  
The lemma is proved. 
\end{proof}
Let us record the following formulas for future use   
(cf. \cite[Lemma~3.4]{Ko14} \cite[Lemma~1.4]{BW13}): for any $i\in \Iblack$, recalling Definition~\ref{def:AP}(2), we have
\begin{align}
 \label{eq:Tw0}
 \begin{split}
\T^{-1}_{\wb} (E_i)  &= - \tK_{-\tau i} F_{\tau i}, \quad \T^{-1}_{\wb}(F_{ i}) = - E_{\tau i} \tK_{\tau i},
\\
\T_{\wb} (E_i) & = -   F_{\tau i} \tK_{\tau i}, \quad \T _{\wb}(F_{ i}) = - \tK_{-\tau i}E_{\tau i}.  
\end{split}
\end{align}
Theorem~\ref{thm:braidX} below confirms the conjecture in \cite{KP11} on the existence of a braid group action on $\Ui$ associated to $W_{\Iblack}$.

\begin{thm}
\label{thm:braidX}
For any $i\in \Iblack$ and $e =\pm 1$, 
the braid group operators $\T'_{i,e}$ and $\T''_{i,e}$ restrict to  isomorphisms of $\Ui$. More explicitly, we have, for $j\neq i$,
\begin{align*}
\T'_{i, e} (\ff_j) &=  \sum_{r+s = - \langle i, j'\rangle} (-1)^r q^{-er}_{i} \ff^{(s)}_i \ff_j \ff^{(r)}_i,  \\
\T''_{i, -e} (\ff_j) &=  \sum_{r+s = - \langle i, j'\rangle} (-1)^r q^{-er}_{i} \ff^{(r)}_i \ff_j \ff^{(s)}_i.
\end{align*}
\end{thm}

\begin{proof}
Let  $i \in \Iblack$. We shall only prove the theorem for $\T_i =\T''_{i,+1}$, as the other cases are proved by similar computations.

Since  $s_i$ preserves $Y^\imath$ in \eqref{XY},  $\T_i$ preserves the subalgebras $\U_{\bullet}$ and $\Ui^0$ of $\Ui$. 
Then it remains to check the action of $\T_i$ on the generators $\ff_j$ for $j \in \Iwhite$. 
Recall $ \ff_j  = F_{j } + \vs_j \T_{\wb} (E_{\tau j})\tK^{-1}_{j} + \kappa_j \tK^{-1}_j$. If $\langle i, j' \rangle = \langle \tau i, \tau j' \rangle = 0$, 
we have $\T_{i} (\ff_j)  = \ff_j$ by Lemma~\ref{lem:braid}.  In particular, if $\kappa_j \neq 0$ 
(and hence $\langle i, j' \rangle = 0$ $\forall i \in \Iblack$), then $\T_{i}  (\ff_j)  = \ff_j$, $\forall i \in \Iblack$.

It remains to consider the case where $i \in \Iblack$, $j \in \Iwhite$ such that $\langle i, j \rangle \neq 0$. Recall by \eqref{eq:braidgroup} that
\[
\T_{i } (F_j) =\sum_{r+s = -\langle i , j' \rangle} (-1)^r q_i^{r} F_i^{(r)} F_j F^{(s)}_{i},
\quad  \T_{\tau i } (E_{\tau j}) =\sum_{r+s = -\langle i , j' \rangle} (-1)^r q_i^{-r} E_{\tau i}^{(s)} E_{\tau j} E^{(r)}_{\tau i} .
\]
We shall also use the identity
\begin{align*}
(F _{ i} \tK_i)^{s} =q_i^{-s(s-1)} F _{ i} ^s\tK_i^{s}. 
\end{align*}
By Lemma~\ref{lem:braid} and \eqref{eq:Tw0}, we have
\begin{align*}
\T_{ i }(\ff_j) 
&=  \T_{ i } (F_j) + \vs_j  \T_{\wb} \T_{\tau i } (E_{\tau j})  \T_{i} ( \tK^{-1} _{ j } )\\
	=& \sum_{r+s = -\langle i , j' \rangle} (-1)^r q_i^{r} F_i^{(r)} F_j F^{(s)}_{i} + \vs_j  \T_{\wb}  \sum_{r+s = -\langle i , j' \rangle} (-1)^r q_i^{-r} E_{\tau i}^{(s)} E_{\tau j} E^{(r)}_{\tau i}   \cdot \T_{i} ( \tK^{-1} _{ j } )
	\displaybreak[0]\\
	=&  \sum_{r+s = -\langle i , j' \rangle} (-1)^r q_i^{r} F_i^{(r)} F_j F^{(s)}_{i}  + \\
	& \qquad \vs_j     \sum_{r+s = -\langle i , j' \rangle} (-1)^r q_i^{-r} (-F _{ i} \tK_i)^{(s)} \T_{\wb}( E_{\tau j})  (-F _{ i} \tK_i)^{(r)}   \cdot  \T_{i} ( \tK^{-1} _{ j } ) \displaybreak[0]\\
	=&  \sum_{r+s = -\langle i , j' \rangle} (-1)^r q_i^{r} F_i^{(r)} F_j F^{(s)}_{i}  + 
	\\ & \qquad\vs_j     \sum_{r+s = -\langle i , j' \rangle} (-1)^s q_i^{-r}  (F _{ i} \tK_i)^{(s)} \T_{\wb}( E_{\tau j})  (F _{ i} \tK_i)^{(r)}   \cdot  \T_{i} ( \tK^{-1} _{ j } ) \displaybreak[0]\\
	=& \sum_{r+s = -\langle i , j' \rangle} (-1)^r q_i^{r} F_i^{(r)} F_j F^{(s)}_{i}  + \\
	&   
\qquad \vs_j     \sum_{r+s = -\langle i , j' \rangle} (-1)^s q_i^{-r -(r+s)(r+s-1)-s \langle i, j' \rangle} 
 F _{ i}^{(s)} \T_{\wb}( E_{\tau j}) F _{ i}^{(r)}   \cdot   \tK_i^{-\langle i , j' \rangle} \T_{i} ( \tK^{-1} _{ j } )\displaybreak[0]\\
=& \sum_{r+s = -\langle i , j' \rangle} (-1)^r q_i^{r} F_i^{(r)} F_j F^{(s)}_{i}   +  
\\ &
 \qquad \vs_j     \sum_{r+s = -\langle i , j' \rangle} (-1)^s q_i^{-r + (r-1)\langle i, j' \rangle} F _{ i}^{(s)} \T_{\wb}( E_{\tau j}) F _{ i}^{(r)}   \cdot   \tK^{-1} _{ j }\displaybreak[0]\\
	=&  \sum_{r+s = -\langle i , j' \rangle} (-1)^r q_i^{r} F_i^{(r)} F_j F^{(s)}_{i}   +  
 \vs_j     \sum_{r+s = -\langle i , j' \rangle} (-1)^s q_i^{s} F _{ i}^{(s)} \T_{\wb}( E_{\tau j}) \tK^{-1} _{ j }  F _{ i}^{(r)}    \\
	=& \sum_{r+s = -\langle i , j' \rangle} (-1)^r q_i^{r} F_i^{(r)} \ff_j F^{(s)}_{i} \\
	=& \sum_{r+s = -\langle i , j' \rangle} (-1)^r q_i^{r} \ff_i^{(r)} \ff_j \ff^{(s)}_{i} \in  \Ui.
\end{align*}
The theorem follows.
\end{proof}

\begin{cor}\label{cor:Tpsi}
For any $u \in \Ui$, $e=\pm1$, and $i \in \Iblack$, we have 
$$\psi_{\imath}(\T''_{i, e}(u)) = \T''_{i, -e} (\psi_{\imath} (u)).
$$
\end{cor}

\begin{proof}
As $\psi_{\imath}$ and $\T''_{i, e}$ are algebra isomorphisms, it suffices to check the identity for $u$ being generators of $\Ui$. 
For generators in $\U_{\Iblack}$ or in $\Ui^0$, this follows from \cite[\S37.2.4]{Lu94}. For $u = \ff_{i}$  with $i \in \Iwhite$, the identity follows from the formulas in Theorem~\ref{thm:braidX} and that $\psi_{\imath} (\ff_i) = \ff_i$ from Lemma~\ref{lem:bar}.
\end{proof}

\subsection{Anti-involution $\wp$ on $\Ui$}

We study the restriction to $\Ui$ of the anti-involution $\wp: \U \rightarrow \U$ in Proposition~\ref{prop:invol}. 

\begin{lem}\label{lem:rho}
For all  $i, j \in \I$, the following identities hold on $\U$: 
\begin{align*}
\wp  (\T''_{i,e} (E_{j})) &= (-q_i)^{e \langle i, j' \rangle} \T'_{i,-e} (\wp (E_j)),
\\
\wp  (\T'_{i,e} (E_{j})) &= (-q_i)^{-e \langle i, j' \rangle} \T''_{i,-e} (\wp (E_j)).
\end{align*}
\end{lem}

\begin{proof}
We shall prove the first identity only, as the second one is similar. 
When $\langle i , j' \rangle =0$, the first identity is trivial. 

For $i =j$, we have 
\begin{align*}
\wp (\T''_{i,e} (E_{i})) =  \wp (-F_i \tK_{ei}) = -q^{-1}_i \tK_{ei} E_i \tK_{-i} = -q^{-1 + e \langle i, i' \rangle}_i  E_i \tK_{ei} \tK_{-i} .
\end{align*}
On the other hand, we have 
\[
\T'_{i,-e}(  \wp (E_i)) =\T'_{i,-e} (q^{-1}_i F_i \tK_i) = -q^{-1}_i E_i \tK_{ei}  \tK_{-i}.
\]
Hence the first identity for $i = j$ holds.

For $i \neq j$ with $\langle i , j' \rangle \neq 0$, we have 
\begin{align*}
\displaybreak
\wp (\T''_{i,e} (E_{j})) =&  \wp \Big( \sum_{r+s = - \langle i , j' \rangle} (-1) ^r q^{-er}_i  E^{(s)}_i\ E_j E^{(r)}_i \Big)\\
 =& \sum_{r+s = - \langle i , j' \rangle} (-1)^r q^{-er}_i  (q^{-1}_i F_i \tK_i )^{(r)} (q^{-1}_j F_j \tK_j) (q^{-1}_i F_i \tK_i )^{(s)}\\
  =& \sum_{r+s = - \langle i , j' \rangle} (-1)^r q^{-er}_i  q^{-1}_j F_i^{(r)} F_j F_i ^{(s)}    \tK^{-\langle i , j'\rangle}_i \tK_{j}     \\
 =&  \sum_{r+s = - \langle i , j' \rangle} (-1)^s q^{-1}_j  (-q_i)^{e \langle i, j' \rangle}q^{ e s}_i  F_i^{(r)} F_j  F_i^{(s)} \tK^{-\langle i , j'\rangle}_i \tK_{j}\\
 = & q^{-1}_j  (-q_i)^{e \langle i, j' \rangle} \sum_{r+s = - \langle i , j' \rangle} (-1)^s q^{e s}_i  F_i^{(r)} F_j  F_i^{(s)} \T_{i} ( \tK _{ j } ).
\end{align*}
On the other hand, we have 
\begin{align*}
\T'_{i,-e} (  \wp (E_j)) &= \T'_{i,-e}( q^{-1}_j F_j \tK_j) \\
=& q^{-1}_j   \sum_{r+s = - \langle i , j' \rangle} (-1)^r q^{er}_i  F_i^{(s)} F_j  F_i^{(r)}  \T_{i} ( \tK _{ j } ).
\end{align*}
This completes the proof of the first identity and the lemma.
\end{proof}

\begin{cor}
  \label{cor:rhoTwb}
  For $i \in \I_{\circ}$ and $e = \pm 1$, we have 
\begin{align*}
\wp  (\T''_{\wb, e} (E_{i}))  &= (-1)^{\langle 2\rho^\vee_{\bullet}, i' \rangle} q_i^{e \langle i, 2 \rho_{\bullet} \rangle} \T'_{\wb, -e} (\wp(E_{i})),
\\
\wp  (\T'_{\wb, e} (E_{i}))  &= (-1)^{\langle 2\rho^\vee_{\bullet}, i' \rangle} q_i^{-e \langle i, 2 \rho_{\bullet} \rangle} \T''_{\wb, -e} (\wp(E_{i})).
\end{align*}
\end{cor}

\begin{proof}
We shall only prove the first identity, as the proof of the second one is similar. 
Let $\wb = s_{i_1}  s_{i_2} \cdots s_{i_l}$ be a reduced expression of $\wb$. 
Write $w_k = s_{i_k}  s_{i_{k+1}} \cdots s_{i_l}$ for $1 \le k \le l$. Note that $\T_{w_k} (E_{i})\in \U^{+}$. 
Thus applying Lemma~\ref{lem:rho} repeatedly, we have 
\begin{align*}
\wp  (\T''_{\wb, e} (E_{i})) 
&= (-q_{i_1})^{e \langle i_1, w_2(i')\rangle}  (-q_{i_2})^{ e \langle i_2, w_3(i')\rangle} \cdots  (-q_{i_l})^{e \langle i_l, i' \rangle}\T'_{\wb, -e}  (\wp (E_j))\\
&= (-q_{i_1})^{ e \langle w_2^{-1}( i_1), i ' \rangle}  (-q_{i_2})^{e \langle w^{-1}_3 (i_2), i' \rangle} \cdots  (-q_{i_l})^{e \langle i_l, i' \rangle} \T'_{\wb, -e}  (\wp (E_j))\\
&\stackrel{\spadesuit}{=} (-1)^{\langle 2\rho^\vee_{\bullet}, i \rangle} q_i^{e \langle i, 2 \rho_{\bullet} \rangle}\T'_{\wb, -e}  (\wp (E_j)), 
\end{align*}
where the identity $\spadesuit$ follows from the fact that $\{w^{-1}_{i+1} (\alpha_i) \vert 1 \le i \le l\}$ consists of all positive roots in $Y_{\bullet}$ and the equality 
\begin{equation}
  \label{eq:sumrho}
\sum^l_{k=1} \frac{i_k \cdot i_k}{2} \big \langle w^{-1}_{k+1} (i_k), i' \big\rangle 
= \frac{i \cdot i}{2} \langle i, 2 \rho_\bullet \rangle. 
\end{equation}
The equation \eqref{eq:sumrho} can be verified as follows: 
\begin{align*}
 \sum^l_{k=1} \frac{i_k \cdot i_k}{2} \langle w^{-1}_{k+1} (i_k), i' \rangle  
= & \sum^l_{k=1} \frac{i_k \cdot i_k}{2} \langle  i_k, w_{k+1} (i') \rangle \\
= &\sum^l_{k=1}   i_k \cdot w_{k+1} (i)  \stackrel{\heartsuit}{=} \sum^l_{k=1}  w^{-1}_{k+1}( i_k) \cdot  i  \\
= &\sum^l_{k=1}  \frac{i \cdot i}{2}  \langle i , w^{-1}_{k+1}( i_k) \rangle  =  \frac{i \cdot i}{2} \langle i, 2 \rho_\bullet \rangle,
\end{align*}
where the identity $\heartsuit$ follows from the $W$-invariance of the bilinear pairing $ \cdot : \Z[\I] \times \Z[\I] \rightarrow \Z$.
\end{proof}


Recall from Lemma~\ref{lem:parameter vs} that $\vs_1 = \pm q_1^{-1}$ in type AI$_1$, and by Remark~\ref{lem:nonzero}, this is the only local rank one case which allows the parameter $\kappa_1 \neq 0$. 

\begin{prop}
\label{prop:rho}
Assume that $\vs_i = q_i^{-1} \; \text{if } \kappa_i \neq 0$.
The anti-involution $\wp$ on $\U$ restricts to an anti-involution $\wp$ on $\U^\imath$ such that 
\begin{align}
\wp(E_{i}) &= q^{-1}_i F_{i} \Ktilde_{i}, \quad \wp(F_{i}) = q^{-1}_i E_{i} \Ktilde_i^{-1}, \quad \wp(K_{\mu}) = K_{\mu}, \quad \forall i \in \Iblack; 
\label{eq:wpblack}
 \\
\wp (\ff_i ) &=  q_i^{-1} \vs_{\tau i}^{-1} \T_{\wb}^{-1} (\ff_{\tau i} ) \cdot \tK_{\wb \tau i} \tK_i^{-1},\quad \forall i \in \Iwhite.
\label{eq:rhoB}
\end{align}
\end{prop}

\begin{proof}
Equation~\eqref{eq:wpblack} follows from  the formula for $\wp$ on $\U$ in Proposition~\ref{prop:invol}.

Let us prove \eqref{eq:rhoB}. Recall 
$
\ff_i =F_i  +\varsigma_i  \T_{w_\bullet} (E_{\tau i}) \tK_i^{-1} +\kappa_i \tK_i^{-1}$ for $i \in \Iwhite.$
Note $q_{\tau i} =q_i$. 
Then since $\wp$ is an anti-isomorphism on $\U$, applying Propsition~\ref{prop:invol}(2), Corollary~\ref{cor:rhoTwb} and then \eqref{vs2}, we have
\begin{align}
\label{rhoiB}
\begin{split}
\wp(\ff_i) &=q_i^{-1} E_i\tK_i^{-1} + \vs_i   \tK_i^{-1}\wp(\T_{\wb} (E_{\tau i}))  + \kappa_i \tK_i^{-1}
\\
&=q_i^{-1} \vs_{\tau i}^{-1} \Big(\vs_{\tau i}E_i + q_i  \vs_{\tau i} \vs_i \tK_i^{-1}  \wp  (\T_{\wb} (E_{\tau i})) \tK_i  + q_i \vs_{\tau i} \kappa_i \Big)    \tK_i^{-1}
\\
&=q_i^{-1} \vs_{\tau i}^{-1} 
\Big(\vs_{\tau i}E_i 
+ q_i  \vs_{\tau i} \vs_i \tK_i^{-1}  (-1)^{\langle 2\rho^\vee_{\bullet}, i' \rangle} q_i^{ \langle i, 2 \rho_{\bullet} \rangle} q_{\tau i}^{-1} \T_{\wb}^{-1} (F_{\tau i}) \tK_{\wb \tau i} \tK_i  
+ q_i \vs_{\tau i} \kappa_i \Big)  \tK_i^{-1}
\\
&=q_i^{-1} \vs_{\tau i}^{-1} 
\Big(\vs_{\tau i}E_i 
+ \T_{\wb}^{-1} (F_{\tau i}) \tK_{\wb \tau i}  
+ q_i \vs_{\tau i} \kappa_i \Big) \tK_i^{-1}.
\end{split}
\end{align}
On the other hand, we have 
\begin{align}
\label{TB}
\begin{split}
\T_{\wb}^{-1} (\ff_{\tau i} ) &= \T_{\wb}^{-1}  \Big( F_{\tau i} + \vs_{\tau i} \T_{\wb} (E_i) \tK_{\tau i}^{-1} + \kappa_{\tau i} \tK_{\tau i}^{-1} \Big)
\\
&= \Big( \vs_{\tau i}E_i + \T_{\wb}^{-1} (F_{\tau i})  \tK_{\wb \tau i} +  \kappa_{\tau i} \Big) \tK_{\wb \tau i}^{-1}. 
\end{split}
\end{align}
Recall the assumption  that $\vs_i = q_i^{-1} \; \text{if } \kappa_i \neq 0$, and note that $\tau i=i$ if $\kappa_i \neq 0$. 
The formula \eqref{eq:rhoB}  follows now by a comparison of \eqref{rhoiB}--\eqref{TB}. 

Note $\tK_{\wb \tau i} \tK_{i}^{-1} =\theta (\tK_i)^{-1} \tK_{i}^{-1} \in \Ui$. It now follows by Theorem~\ref{thm:braidX} that  
\[
\wp(\ff_i) =q_i^{-1} \vs_{\tau i}^{-1} \T_{\wb}^{-1} (\ff_{\tau i} ) \cdot  \tK_{\wb \tau i} \tK_i^{-1} \in \Ui.
\]
Hence $\wp$ maps every generator of $\Ui$ to elements in $\Ui$, and so  it is an anti-involution on $\Ui$. 
\end{proof}

\begin{rem}
 \label{rem:characterization}
A more careful analysis of the proof of Proposition~\ref{prop:rho} shows that the anti-involution $\wp: \U \rightarrow \U$ restricts to an anti-involution on $\U^\imath$ if and only if 
the following conditions \eqref{rhobil}-\eqref{rhobil2} hold:
\begin{align}
\vs_i &= q_i^{-1} \; (\text{if } \kappa_i \neq 0), 
  \label{rhobil} \\
\wp  (\T_{\wb} (E_{i}))  &= q_i^{-\langle i, \wb \tau i' \rangle}   \vs_{\tau i}^{-1} \vs_i^{-1} \T_{\wb}^{-1} (\wp(E_{i})),
\quad \forall i \in \Iwhite.
\label{rhobil2}
\end{align}
Together with Corollary~\ref{cor:rhoTwb}, we see that the anti-involution $\wp$ on $\Ui$ requires Condition~\eqref{vs2}  
for the parameters $\vs_i$ $(i \in \Iwhite)$ and the additional stronger assumption in Proposition~\ref{prop:rho}.  
\end{rem}

{\it For the remainder of the paper, we shall assume that on top of Conditions \eqref{kappa}-\eqref{vs2} on the parameters  the additional Condition~\ref{rhobil} holds. }


\subsection{The intertwiner}
Let $\widehat{\U}$ be the completion of the $\Qq$-vector space $\U$ 
with respect to the descending sequence of subspaces 
$\U^- \U^0 \big(\sum_{\hgt(\mu) \geq N}\U_{\mu}^+ \big)$,   for $N \ge 1.$
Then we have the obvious embedding of $\U$ into $\widehat{\U}$. 
We let $\widehat{\U}^+$ be the closure of $\U^+$ in $\widehat{\U}$, and so $\widehat{\U}^+ \subseteq \widehat{\U}$. By continuity the $\Q(q)$-algebra structure 
on $\U$ extends to a $\Q(q)$-algebra structure on $ \widehat{\U}$.  The bar involution $\psi$  on $\U$ extends 
by continuity to an anti-linear involution $\psi$ on $\widehat{\U}$.

Note the inclusion map $\Ui \to \U$ is not compatible with the two bar maps on $\Ui$ and $\U$.  Recall the $\imath$-weight lattice $X_\imath$ from \eqref{XY}. The following theorem, which has appeared in the literature recently, is one key ingredient for the current work. 

\begin{thm}  (cf. \cite[Theorem~6.10]{BK15})
   \label{thm:Upsilon}
There exists a unique family of elements $\Upsilon_{\mu} \in \U^+_{\mu}$, 
such that $\Upsilon_0 =1$ and $\Upsilon = \sum_{\mu} \Upsilon_{\mu}$ satisfies the following identity (in $\widehat{\U}$):
\begin{equation}\label{eq:Upsilon}
 \ipsi(u) \Upsilon = \Upsilon \psi(u), \qquad \text{for all } u \in \Ui.
\end{equation}
Moreover, $\Upsilon_{\mu} =0$ unless ${\mu^\inv} = - \mu \in X$. 
\end{thm}

\begin{rem}
  \label{rem:history}
This theorem  was first formulated and established in the special case of type AIII/AIV (with $\Iblack =\emptyset$) in  \cite[Theorem~2.10]{BW13}, generalizing Lusztig's  quasi-$R$-matrix; note that $\Upsilon$ therein lies in a completion of $\U^-$ (not $\U^+$) due to a different convention on the comultiplication $\Delta$. The theorem in general was expected by the authors as one of the main building blocks in a program of $\imath$-canonical bases arising from general quantum symmetric pairs announced in  \cite[\S0.5]{BW13}, since it leads to a new bar involution $\ipsi$ on based $\U$-modules (see Proposition~\ref{prop:compatibleBbar}); we verified the theorem in the cases when $\Iblack = \emptyset$. In the meantime, this theorem has appeared with a complete proof in full generality in Balagovi\'c and Kolb \cite{BK15} (where $\Upsilon$ was denoted by $\mathfrak X$ and called a quasi-$\mc K$-matrix). 
\end{rem}

\begin{rem}\label{rem:QSPLu}
It is instructive to view the following familiar cases as two extreme cases of quantum symmetric pairs.

\begin{enumerate}
\item 
When $\I =\I_{\bullet}$ (and recall $\tau =- w_\bullet$ on $\Iblack$), we have $\Ui =\U_{\I_{\bullet}} = \U$, the usual Drinfeld-Jimbo quantum group.
In this case, the intertwiner $\Upsilon$ is  the identity.

\item
Consider the  algebra imbedding $\phi =(\omega \otimes 1) \circ \Delta: \U \rightarrow \U \otimes \U$. 
One checks that $\phi(\U)$ is a coideal subalgebra of $\U\otimes\U$, and 
hence we have a quantum symmetric pair of  {\em diagonal type} $(\U \otimes \U, \U)$. Then the intertwiner in this case is Lusztig's  
quasi-$R$-matrix twisted by $\omega \otimes 1$. Hence, Lustig's construction of the bar involution and the canonical basis on the tensor product modules 
fits well with our general construction. 
\end{enumerate}
\end{rem}

Since $ \ipsi(u) = \psi(u)$, for $ u \in \U_{\I_{\bullet}}$, the identity \eqref{eq:Upsilon} implies that  
\begin{equation}\label{eq:Upsilonblack}
 u \Upsilon = \Upsilon u, \qquad \text{ for } u \in \U_{\I_{\bullet}}.
\end{equation}

The following corollary is the same as for \cite[Corollary~2.13]{BW13}, which follows by the uniqueness of $\Upsilon$.

\begin{cor}\label{cor:barUpsilon}
We have $\psi(\Upsilon) = \Upsilon^{-1}$.
\end{cor}


\subsection{Properties of the intertwiner}

We now establish several new properties of the intertwiner $\Upsilon$ in connection with braid group action on $\Ui$.
They will be used later on to establish the integrality of $\Upsilon$.

\begin{lem}
  \label{lem:TUpsilon}
For $\mu \in \N [\I], i \in \Iblack$ and $e =\pm 1$, we have 
  \begin{enumerate}
  \item
  $r_{i} (\Upsilon_{\mu}) =  {_i}r (\Upsilon_{\mu}) =0$;
  
\item
$\T''_{i, e} (\Upsilon_{\mu}) \in \U^{+}$ and $\T'_{i, e} (\Upsilon_{\mu}) \in \U^{+}$. 
\end{enumerate}
\end{lem}
\begin{proof}
Note  by Equation \eqref{eq:Upsilonblack} and by \cite[3.1.6]{Lu94}
that, for $i \in \Iblack$, 
\begin{align*}
F_{i} \Upsilon_\mu - \Upsilon_\mu F_{i}& =0, 
\\
F_{i} \Upsilon_\mu - \Upsilon_\mu F_{i} &=\frac{\tK_{-i} \,  {}_ir(\Upsilon_\mu) - r_i (\Upsilon_\mu) \tK_i  }{q_i -q_i^{-1} }. 
\end{align*}
Hence it follows that $r_{i} (\Upsilon_{\mu}) =  {_i}r (\Upsilon_{\mu}) =0$. 
Therefore we have $\T''_{i, 1} (\Upsilon_{\mu}) \in \U^{+}$ and $\T'_{i,-1} (\Upsilon_{\mu}) \in \U^{+}$ by \cite[Proposition~38.1.6]{Lu94}. 
On the other hand, since $\Upsilon_\mu$ is homogeneous, we have $\T''_{i,e} (\Upsilon_\mu) = (-q_i)^{e \langle i, \mu \rangle } \T'_{i,e} ( \Upsilon_\mu )$. The lemma follows.
\end{proof}

The symmetries $\T'_{i,e}$ and $\T''_{i,e}$ extend by continuity to symmetries of $\widehat{\U}$.

\begin{prop}\label{prop:TUpsilon}
We have $\T''_{i, e} (\Upsilon) = \Upsilon$ and $\T'_{i, e} (\Upsilon) = \Upsilon$ for  all $i \in \Iblack, e =\pm 1$.
\end{prop}

\begin{proof}
For any $u \in \Ui$,  applying $\T''_{i, e}$ to Equation~\eqref{eq:Upsilon} gives us 
\begin{equation}
 \label{eq:TUp}
\T''_{i, e} (\psi_{\imath}(u) ) \T''_{i, e} (\Upsilon) = \T''_{i, e} (\Upsilon) \T''_{i, e} ( \psi(u) ). 
\end{equation}
By Corollary~\ref{cor:Tpsi}  
and by Proposition~\ref{sec1:prop:braid}(2), we rewrite \eqref{eq:TUp} as 
\begin{equation}
\psi_{\imath}(\T''_{i, -e} (u)) \T''_{i, e} (\Upsilon) =\T''_{i, e} (\Upsilon)  \psi(\T''_{i, -e} (u)).
\end{equation}
Since $\T''_{i, -e}$ restricts to an isomorphism  of $\Ui$ by Theorem~\ref{thm:braidX}, $\T''_{i, -e} (u)$ can be any element in $\Ui$. Therefore we have 
\[
\psi_{\imath} (x) \T''_{i, e} (\Upsilon) = \T''_{i, e} (\Upsilon) \psi(x), \quad \text{ for all } x \in \Ui. 
\]
It is clear that $\T''_{i, e} (\Upsilon_0) = \Upsilon_0 =1$. Thanks to Lemma~\ref{lem:TUpsilon} and the uniqueness of the intertwiner, we have $\Upsilon = \T''_{i, e} (\Upsilon)$. The other identity is entirely similar. 
\end{proof}

\begin{rem}
Proposition~\ref{prop:TUpsilon} also follows from \cite[\S3.1]{BK15} if one considers $T_i$ as an element in the algebra of natural transformations of the forgetful functor.
\end{rem}

Recall from \eqref{PBW} the subspace $\U^+(w)$ of $\U^+$, for $w \in W$.
Note $\ell(w_{\bullet} w_0) =\ell(w_0) - \ell(w_\bullet)$. 

\begin{prop}\label{prop:wcirc}
We have $\Upsilon_\mu \in \U^+(w_{\bullet} w_0)$ for any $\mu \in \N [\I]$. 
\end{prop}

\begin{proof}
If $\Iblack = \emptyset$, the statement follows by Theorem~\ref{thm:Upsilon}. 
So let us assume $\Iblack \neq \emptyset$. Now choose a reduced expression $s_{i_1} \cdots s_{i_l}$ of   $w_0$ 
such that $s_{i_1} \cdots s_{i_k} = \wb$ (in particular $i_1 \in \Iblack$) and $s_{k+1} \cdots s_{i_l} = w_{\bullet} w_0$.
We have a PBW basis \eqref{PBW} for  $\U^+$ associated to this reduced expression of $w_0$.  
We have $r_{i_1} (\Upsilon_{\mu}) =0$ by Lemma~\ref{lem:TUpsilon}(1), and thus by \cite[Proposition~38.1.6]{Lu94} we  can write 
\[
\Upsilon_{\mu} = \sum c(c_{i_2}, \dots, c_{i_l}) \, \T_{i_1}(E^{(c_{i_2})}_{i_2})\cdots \big( \T_{i_1} \ldots \T_{i_{l-1}} (E^{(c_{i_l})}_{i_l}) \big),
\]
for scalars $c_{i_a} \in \Qq$. If $\ell (w_\bullet) =1$, we are done. 

If $\ell (w_\bullet) >1$, then $i_2 \in \Iblack$.
By Proposition~\ref{prop:TUpsilon}, we have 
\[
\Upsilon_{\mu}  = \T^{-1}_{i_1}(\Upsilon_{\mu})
= \sum c(c_{i_2}, \dots, c_{i_l}) \,E^{(c_{i_2})}_{i_2} \, \T_{i_2}(E^{(c_{i_3})}_{i_3}) \cdots \big( \T_{i_2} \ldots \T_{i_{l-1}} (E^{(c_{i_l})}_{i_l}) \big),
\]
which, by Lemma~\ref{lem:TUpsilon}(1) and \cite[Proposition~38.1.6]{Lu94}, is of the form
\[
\Upsilon_{\mu}   
= \sum c(0, c_{i_3}, \dots, c_{i_l})   \T_{i_2}(E^{(c_{i_3})}_{i_3}) \cdots  \big( \T_{i_2} \ldots \T_{i_{l-1}} (E^{(c_{i_l})}_{i_l}) \big).
\]
Repeating  the process $\ell(w_{\bullet}) =k$ times, we obtain
\[
\Upsilon_{\mu}   
= \sum c(0, \cdots, 0, c_{i_{k+1}}, \dots, c_{i_l})   \T_{i_k}(E^{(c_{i_{k+1}})}_{i_{k+1}}) \cdots  \big( \T_{i_k} \ldots \T_{i_{l-1}} (E^{(c_{i_l})}_{i_l}) \big). 
\] 
This shows that $\Upsilon_{\mu} \in \U^+(w_{\bullet} w_0)$.
\end{proof}

The strong constraint on $\Upsilon$ proved in 
Proposition~\ref{prop:wcirc} shall allow us to compute the intertwiner $\Upsilon$ (almost) explicitly and to establish the integrality of $\Upsilon$ in all real rank one cases
(as listed in Table~\ref{table:localSatake}). See the Appendix for the detailed computation.

\subsection{The isomorphism $\mc{T}$}
 \label{subsec:T}

Consider the automorphism obtained by the composition 
$$
\vartheta =  \sigma \circ \wp \circ \tau : \U \longrightarrow \U, 
$$
which sends
\begin{equation}\label{eq:Ttauoemga}
\vartheta (E_i) = q_{\tau i}F_{\tau i} \tK_{- \tau i} , \quad  \vartheta (F_i) = q_{\tau i}  E_{\tau i} \tK_{\tau i},  \quad \vartheta (K_{\mu}) = K_{-\tau \mu}.
\end{equation}

For any finite-dimensional $\U$-module $M$, we define a new $\U$-module ${}^\vartheta M $ as follows: 
${}^\vartheta M$ has the same underlying $\Qq$-vector space as $M$ but we shall denote a vector in ${}^\vartheta M$ by ${}^\vartheta m$ for $m\in M$, and
the action of $u\in \U$ on ${}^\vartheta M$ is now given by $u \; {}^\vartheta m = {}^\vartheta (\vartheta^{-1} (u)m)$. 

Hence we have 
\begin{equation}
  \label{twisted}
\vartheta  (u) \; {}^\vartheta m = {}^\vartheta ( u m), \qquad \text{ for } u \in \U, m \in M.
\end{equation}
As ${}^\vartheta M$ is simple if the $\U$-module $M$ is simple, one checks by definition that
\[
{}^\vartheta L(\lambda) \cong {}^{\omega}L(\lambda^{\tau}).
\]

Let 
\begin{equation}
 \label{eq:g}
g : X \longrightarrow  \Qq
\end{equation}
be a function such that for all $\mu \in X$, we have 
\begin{align}
  \label{eq:def:zeta1}{g} (\mu) & = {g} (\mu-i' ) \vs_i  (-1)^{\langle 2\rho^\vee_{\bullet}, i' \rangle} q_i^{\langle i, 2 \rho_{\bullet} \rangle} q_i q^{\langle -i, \mu \rangle }_{i} q_{ \tau i}^{\langle \tau i, \wb\mu \rangle}, &\quad \text{for } i \in \Iwhite, \\
 \label{eq:def:zeta2} {g}(\mu) &=   -q^{-1- 2\langle i , \mu - i' \rangle}_{ i}  {g} (\mu -i') = -q^{3- 2\langle i , \mu \rangle}_{ i}  {g} (\mu -i'), &\quad \text{for } i \in \Iblack.
\end{align}
Such a function $g$ exists. (Actually we can construct such a $g$ taking values in $\mA$.)

\begin{lem}\label{lem:zeta}
For any $\mu \in X$, we have
\begin{align*}
 {g}(\mu) &={g}(\mu - \wb i')  \vs^{-1}_{\tau i} (-1)^{ \langle 2\rho^\vee_{\bullet},  {\tau i'} \rangle } q_{\tau i}^{-\langle {\tau i}, 2\rho_{\bullet} \rangle} q_{\tau i}^{\langle {\tau i}, \mu- \wb  i' \rangle} q^{}_{i} q_{i}^{ -\langle i, \wb\mu  \rangle}, &\quad \text{for } i \in \Iwhite, \\
 {g}(\mu) &=-q^{1+ 2\langle i , \mu \rangle}_{ i} {g} (\mu + i'), &\quad \text{for } i \in \Iblack.
\end{align*}
\end{lem}
\begin{proof}
The second identity follows from \eqref{eq:def:zeta2} directly. 
We shall prove now  the first one. Let $i \in \Iwhite$. For any $j \in \Iblack$, by definition \eqref{eq:def:zeta2}, we have
\begin{align*}
{g}(\mu - \nu) &=  {g}(\mu - \nu + \langle j, \nu \rangle j') (-1)^{ \langle j, \nu \rangle} q_j^{2 \langle j, \nu \rangle^2- \langle j, \nu \rangle } q_j^{2 \langle j, \nu \rangle  \langle j , \mu-\nu\rangle} \\
& =  {g}(\mu - s_j(\nu)) (-1)^{ \langle j, \nu \rangle} q_j^{- \langle j, \nu \rangle } q_j^{2   \langle j, \nu \rangle\langle j , \mu\rangle}.
 \end{align*}
 Then a similar computation as Corollary~\ref{cor:rhoTwb} shows that
 \begin{equation}\label{eq:zeta1}
 {g}(\mu - i') =  {g}(\mu - \wb i') (-1)^{ \langle 2\rho^\vee_{\bullet},  { i'} \rangle } q_{ i}^{-\langle { i}, 2\rho_{\bullet} \rangle} q_{i}^{ 2 \langle i -\wb i, \mu  \rangle}.
 \end{equation}
Note that since $\tau i -\tau \wb i \in \Z[\Iblack] \subset Y$, we have $\tau (\tau i - \tau \wb i ) = -\wb (\tau i - \tau \wb i )$. Hence we have $i - \wb i = \tau i - \wb \tau i \in Y$. Then using 
\eqref{eq:def:zeta1}, we see Equation \eqref{eq:zeta1} can be written as 
\[
{g}(\mu) = {g}(\mu - \wb i')  \vs_{i} (-1)^{ \langle 2\rho^\vee_{\bullet},  i' \rangle } q_i^{\langle i, 2\rho_{\bullet} \rangle} q_{i}  q_i^{\langle i, \wb\tau i' \rangle}  (-1)^{ \langle 2\rho^\vee_{\bullet},  {\tau i'} \rangle } q_{\tau i}^{-\langle {\tau i}, 2\rho_{\bullet} \rangle} q_{\tau i}^{\langle {\tau i}, \mu- \wb  i' \rangle} q_{i}^{ -\langle i, \wb\mu  \rangle}.
\]
Now the desired equation follows from the constraint \eqref{vs2} on the parameters $\vs_i \vs_{\tau i}$. 
\end{proof}

The function $g$ induces a $\Qq$-linear map from any finite-dimensional $\U$-module $M$ to itself: 
\[
\widetilde{g}: M\longrightarrow M, 
\qquad 
\widetilde{g}(m) = g (\mu) m, \quad \text{ for } m \in M_{\mu}.
\]

The following lemma is similar to  Corollary~\ref{cor:rhoTwb}, and it can also be read off from the proof of \cite[Lemma~2.9]{BK14}.
\begin{lem}
  \label{lem:barTwb}
For $i \in \I_{\circ}$, we have 
\begin{align*}
\overline {\T_{\wb} (E_{i})}  &= (-1)^{\langle 2\rho^\vee_{\bullet}, i' \rangle} q_i^{-\langle i, 2 \rho_{\bullet} \rangle} \T_{\wb}^{-1} ({E_{i}}).
\end{align*}
\end{lem}

\begin{proof}
Thanks to \cite[\S37.2.4]{Lu94}, we have 
\[
\overline{\T_j (E_i)} = \overline{\T''_{j, +1} (E_i)} = \T''_{j, -1} (\overline{E_i}) = (-q_i)^{-\langle j, i' \rangle}  \T'_{j, -1} ({E_i}) = (-q_i)^{-\langle j, i'\rangle}  \T^{-1}_j ({E_i}).
\]
The rest of the proof is essentially the same as of the proof of Corollary~\ref{cor:rhoTwb}, and will be skipped.
\end{proof}

Recall we denote by $\eta = \eta_{\lambda}$ the highest weight vector in $L(\lambda)$.
 Let $\eta^{\bullet}= \eta^{\bullet}_{\lambda}$ be the unique canonical basis element in $L(\lambda)$ of weight $\wb \lambda$. 

\begin{thm}  (cf. \cite[Theorem~7.5]{BK15}) 
   \label{thm:mcT}
For any finite-dimensional $\U$-module $M$, we have the following isomorphism of $\Ui$-modules 
\[
\mc{T} := \Upsilon \circ \widetilde{g} \circ  \T^{-1}_{\wb} : M \longrightarrow {}^\vartheta M.
\]

In particular, we have the isomorphism of $\Ui$-modules
\[
\mc{T} : L(\lambda) \longrightarrow {}^\omega L(\lambda^{\tau}),
\qquad \etab_{\lambda} \mapsto \xi_{-\lambda^{\tau}}.
\]
Moreover, we can choose a function ${g}$ such that $\mc{T}$ is an isomorphism of the $\mA$-form ${}_\mA L(\lambda) \longrightarrow {}^\omega_\mA L(\lambda^{\tau})$ [once we establish Theorem~\ref{thm:intUpsilon}].
\end{thm}

\begin{proof}
It is clear that $\mc{T}$ is a $\Qq$-linear isomorphism. Thus it suffices to verify that 
$\mc{T}$ defines a homomorphism of $\Ui$-modules. We shall prove $
{}^\vartheta( \mc{T} (\vartheta(u) \cdot m)) = \vartheta(u)  \cdot {}^\vartheta( \mc{T} (  m)) = {}^\vartheta(u   \cdot\mc{T}  (  m))$, or equivalently, 
\[
\tag{$\clubsuit$} \mc{T} (\vartheta(u) \cdot m) = u \cdot \mc{T}  ( m), \quad \text{ for } u \in \Ui, \quad m\in M_\mu.
\]

The identity $(\clubsuit)$ is clear for $ u = K_{\nu}$  $(\nu \in Y^{\imath})$, and we verify it on generators $u = B_i$ $(i \in \Iwhite)$, $u = F_j, E_j$ $(j \in \Iblack)$ in (1)--(3) below.

Case (1). We first check this identity $(\clubsuit)$ for $u = B_i$ with $i \in \Iwhite$.
Recall $\ff_i = F_i + \vs_i \T_{\wb} (E_{\tau i}) \tK^{-1}_i + \kappa_i \tK^{-1}_{i}$. By Corollary~\ref{cor:rhoTwb} we have
\begin{align*}
\vartheta (\ff_i ) &= q_{\tau i}  E_{\tau i} \tK_{\tau i} + \vs_i   \sigma \circ \wp \circ \tau \circ \T_{\wb} (E_{\tau i}) \tK_{\tau i} +  \kappa_i \tK_{\tau i} \\
& = q_{\tau i}  E_{\tau i} \tK_{\tau i} + \vs_i   (-1)^{\langle 2\rho^\vee_{\bullet}, i' \rangle} q_i^{\langle i, 2 \rho_{\bullet} \rangle}  \T_{\wb} (\sigma \circ \wp  (E_{ i})) \tK_{\tau i} +  \kappa_i \tK_{\tau i} \\
%
%
& = q_{\tau i}  E_{\tau i} \tK_{\tau i} + \vs_i  (-1)^{\langle 2\rho^\vee_{\bullet}, i' \rangle} q_i^{\langle i, 2 \rho_{\bullet} \rangle}  q_i \T_{\wb} ( F_i)  \T_{\wb}(\tK_{-i}) \tK_{\tau i} +  \kappa_i \tK_{\tau i}.
\end{align*}
Therefore we have 
\begin{align*}
\T^{-1}_{\wb} \circ \vartheta (\ff_i ) =  &q_{\tau i}  \T^{-1}_{\wb}  (E_{\tau i}) \T^{-1}_{\wb}(\tK_{\tau i}) \\
 &+  \vs_i   (-1)^{\langle 2\rho^\vee_{\bullet}, i' \rangle} q_i^{\langle i, 2 \rho_{\bullet} \rangle} q_i F_i  \tK_{ -i} \T^{-1}_{\wb}(\tK_{\tau i}) +  \kappa_i \T^{-1}_{\wb}(\tK_{\tau i}).
\end{align*}
On the other hand, thanks to Lemma~\ref{lem:barTwb}, we have 
\begin{align*}
\psi (\ff_i) &= F_i + \vs^{-1}_i \overline{\T_{\wb} (E_{\tau i})} \tK _i + \kappa_i \tK _{i} \\
&= F_i + \vs^{-1}_i (-1)^{ \langle 2\rho^\vee_{\bullet}, \tau i' \rangle } q_{\tau i}^{-\langle \tau i, 2\rho_{\bullet} \rangle} \T^{-1}_{\wb} (E_{\tau i}) \tK _i + \kappa_i \tK _{i}.
\end{align*}
The left-hand side of $(\clubsuit)$ can now be computed as follows: 
\begin{align*}
&\mc{T} ( \vartheta (B_i) m )= \Upsilon \circ \widetilde{g} \circ\T^{-1}_{\wb} \big( \vartheta (B_i) m  \big)
=  \Upsilon \circ \widetilde{g}  \Big(\big(\T^{-1}_{\wb} \circ \vartheta (B_i) \big(  \T^{-1}_{\wb} ( m) \big)  \big) \Big) \\
= &\Upsilon \circ \tilde{g} \Big(  q_{\tau i}  \T^{-1}_{\wb}  (E_{\tau i}) \T^{-1}_{\wb}(\tK_{\tau i}) \T^{-1}_{\wb} ( m) + \kappa_i \T^{-1}_{\wb}(\tK_{\tau i})  \T^{-1}_{\wb} ( m) \\
&\phantom{=}+ \vs_i   (-1)^{\langle 2\rho^\vee_{\bullet}, i' \rangle} q_i^{\langle i, 2 \rho_{\bullet} \rangle} q_i F_i  \tK_{ -i} \T^{-1}_{\wb}(\tK_{\tau i})  \T^{-1}_{\wb} ( m)   \Big)\\
= &\Upsilon  \Big(   g(w_{\bullet}\mu + w_{\bullet}\tau i') q_{\tau i}  \T^{-1}_{\wb}  (E_{\tau i}) \T^{-1}_{\wb}(\tK_{\tau i}) \T^{-1}_{\wb} ( m) + g(w_{\bullet}\mu) \kappa_i \T^{-1}_{\wb}(\tK_{\tau i})  \T^{-1}_{\wb} ( m) \\
&\phantom{=}+  g(w_{\bullet}\mu - i')  \vs_i   (-1)^{\langle 2\rho^\vee_{\bullet}, i' \rangle} q_i^{\langle i, 2 \rho_{\bullet} \rangle} q_i F_i  \tK_{ -i} \T^{-1}_{\wb}(\tK_{\tau i})  \T^{-1}_{\wb} ( m)  ) \Big)\\
= &\Upsilon  \Big(   g(w_{\bullet}\mu + w_{\bullet}\tau i') q_{\tau i}  q_{\tau i}^{\langle \tau i, \mu \rangle} \T^{-1}_{\wb}  (E_{\tau i})  \T^{-1}_{\wb} ( m) + g(w_{\bullet}\mu) q_{\tau i}^{\langle \tau i, \mu \rangle} \kappa_i \T^{-1}_{\wb} ( m) \\
&\phantom{=}+  g(w_{\bullet}\mu - i')  \vs_i   (-1)^{\langle 2\rho^\vee_{\bullet}, i' \rangle} q_i^{\langle i, 2 \rho_{\bullet} \rangle} q_i q_{\tau i}^{\langle \tau i, \mu \rangle} q_{i}^{-\langle i, w_\bullet \mu \rangle} F_i  \T^{-1}_{\wb} ( m)  ) \Big).
\end{align*} 
The right-hand side of $(\clubsuit)$ is given by
\begin{align*} 
B_i \cdot \mc{T} (m)  
=& B_i \Big( \Upsilon \circ  \widetilde{g} \circ\T^{-1}_{\wb} ( m) \Big)
= \Upsilon \Big(\psi(B_i) \big(  \widetilde{g} \circ\T^{-1}_{\wb} ( m) \big) \Big)\\
=& \Upsilon  \Big( g(w_\bullet \mu )F_i \T^{-1}_{\wb} (m)+  g(w_\bullet \mu )  \kappa_i q_i^{\langle i, w_\bullet \mu \rangle} \T^{-1}_{\wb} (m)\\
&+  g(w_\bullet \mu )  \vs^{-1}_i (-1)^{ \langle 2\rho^\vee_{\bullet}, \tau i' \rangle } q_{\tau i}^{-\langle \tau i, 2\rho_{\bullet} \rangle} q_i^{\langle i, w_\bullet \mu \rangle}\T^{-1}_{\wb} (E_{\tau i})  \T^{-1}_{\wb} (m)\Big).
\end{align*} 
Now the identity $\mc{T} (\vartheta(B_i) \cdot m) = B_i \cdot \mc{T}  ( m)$ follows by comparing the coefficients using \eqref{eq:def:zeta1} and Lemma~\ref{lem:zeta}. Note that if $\kappa_i \neq 0$, we have $q_i^{\langle \tau i, \mu\rangle} = q_i^{\langle i, w_\bullet \mu \rangle}$.

Case (2). 
Let $u = F_j$ $(j \in \Iblack)$. We have $\T^{-1}_{\wb} \circ \vartheta (F_i) = -q^{3}_{\tau i} F_i \tK^{-2}_{i}$. So the left-hand side  of $(\clubsuit)$ can be computed as follows: 
\begin{align*}
\mc{T} (\vartheta (F_j) m) 
=& \Upsilon \circ \widetilde{g} \circ \T^{-1}_{\wb} (\vartheta(F_j) m) 
\\
=& \Upsilon \circ \widetilde{g} \Big( -q^{3}_{\tau i} F_i \tK^{-2}_{i}\T^{-1}_{\wb} (m) \Big) \\
= &\Upsilon   \Big( - g (\wb\mu -i') q^{3}_{\tau i} q_i^{-2\langle i, \wb\mu \rangle }F_i \T^{-1}_{\wb} (m) \Big).
\end{align*}
The right-hand side  of $(\clubsuit)$ reads
\begin{align*}
F_j \cdot \mc{T} (m) = F_j \cdot \Upsilon \circ \widetilde{g} \circ \T^{-1}_{\wb} (m) = \Upsilon \big( g(\wb\mu)F_j  \T^{-1}_{\wb} (m) \big).
\end{align*}
Now the desired identity  $(\clubsuit)$ for $u = F_j$ $(j \in \Iblack)$ follows from \eqref{eq:def:zeta2}.

Case (3). 
Let $u = E_j$ $(j \in \Iblack)$. We have $\T^{-1}_{\wb} \circ \vartheta (E_i) = -q_{\tau i} E_i \tK^{2}_{i}$. So the left-hand side of  $(\clubsuit)$ can be computed as follows:  
\begin{align*}
\mc{T} (\vartheta (E_j) m) =& \Upsilon \circ \widetilde{g} \circ \T^{-1}_{\wb} (\vartheta(E_j) m) 
\\
=& \Upsilon \circ \widetilde{g} \Big( -q_{\tau i} E_i \tK^{2}_{i}  \T^{-1}_{\wb} (m) \Big) \\
= &\Upsilon   \Big( - g (\wb\mu + i') q_{\tau i} q_i^{2\langle i, \wb\mu \rangle }E_i \T^{-1}_{\wb} (m) \Big).
\end{align*}
The right-hand side  of $(\clubsuit)$ reads 
\[
E_j \cdot \mc{T} (m) =E_j \cdot \Upsilon \circ \widetilde{g} \circ \T^{-1}_{\wb} (m) = \Upsilon \big( g(\wb\mu)E_j  \T^{-1}_{\wb} (m) \big).
\]
Now the desired identity  $(\clubsuit)$ for $u =E_j$ $(j \in \Iblack)$ follows from Lemma~\ref{lem:zeta}.

%

The statement on the $\mA$-form follows from the definition of the function $g$ now.
This finishes the proof.
%
%
\end{proof}

\begin{rem}
The $\Ui$-module isomorphism $\mc T$ was first constructed in the special case 
of quantum symmetric pairs of type AIII/AIV with $\Iblack =\emptyset$ \cite[\S2.5, \S6.2]{BW13}, and readily generalized by the authors to quantum symmetric pairs with $\Iblack =\emptyset$ (before the posting of \cite{BK15} in arXiv). The first statement in Theorem~\ref{thm:mcT} in this generality is due to \cite[Theorem~7.5]{BK15} (in which the notation $\mc K'$ in place of $\mc T$ is used). Our construction here uses a $\Qq$-valued function $g$ and a twisting by $\vartheta$,  slightly different from the twisting $\tau \tau_0$ therein, so we can take advantage of earlier results (mainly Corollary~\ref{cor:rhoTwb}). It is further shown in \cite{BK15}  that $\mc K'$ leads to a universal $K$-matrix $\mc K$, which provides solutions to the quantum reflection equation (just like Drinfeld's universal ${R}$-matrix provides solutions to Yang-Baxter equation). 

Our second statement on the isomorphism of $\Ui$-modules
$\mc{T} : L(\lambda) \rightarrow {}^\omega L(\lambda^{\tau})$ will be used later on.  
Later as a consequence of Theorem~\ref{thm:intUpsilon} we see that $\mc T$ preserves the $\mA$-forms, i.e., 
$\mc{T} : {}_\mA L(\lambda) \rightarrow {}^\omega_\mA L(\lambda^{\tau}).$
\end{rem}

\section{Integrality of the intertwiner and $\imath$-canonical bases for modules}
  \label{sec:integral}

In this section, we establish the integrality of the intertwiner $\Upsilon$. 
This is first carried out by a case-by-case computation in the real rank one case.
In the real rank one case, the integrality of $\Upsilon$ eventually leads to the existence of $\imath$-canonical basis of $\Uidot$, which ensures
the existence of (integral) $\imath$-divided powers for all $i \in \I$. The existence of $\imath$-divided powers is then used here to complete the proof
of integrality of $\Upsilon$ in the general finite type case. 
We then construct an $\imath$-canonical basis on any finite-dimensional simple $\U$-module $L(\la)$
as well as on the tensor products of several such simple modules.

\subsection{Bar involution $\ipsi$ on modules} 
Recall \cite[Chapter 27]{Lu94} has developed a theory of finite-dimensional based $\U$-modules $(M,B)$.
The basis $B$ generates a $\Z[q^{-1}]$-submodule $\mc{M}$ and an $\mA$-submodule ${}_\mA M$ of $M$.

Recall from Lemma~\ref{lem:bar} the bar involution $\ipsi$ on $\Ui$. 
Recall \cite[Definition~3.9]{BW13} that a $\Ui$-module $M$ equipped with an anti-linear involution $\ipsi$ 
is called {\em involutive} (or {\em $\imath$-involutive}) if
$$
\ipsi(u m) = \ipsi(u) \ipsi(m),\quad \forall u \in \Ui, m \in M.
$$
The following proposition is \cite[Proposition~3.10]{BW13} verbatim in our more general setting, and we repeat its short proof for the sake of completeness. 

\begin{prop}
  \label{prop:compatibleBbar}
Let $M$ be a based $\U$-module with bar involution $\psi$. Then $M$ is an $\imath$-involutive  
$\Ui$-module with involution
\begin{equation}
\label{ibar}
\ipsi := \Upsilon \circ \psi.
\end{equation}
\end{prop}

\begin{proof}
It follows by Theorem \ref{thm:Upsilon} and \eqref{ibar} that
$
\ipsi(u m) = \Upsilon \psi( u m) = \Upsilon \psi(u) \psi(m) =  \ipsi (u)  \Upsilon \psi (m) = \ipsi (u) \ipsi(m),
$ 
for all $ u \in \Ui$  and $m \in M$. Hence $M$ is $\imath$-involutive. 
By Corollary~\ref{cor:barUpsilon},  we have
$
\ipsi(\ipsi(m)) = \Upsilon \psi(\Upsilon \psi (m)) = \Upsilon \ov{\Upsilon} \psi (\psi(m)) = \Upsilon \ov{\Upsilon} m = m.
$ 
 Hence $\ipsi$ is an involution.
\end{proof}

Recall Lusztig defined a bar involution $\psi$
on (tensor products of) finite-dimensional simple $\U$-modules (via quasi-$\mc R$-matrix); cf. \cite[27.3]{Lu94}.
 
\begin{cor}
  \label{cor:ipsi-lambda}
There is a bar involution $\ipsi =\Upsilon \psi$ 
on the $\U$-module $L(\la)$ 
and  on the tensor product $\U$-modules $L(\la_1)  \otimes \cdots \otimes L(\la_r)$,
for $\la, \la_1, \ldots, \la_r \in X^+$, and $r\ge 1$. 
\end{cor}

\subsection{Integrality of the intertwiner}

In this section we prove the integrality of the intertwiner for an arbitrary finite type. 

\begin{thm}
\label{thm:intUpsilon}
{\quad}
\begin{enumerate}
	\item For quantum symmetric pair $(\U, \Ui)$ of real rank one,  the intertwiner $\Upsilon$ is integral; that is, $\Upsilon = \sum_{\mu} \Upsilon_{\mu}$ with $\Upsilon_\mu \in {}_\mA \U^+$ for all $\mu$.
	\item 
	Let $(\U, \Ui)$ be any quantum symmetric pair of finite type. 
	Under the assumption of the validity of Theorem~\ref{thm:iCBUi}  for all quantum symmetric pairs of real rank one,  the intertwiner $\Upsilon$ for $(\U, \Ui)$ 
	is integral; that is, $\Upsilon = \sum_{\mu} \Upsilon_{\mu}$ with $\Upsilon_\mu \in {}_\mA \U^+$ for all $\mu$.
\end{enumerate}
\end{thm}

\begin{proof}

Part (1) is proved by a tedious though straightforward case-by-case computation in Appendix~\ref{sec:Upsilonrank1}.

Let us prove (2). (The reader is supposed to know Theorem~\ref{thm:iCBUi} in the special case of real rank one, which will be established using Part (1) only.)
Fix an $i \in \Iwhite$. There exists a Levi subalgebra $\U^\imath_i$ of $\Ui$ containing $\ff_i$ of real rank one (see Table~\ref{table:decomposition}). 
Consider ${}_\mA \Uidot_i = \sum_{\zeta \in X_\imath} {}_\mA \U^\imath_i \one_\zeta$, and 
the canonical basis elements 
$\ff_i^{(a)} := (1 \diamondsuit^\imath_\zeta F^{(a)}_{i})\in {}_\mA \Uidot_i$ (see Theorem~\ref{thm:iCBUi}),
for $a \ge 0$ and $\zeta \in X_\imath$. 
We have a natural embedding ${}_\mA \Uidot_i \hookrightarrow {}_\mA \Uidot$. 
By abuse of notation we shall denote by the same notation $\ff_i^{(a)}$ for the image of $\ff_i^{(a)}$ in $\Uidot$. 
It follows by Theorem~\ref{thm:iCBUi} (for $\Uidot_i$) that $\ff_i^{(a)} \in {}_\mA \Uidot$. 

Denote by ${}_\mA'\Uidot$ the $\mA$-subalgebra of $\Uidot$ generated by $\ff_i^{(a)}$ for $i \in \Iwhite$,
and $F_j^{(a)} \one_\zeta, E_j^{(a)} \one_\zeta$ for $j \in \Iblack$,  for all $a \ge 0$ and $\zeta \in X_\imath$. 
By Corollary~\ref{cor:AA},   $\ff_i^{(a)}$ in ${}_\mA \Uidot$ preserves ${}_\mA L(\lambda)$, for all $\lambda \in X^+.$ 
As the other generators of ${}_\mA'\Uidot$ clearly preserve  ${}_\mA L(\lambda)$,
we have ${}_\mA'\Uidot {}_\mA L(\lambda) \subseteq {}_\mA L(\lambda)$. Actually a stronger statement holds as follows.

{\bf Claim ($\star$).} 
We have ${}_\mA'\Uidot \eta_\lambda = {}_\mA L(\lambda), \quad \text{ for } \lambda \in X^+.$

A spanning set for ${}_\mA L(\lambda)$ is given by $F_{i_1}^{(a_1)} F_{i_2}^{(a_2)} \cdots F_{i_s}^{(a_s)} \eta$ for various $s\geq 0$, $i_j \in \I$ and $a_j \geq 0$. 
We shall argue  that $x= F_{i_1}^{(a_1)} F_{i_2}^{(a_2)} \cdots F_{i_s}^{(a_s)} \eta \in {}_\mA'\Uidot \eta_\lambda$,
by induction on  the height $\text{ht}(x) =\sum_{j=1}^s a_j$. 
We can assume without loss of generality that $a_1>0$, and so $x' :=F_{i_2}^{(a_2)} \cdots F_{i_s}^{(a_s)} \eta$ 
lies in ${}_\mA'\Uidot \eta_\lambda$ by the inductive assumption.
If $i_1 \in \Iblack$, then $F_{i_1}^{(a_1)}  \in {}_\mA'\Uidot$ and $x= F_{i_1}^{(a_1)} x'  \in {}_\mA'\Uidot x' \in {}_\mA'\Uidot \eta_\lambda$.
Assume now $i_1 \in \Iwhite$. Define $y =\ff_{i_1}^{(a_1)} x' \in {}_\mA'\Uidot \eta_\lambda$. 
As $x-y =(\ff_{i_1}^{(a_1)}  -F_{i_1}^{(a_1)} ) x'\in {}_\mA L(\lambda)$ has height less than the 
height of $x$,  we have $x-y \in  {}_\mA'\Uidot \eta_\lambda$ by the inductive assumption, and so we also have $x=y+(x-y) \in {}_\mA'\Uidot \eta_\lambda$. 
This proves Claim~($\star$). 
 
The $\mA$-algebra ${}_\mA'\Uidot$ is clearly stable under the bar map $\ipsi$, and recall $\ipsi(\eta) =\eta$. 
It follows by Claim ($\star$) and Proposition~\ref{prop:compatibleBbar} that ${}_\mA L(\lambda)$ is $\ipsi$-invariant.  
Hence ${}_\mA L(\lambda)$ is stable under the action of $\Upsilon=\ipsi \circ \psi$.  In particular, we have (recall $w_0$ is the longest element in $W$)
$
\Upsilon \eta_{w_0\lambda} \in {}_\mA L(\lambda),$ for $\lambda \in X^+.$
By taking $\lambda \gg 0$, we conclude that $\Upsilon_\mu \in {}_\mA \U^+$, for each $\mu$.  
\end{proof}

\begin{rem}
In the proof above, we only need to assume the validity of Theorem~\ref{thm:iCBUi} for all the Levi subalgebras of $\Ui$ of 
real rank one (see Table~\ref{table:decomposition}).  For example the integrality of $\Upsilon$ for type FII is not used 
in the proof of any other quantum symmetric pairs.
\end{rem}

\begin{rem}
 \label{rem:logic}
Logically, the reader should read through the remainder of the paper under the additional assumption of real rank one so Theorem~\ref{thm:iCBUi} for $\Ui$ of all
real rank one (which is the assumption of Theorem~\ref{thm:intUpsilon}(2)) is established fully.
Then the integrality of $\Upsilon$ for $\Ui$ of any finite type follows by Theorem~\ref{thm:intUpsilon}(2).
\end{rem}

{\em In the remainder of the paper, we shall use the integrality of $\Upsilon$ for any finite type freely without mentioning further the assumption in Theorem~\ref{thm:intUpsilon}(2), thanks to Remark~\ref{rem:logic}.
}

\subsection{$\imath$-Canonical bases on based $\U$-modules}\label{subsec:imathpartial}

Recall the quotient map $X \rightarrow X_\imath$, $\mu \mapsto \ov{\mu}$. 
We define a partial ordering $\ile$ on $X$ by letting, for $\mu', \mu \in X$, 
\begin{equation}
 \label{eq:iorder}
 \mu' \ile \mu \quad  \Leftrightarrow   \quad  
  \overline{\mu'} = \overline{\mu},  
   \text{ and }  \mu' -\mu \in \N[\I] \cap \N[\wb \I]. 
\end{equation}
Note that if $\overline{\mu'} = \overline{\mu}$ 
then $\overline{ \wb\tau \mu'} = \overline{ \wb\tau \mu}$ by definition of $X_\imath$. We also define $\mu' \il \mu$ if $\mu' \ile \mu$ and $\mu' \neq \mu$. 

We formally extend the partial ordering $\ile$ to any set $S$ with a natural weight function $|\cdot |: S\rightarrow X$ (such as $\B(\la)$ for $\la \in X^+$, 
or any basis $B$ in a based $\U$-module below), by declaring that 
\begin{equation*}
 b' \ile b \quad  \Leftrightarrow   \quad  |b'| \ile |b|, \qquad \text{ for all } b', b \in S.
\end{equation*}

%

\begin{lem}\label{lem:ipsipartialodering}
Let $(M,B)$ be a finite-dimensional based $\U$-module (cf. \cite[Definition~27.1.2]{Lu94}).
Then we have 
\[
\ipsi (b) = b + \sum_{b' \il b} f(b;b') b', \quad \text{ for } b' \in \B,   f(b;b') \in \mA.
\]
\end{lem}

\begin{proof}
Since $\psi(b)=b$ for $b\in B$ and $\ipsi =\Upsilon \psi$, we have by Theorem~\ref{thm:Upsilon} and Proposition~\ref{prop:wcirc} that
\[
\ipsi (b) = \Upsilon(b) = b + \sum_{b'\in B, b' \il b} f(b;b') b',\quad \text{ for }  f(b;b') \in \Qq.
\]
Note that $ f(b;b')  \in \mA$ by Theorem~\ref{thm:intUpsilon}. 
\end{proof}

By applying a standard procedure (cf. \cite[Lemma~24.2.1]{Lu94}) to the involution $\ipsi$
with the help of Lemma~\ref{lem:ipsipartialodering}, we have proved the following.

\begin{thm}   \label{thm:iCBbased}
Let $(M,B)$ be a finite-dimensional based $\U$-module. 
 
\begin{enumerate}
\item
The $\Ui$-module $M$ admits a unique basis (called $\imath$-canonical basis)
$
B^\imath  := \{b^\imath \mid b \in B \}
$
which is $\ipsi$-invariant and of the form
\begin{equation} \label{iCB}
b^\imath = b +\sum_{b' \in B, b' <_\imath b}
t_{b;b'} b',
\quad \text{ for }\;  t_{b;b'} \in q^{-1}\Z[q^{-1}].
\end{equation}

\item
$B^\imath$ forms an $\mA$-basis for the $\mA$-lattice ${}_\mA M$ (generated by $B$), and
forms a $\Z[q^{-1}]$-basis for the $\Z[q^{-1}]$-lattice $\mc{M}$ (generated by $B$).
\end{enumerate}
\end{thm}   

\begin{rem}
When $\I =\Iblack$, we have $- \I = \wb \I$, $X = X_\imath$, and hence $b \ile b'$ actually means $|b| = |b'|$.
Therefore, in this case the $\imath$-canonical basis reduces to the usual canonical basis. 
\end{rem}

\begin{rem}
Similar to Lusztig's canonical basis, the $\imath$-canonical bases are computable algorithmically. 
The $\imath$-canonical basis  is not homogenous in terms of the weight lattice $X$,
though it is homogenous in  $X_\imath$.  
\end{rem}

Recall for $\lambda \in X^+$, we denote by ${}_\mA L(\lambda)$ (respectively, $\mc{L}(\lambda)$)  the $\mA$-lattice 
(respectively, the $\Z[q^{-1}]$-lattice) spanned  by $\{b^- \eta \vert b \in \B(\lambda)\}$. 
The following theorem is an important special case of Theorem~\ref{thm:iCBbased}, since $L(\la)$ is 
well known \cite{Lu90, Ka91} to be a based $\U$-module.

\begin{thm}\label{thm:iCBonL}
\begin{enumerate}
	\item	 
	For any $b \in \B$, there is a unique element $(b^-\eta)^\imath \in L(\lambda)$ 
	which is $\ipsi$-invariant and of the form 
	\[
	(b^-\eta)^\imath \in b^-\eta + \sum_{b' \il b} q^{-1} \Z[q^{-1}] b'^- \eta;
	\]
	
	\item	
	The set $\{ (b^-\eta)^\imath \vert b \in \B(\lambda) \}$ forms 
	a $\Qq$-basis of $L(\lambda)$, an $\mA$-basis of ${}_\mA L(\lambda)$, and a $\Z[q^{-1}]$-basis of $\mc{L}(\lambda)$
	(called the $\imath$-canonical basis).
\end{enumerate}
\end{thm}

Recall that a tensor product of several finite-dimensional simple $\U$-modules is a based $\U$-module
by \cite[Theorem~27.3.2, \S27.3.6]{Lu94}. Theorem~\ref{thm:iCBbased} also implies the following. 

\begin{cor}  \label{thm:iCBtensor}
Let $\la_1, \ldots, \la_r \in  X^+$. The tensor product of finite-dimensional simple $\U$-modules
$L (\la_1) 
\otimes \ldots \otimes L (\la_r)$
admits a unique $\ipsi$-invariant basis of the form \eqref{iCB}, where $B$ is understood as 
Lusztig's canonical basis on the tensor product.
\end{cor}

Recall we write $\eta^\bullet_{\lambda}$ or simply $\eta^\bullet$ for the unique canonical basis element in $L(\lambda)$ of weight $\wb \lambda$. 
Moreover, by \cite[Lemma~39.1.2]{Lu94}, we have 
\begin{equation}
  \label{eq:eta-b}
  \eta^\bullet_{\lambda} = \T^{-1}_{\wb} (\eta_\lambda).
\end{equation}

Some $\imath$-canonical basis elements are easy to identify as follows (even though the $\imath$-canonical basis differs from canonical basis in general, for example, already in the natural $\mathfrak{sl}_n$-module; cf. \cite[Remark~5.10]{BW13}).

\begin{cor}
 \label{cor:etab}
For any $b \in \B_{\I_\bullet} (\lambda)$, the element $b^- \eta \in L(\lambda)$ is an $\imath$-canonical basis element.
In particular, $\eta^{\bullet}  \in L(\lambda)$ is an $\imath$-canonical basis element.
\end{cor}

\begin{proof}
We already know that $\psi (b^- \eta) = b^- \eta$ by definition. On the other hand, we 
have $\Upsilon (b^- \eta) = b^- \eta$ for weight reason, by Proposition~\ref{prop:wcirc}. 
So $\ipsi (b^- \eta) =\Upsilon \psi(b^- \eta) = b^- \eta.$ Note $\eta^{\bullet}$ is equal to $b^- \eta \in L(\lambda)$ for some particular $b$.
By the uniqueness, a canonical basis element which is $\ipsi$-invariant must be $\imath$-canonical. 
The corollary follows.
\end{proof}


\subsection{On QSP of Kac-Moody type} 
  \label{sec:KM}
  
  The theory of quantum symmetric pairs of Kac-Moody (KM) type is developed in Kolb \cite{Ko14}. 
As we follow Lusztig's book on quantum groups, we shall assume that the root datum is $X$-regular and $Y$-regular in the sense of \cite{Lu94}. 
We briefly comment on the extensions of  results in Section~\ref{sec:QSP}--\ref{sec:integral}  to QSPs of KM type. 

An $\imath$quantum group of KM type is called {\em locally finite} if all of its Levi subalgebras of real rank one 
have Satake diagrams listed in Table~\ref{table:localSatake}.

On Section~\ref{sec:QSP}. 
All constructions therein make sense for the quantum symmetric pairs $(\U, \Ui)$ of KM type 
(under the assumption ``$\nu_i=1$ for all $i \in \Iwhite$", which is conjectured to always be true in \cite[Conjecture~2.7]{BK14}). 
The conjecture ``$\nu_i=1$ for all $i \in \Iwhite$" holds for $\Ui$ of locally finite KM type, and the values $\vs_i$ are computed 
as in Table~\ref{table:values}. 
The study of $\imath$-canonical basis for general KM type leads to the question of 
studying in depth $\Ui$ of KM type of real rank one.

On Section~\ref{sec:braid}. 
The statements on braid group action
for $\Ui$ make sense for the quantum symmetric pairs $(\U, \Ui)$ of KM type and in more general parameters as in \cite{BK14} 
(under the assumption ``$\nu_i=1$ for all $i \in \Iwhite$" as above).  
However, the statement on the anti-involution $\wp$ for $\Ui$ of KM type 
requires the stronger Condition~\eqref{vs2} on parameters as explained in Remark~\ref{rem:characterization}. We refer to Remark~\ref{rem:BKvsBW} 
for the comparison of parameters and their constraints used in this paper and in \cite{BK14}. 

On Section~\ref{sec:integral}. The construction of $\imath$-canonical bases on based $\U$-modules 
can be extended to some class of QSP of KM type. We shall return to study the $\imath$-canonical bases arising from QSP of KM type in a separate work, as additional new ideas are needed to develop a comprehensive theory of $\imath$-canonical bases  in the KM case.


\section{Canonical basis for the modified $\imath$quantum group $\Uidot$}
  \label{sec:iCBU}
 
In this section we shall formulate and study a projective system of $\Ui$-modules $\big \{ L^\imath(\lambda+\nu^\tau, \mu+\nu)  \big \}_{\nu \in X^+}$, and establish the asymptotic compatibility of $\imath$-canonical bases between these modules. 
Then we construct the $\imath$-canonical basis on the modified $\imath$quantum group $\Uidot$
(Theorem~\ref{thm:iCBUi}).

\subsection{Based modules $L^{\imath} (\lambda, \mu)$}

In this section we shall consider the based submodule $L^\imath(\lambda,\mu) = \U(\eta_{\lambda}^{\bullet} \otimes \eta_{\mu})$
of $L(\lambda) \otimes L(\mu)$ introduced in \eqref{eq:Lbullet}, for $\lambda, \mu \in X^+$. 
Thanks to Corollary~\ref{cor:based} and Theorem~\ref{thm:iCBbased}, we already know 
the existence of the $\imath$-canonical basis on $L^\imath(\lambda,\mu)$. 
The main goal of this subsection is to improve the partial ordering $\le_\imath$ in Theorem~\ref{thm:iCBbased} (1) for the based module $L^\imath(\lambda,\mu)$. 

\begin{rem}
(1) When $\Iblack = \emptyset$, we have $L^{\imath} (\lambda, \mu) \cong L (\lambda+ \mu)$ canonically. 

(2) When $\Iblack = \I$, we have $L^{\imath} (\lambda, \mu) = {}^\omega L(-w_0 \lambda) \otimes L(\mu)$ 
since $\eta_{\lambda}^{\bullet} \otimes \eta_{\mu} = \xi_{w_0\lambda} \otimes \eta_{\mu}$. Then we are back to Lusztig's setting \cite[Chapter~25]{Lu94}. 
\end{rem}

 \begin{lem}
  \label{lem:Uietabotimeseta}
Let $\lambda, \mu \in X^+$. We have
\begin{enumerate}
\item
  $\Ui \etab_\la = L(\lambda)$ and $\Ui \eta_\la = L(\lambda)$;

\item 
$L^{\imath} (\lambda, \mu) = \U (\etab_{\lambda} \otimes \eta_{\mu}) = \Pa(\etab_{\lambda} \otimes \eta_{\mu}) = \Ui (\etab_{\lambda} \otimes \eta_{\mu}).$
\end{enumerate}
\end{lem}

\begin{proof}
Part (1) is a special case of Part (2) by taking $\lambda =0$ or $\mu =0$, and so let us prove (2).
It follow by definition that $L^{\imath} (\lambda, \mu) = \U (\etab_{\lambda} \otimes \eta_{\mu})$, which is equal to
$\Pa(\etab_{\lambda} \otimes \eta_{\mu})$ thanks to $E_i (\etab_{\lambda} \otimes \eta_{\mu}) =0$ for $i\in \Iwhite$.
Since the action of $\Uidot  \one_{\overline{\wb \lambda+\mu}}$ on $\etab_{\lambda} \otimes \eta_{\mu}$ factors through the projection \eqref{eq:comp3}, 
it follows by Lemma~\ref{lem:pimath} that 
$\Uidot (\etab_{\lambda} \otimes \eta_{\mu}) =\Pa(\etab_{\lambda} \otimes \eta_{\mu})$. The lemma is proved. 
%
%
%
%
%
%
\end{proof}

\begin{lem}
   \label{lem:etabotimeseta}
 Let $\lambda, \mu \in X^+$. 
For any $b \in \B_{ \I_\bullet} (\lambda)$, the element $(b^- \eta_\lambda) \otimes \eta_\mu \in L(\lambda) \otimes L(\mu)$
is an $\imath$-canonical basis element. In particular, $\etab_{\lambda} \otimes \eta_\mu = \T^{-1}_{\wb} (\eta_{\lambda}) \otimes \eta_{\mu}$ is an $\imath$-canonical basis element. 
\end{lem}

\begin{proof}
To prove the first statement, it suffices to check that $(b^- \eta_\lambda) \otimes \eta_\mu$ is a Lusztig canonical basis element and $\ipsi$-invariant. 
Indeed, since
$\psi ((b^- \eta_\lambda) \otimes \eta_\mu ) =\Theta ((b^- \eta_\lambda) \otimes \eta_\mu ) =
(b^- \eta_\lambda) \otimes \eta_\mu$, $(b^- \eta_\lambda) \otimes \eta_\mu$ is a Lusztig canonical basis element. 
On the other hand, we have $\ipsi ( (b^- \eta_\lambda) \otimes \eta_\mu) =
\Upsilon ((b^- \eta_\lambda) \otimes \eta_\mu)  = (\Upsilon b^- \eta_\lambda) \otimes \eta_\mu  =(b^- \eta_\lambda) \otimes \eta_\mu$ 
for weight reason. 
Also recall from \eqref{eq:eta-b}
that $\eta^\bullet_\la = \T^{-1}_{\wb} (\eta_\lambda)$ is a canonical basis element of $L(\la)$. The lemma follows. 
\end{proof}

For $N\ge 0$, let $P(N)$ be the $\Qq$-subspace of $\Udot$ spanned by elements of the form 
$b^+_1 b^-_2 \one_{\zeta}$ for $\zeta \in X, (b_1, b_2) \in \B_{\Iblack}   \times \B$ such that $\rm{ht} (|b_2|) \le N$. It is clear that $P(N)\cap \Bdot$ is a 
canonical basis of $P(N)$.

\begin{lem}\label{lem:P(N)}
For $\lambda,\mu \in X^+$, the subspace $P(N)  (\eta^\bullet_\lambda \otimes \eta_\mu)$ of
$ L(\lambda) \otimes L(\mu)$ is stable under the action of $\ipsi$.
\end{lem}

\begin{proof}
Observe by Lemma~\ref{lem:etabotimeseta} 
that $\ipsi(b_1^+b^-(\eta^\bullet_\lambda \otimes \eta_\mu) ) =b_1^+\ipsi(b^-(\eta^\bullet_\lambda \otimes \eta_\mu) )$ for $(b_1, b) \in \B_{\Iblack}   \times \B$.
Hence it suffices to show that 
$$\ipsi(u (\eta^\bullet_\lambda \otimes \eta_\mu) ) \in P(N)  (\eta^\bullet_\lambda \otimes \eta_\mu),
\qquad \text{ for } u  \in \U^-_\nu \text{ with } \rm{ht} (\nu) \le N.
$$

We prove this by induction on $N$. When $N=0$, the statement is trivial.
Let us prove for $u  = F^{a_1}_{i_1} F^{a_2}_{i_2} \cdots F^{a_k}_{i_k}$ with 
$\sum^k_{i=1}{a_i} =N$ that $\ipsi \big(u (\eta^\bullet_\lambda \otimes \eta_\mu) \big) \in  P(N)  ( \eta^\bullet_\lambda \otimes \eta_\mu)$. 
We shall compare $\ipsi\big(u (\eta^\bullet_\lambda \otimes \eta_\mu) \big)$ with $\ff^{a_1}_{i_1} \ff^{a_2}_{i_2} \cdots \ff^{a_k}_{i_k} (\eta^\bullet_\lambda \otimes \eta_\mu)$. 
Since $\ff_{i}$'s are $\ipsi$-invariant, it follows by \eqref{eq:def:ff} that 
\[
	\ipsi \big(u (\eta^\bullet_\lambda \otimes \eta_\mu) \big) \subset 
	\ff^{a_1}_{i_1} \ff^{a_2}_{i_2} \cdots \ff^{a_k}_{i_k} (\eta^\bullet_\lambda \otimes \eta_\mu) + \ipsi \big(P(N-1) (\eta^\bullet_\lambda \otimes \eta_\mu) \big).
\]
Hence the lemma follows by induction.
\end{proof}

Recall the partial order $\leq$ on $X$ from \eqref{eq:leq}.
\begin{definition}
  \label{def:partialorderingBtimesB}
We define a partial ordering $\ile$ on $\B \times \B$,
and hence on $\B_{\Iblack} \times \B$ as well by restriction,
by letting $(b'_1, b'_2) \ile (b_1, b_2) $, for $b, b' \in \B$, if and only if Conditions~(1)--(3) hold:
\begin{enumerate}
	\item	$\overline{|b'_1|  - | b'_2|} = \overline{|b_1|  - | b_2|}$ (in $X_\imath$); 
	\item	 $( |b'_1 | - |b'_2| )- (|b_1 | - |  b_2|) \in \N[\I] \cap \N[\wb \I]$;
	\item	$|b'_2| \le |b_2|$.
 \end{enumerate}
\end{definition}
We say $(b'_1, b'_2) <_\imath (b_1, b_2)$ if $(b'_1, b'_2) \ile (b_1, b_2)$ and $(b'_1, b'_2) \neq (b_1, b_2)$.
This partial ordering is compatible with the partial ordering $\ile$ on $\B$ (by taking $b_1 = b_1' =1$ above). 

\begin{lem}\label{lem:lefinite}
The partial ordering $\le_\imath$ on $\B_{\Iblack} \times \B$ is downward finite, that is, 
for any fixed $(b_1, b_2) \in \B_{\Iblack} \times \B$, there are only finitely many $(b_1', b_2' ) \in \B_{\Iblack} \times \B$ such that $(b_1', b_2') \le_\imath (b_1, b_2)$.
\end{lem}

\begin{proof}
Note that $|b_1'| \in \N[\Iblack]$ for any $b_1' \in \B_{\Iblack}$. 
Thanks to Definition~\ref{def:partialorderingBtimesB}(3), there are only finitely many $b'_2 \in \B$, such that $(b_1', b_2') \le_\imath (b_1, b_2)$. 
Now by Definition~\ref{def:partialorderingBtimesB}(1), we have $\overline{|b_1'|} = \overline{|b_1| - |b_2| + |b_2'|}$. 
There are only finitely many $b_1' \in \B_{\Iblack}$ satisfying this condition thanks to \eqref{XY}. 
\end{proof}
\begin{lem}
  \label{lem:iCBtensorproductrefinement}
Let $\lambda,\mu \in X^+$ and let $\zeta = \wb \lambda + \mu$. For any $(b_1, b_2) \in \B_{\Iblack} \times \B$, we have 
\[
\ipsi \Big((b_1 \diamondsuit_{\zeta} b_2 ) (\eta_\lambda^\bullet \otimes \eta_{\mu}) \Big) = (b_1 \diamondsuit_{\zeta} b_2 ) 
(\eta_\lambda^\bullet \otimes \eta_{\mu}) + \!\!\! \sum_{(b'_1, b'_2) <_\imath (b_1, b_2)} \!\!\!f(b_1, b_2 ; b'_1, b'_2)(b'_1 \diamondsuit_{\zeta} b'_2 ) 
(\eta_\lambda^\bullet \otimes \eta_{\mu}),
\]
where $f(b_1, b_2 ; b'_1, b'_2) \in  \mA$. 
\end{lem}

\begin{proof}
Recall $\ipsi =\Upsilon \circ \psi$. By  Theorem~\ref{thm:intUpsilon} and Lemma~\ref{lem:P(N)}, we have 
\[
\ipsi \Big((b_1 \diamondsuit_{\zeta} b_2 ) (\eta_\lambda^\bullet \otimes \eta_{\mu}) \Big) = 
\sum_{|b'_2| \le |b_2|} f(b_1, b_2 ; b'_1, b'_2)(b'_1 \diamondsuit_{\zeta} b'_2 ) (\eta_\lambda^\bullet \otimes \eta_{\mu}),
\]
where $f(b_1, b_2 ; b'_1, b'_2) \in \mA$ and $(b'_1, b'_2) \in \B_{\Iblack} \times \B$. 

Note that we can first view $L^{\imath} (\lambda, \mu)$ as a based $\U$-module, ignoring the tensor product structure. 
Note that by definition we have $| (b_1 \diamondsuit_{\zeta} b_2 ) (\eta_\lambda^\bullet \otimes \eta_{\mu})  | = |b_1| - |b_2| + \zeta$. 
Hence conditions $(1)$ and $(2)$ can be translated to $ (b'_1 \diamondsuit_{\zeta} b'_2 ) (\eta_\lambda^\bullet \otimes \eta_{\mu})   \le_\imath  (b_1 \diamondsuit_{\zeta} b_2 ) (\eta_\lambda^\bullet \otimes \eta_{\mu}) $. Here $\le_\imath$ is the partial ordering on based $\U$-modules defined in Section~\ref{subsec:imathpartial}. We abuse the notation due to compatibility. 
Therefore thanks to Lemma~\ref{lem:ipsipartialodering}, we see that $f(b_1, b_2 ; b_1, b_2)=1$,
and $f(b_1, b_2 ; b'_1, b'_2) =0$ unless Conditions (1)--(2) in Definition~\ref{def:partialorderingBtimesB} hold. Condition (3) in Definition~\ref{def:partialorderingBtimesB} 
follows from Lemma~\ref{lem:P(N)}.
\end{proof}

Now we obtain the following refinement of Theorem~\ref{thm:iCBbased} for the based module $L^\imath(\lambda, \mu)$.
\begin{prop}
  \label{prop:iCBtensorproduct}
Let $\lambda, \mu \in X^+$ and let $\zeta = \wb \lambda +\mu$ and $\zeta_\imath = \overline{\zeta}$.
	\begin{enumerate}
	\item	For any $(b_1, b_2) \in \B_{\Iblack} \times \B$ , there is a unique element in $L^{\imath}(\lambda, \mu)$, denoted by
	$(b_1 \diamondsuit_{\zeta_\imath} b_2 )_{w_\bullet\la,\mu}^{\imath}$ (or by 
	$\big((b_1 \diamondsuit_{\zeta} b_2 ) (\eta_\lambda^\bullet \otimes \eta_{\mu}) \big)^{\imath}$ sometimes), 
	which is  $\ipsi$-invariant and of the form 
	\begin{equation}  \label{eq:order}
	 (b_1 \diamondsuit_{\zeta} b_2 ) (\eta_\lambda^\bullet \otimes \eta_{\mu})    +\!\!\! \sum_{(b'_1, b'_2) \ile (b_1, b_2)} \!\!\!q^{-1}\Z[q^{-1}] (b'_1 \diamondsuit_{\zeta} b'_2 ) (\eta_\lambda^\bullet \otimes \eta_{\mu}). 
	\end{equation}
	\item	The set $\B^\imath(\la, \mu) =\big\{ (b_1 \diamondsuit_{\zeta_\imath} b_2 )_{w_\bullet\la,\mu}^{\imath} \vert 
	(b_1, b_2) \in \B_{\Iblack} \times \B \big\} \backslash \{0\}$ forms
	an $\mA$-basis of $_{\mA}L^{\imath} (\lambda, \mu)$ and a $\Z[q^{-1}]$-basis of $\mc{L}^{\imath} (\lambda, \mu)$ (defined in \ref{subset:parabolic}) (called the $\imath$-canonical basis).
\end{enumerate}
\end{prop}

\begin{rem}  
  \label{rem:uniq}
By a standard argument (cf. \cite[proof of Proposition~25.1.10]{Lu94}), the uniqueness of the $\imath$-canonical basis elements $(b_1 \diamondsuit_{\zeta_\imath} b_2 )_{w_\bullet\la,\mu}^{\imath}$ does not require the partial ordering constraint in \eqref{eq:order}; that is, $(b_1 \diamondsuit_{\zeta_\imath} b_2 )_{w_\bullet\la,\mu}^{\imath}$ is characterized by being $\ipsi$-invariant and of the form 
	\begin{equation*}
(b_1 \diamondsuit_{\zeta_\imath} b_2 )_{w_\bullet\la,\mu}^{\imath} \in
	 (b_1 \diamondsuit_{\zeta} b_2 ) (\eta_\lambda^\bullet \otimes \eta_{\mu})    +\!\!\! \sum_{(b'_1, b'_2)} \!\!\!q^{-1}\Z[q^{-1}] (b'_1 \diamondsuit_{\zeta} b'_2 ) (\eta_\lambda^\bullet \otimes \eta_{\mu}). 
	\end{equation*}
\end{rem}

Now taking $\mu =0$, we obtain the following generalization of Theorem~\ref{thm:iCBonL}.

\begin{cor}
  \label{cor:iCBL2}
Let $\lambda \in X^+$. Then we have the $\imath$-canonical basis of $L(\lambda)$ which consists of the nonzero elements of the form 
$\big((b_1 \diamondsuit_{\zeta} b_2 ) \eta_\lambda^\bullet \big)^{\imath} 
$
for  $(b_1, b_2) \in \B_{\Iblack} \times \B$. 
\end{cor}

%

\subsection{A projective system of $\Ui$-modules}

 We shall give a construction of a projective system of $\Ui$-modules $\big \{ L^\imath(\lambda+\nu^\tau, \mu+\nu)  \big \}_{\nu \in X^+}$, for fixed $\la, \mu \in X^+$. 
Our construction essentially reduces to \cite[25.1.4--5]{Lu94} in case that $\I =\Iblack$.

\begin{lem}\label{lem:deltaimath}
For any $\nu \in X^+$, there exists a homomorphism of $\Ui$-modules 
\[
\delta^\imath=\delta^\imath_{\nu} : L(\nu^\tau) \otimes L(\nu) \longrightarrow \Qq,
\]
such that $\delta^\imath_\nu (\eta^\bullet_{\nu^\tau} \otimes \eta_{\nu}) = 1$.
\end{lem}

\begin{proof}
Recall the $\Ui$-isomorphism $\mc{T}: L(\nu^\tau) \rightarrow {}^\omega L(\nu)$ in Theorem~\ref{thm:mcT}. 
Define a $\Ui$-homomorphism $\delta^\imath = \delta \circ (\mc{T} \otimes \id)$, where $\delta: {}^\omega L(\nu) \otimes L(\nu) \rightarrow \Qq$ 
is the $\U$-homomorphism defined in  \cite[Proposition~25.1.4]{Lu94}.
Clearly $\delta^\imath_\nu (\eta^\bullet_{\nu^\tau} \otimes \eta_{\nu}) = 1$.
\end{proof}

\begin{prop}
  \label{prop:def:contraction}
For any $\lambda, \mu, \nu \in X^+$, there exists a $\Ui$-homomorphism 
\[
\pi =\pi_{\lambda, \mu, \nu}  :  L(\lambda+\nu^\tau) \otimes L(\mu+\nu) \longrightarrow L(\lambda) \otimes L(\mu)
\]
such that $\pi(\etab_{\lambda+\nu^\tau} \otimes \eta_{\mu+\nu}) = \etab_{\lambda} \otimes \eta_{\mu}$. Hence, 
we have a unique $\Ui$-homomorphism 
\begin{equation}
  \label{eq:LiLi}
\pi: L^{\imath} (\lambda+\nu^\tau, \mu+\nu) \longrightarrow L^{\imath} (\lambda, \mu)
\end{equation}
such that $\pi(\etab_{\lambda+\nu^\tau} \otimes \eta_{\mu+\nu}) = \etab_{\lambda} \otimes \eta_{\mu}$. 
\end{prop}

\begin{proof}
Recall the $\U$-homomorphism $\chi$ from Lemma~\ref{lem:chi} and the $\mc{R}$-matrix $\mc R$ associated with $\U$ defined in \cite[\S32.1]{Lu94}. We shall rescale the $\mc{R}$-matrix accordingly.  
We define a $\U$-homomorphism as the composition $\chi'' = (1 \otimes \mc R \otimes 1) \circ (\chi \otimes \chi)$:
\begin{align*}
 \chi'': L(\lambda+\nu^\tau) \otimes L(\mu+\nu) 
 & \stackrel{\chi \otimes \chi}{\longrightarrow}  L(\nu^\tau) \otimes L(\lambda) \otimes  L(\nu) \otimes L(\mu) 
 \\
&\stackrel{1\otimes \mc{R} \otimes 1}{\longrightarrow}  L(\nu^\tau) \otimes L(\nu) \otimes L(\lambda) \otimes L(\mu)
\end{align*}
such that 
\[
\eta^{\bullet}_{\lambda+\nu^\tau} \otimes \eta_{\mu+\nu} \mapsto \eta_{\nu^\tau}^{\bullet} \otimes \eta_{\nu} \otimes \eta^{\bullet}_{\lambda} \otimes \eta_{\mu}.
\]
Then the $\Ui$-homomorphism $\pi := (\delta^\imath_{\nu} \otimes \id )\circ \chi''$ satisfies that 
$\pi(\etab_{\lambda+\nu^\tau} \otimes \eta_{\mu+\nu}) = \etab_{\lambda} \otimes \eta_{\mu}$. 
The restriction of $\pi$ provides a $\Ui$ -homomorphism \eqref{eq:LiLi}, and its uniqueness 
follows from Lemma~\ref{lem:Uietabotimeseta}.
\end{proof}

\begin{prop}
 \label{piipsi=ipsipi}
The homomorphism $\pi: L^{\imath} (\lambda+\nu^\tau, \mu+\nu) \longrightarrow L^{\imath} (\lambda, \mu)$ commutes with the bar involutions $\ipsi$; that is,
\[
\pi \circ \ipsi (m) = \ipsi \circ \pi (m), \qquad \text{ for all } m \in L^{\imath} (\lambda+\nu^\tau, \mu+\nu).
\]
\end{prop}

\begin{proof}
Thanks to Lemma~\ref{lem:Uietabotimeseta}, we can write $m = u (\etab_{\lambda+\nu^\tau} \otimes \eta_{\mu+\nu})$ for $u \in \Ui$. 
Recall that $\ipsi (\etab_{\lambda+\nu^\tau} \otimes \eta_{\mu+\nu}) = (\etab_{\lambda+\nu^\tau} \otimes \eta_{\mu+\nu})$ by Lemma~\ref{lem:etabotimeseta}. 
We have
\[
\pi \circ \ipsi (m) = \pi \circ \ipsi (u (\etab_{\lambda+\nu^\tau} \otimes \eta_{\mu+\nu}) ) 
= \ipsi(u) \pi(\etab_{\lambda+\nu^\tau} \otimes \eta_{\mu+\nu}) =  \ipsi(u) (\etab_{\lambda} \otimes \eta_{\mu}).
\]
We also have
\[
 \ipsi \circ \pi (m) =  \ipsi \circ \pi (u (\etab_{\lambda+\nu^\tau} \otimes \eta_{\mu+\nu})) 
 = \ipsi(u) \ipsi \circ \pi(\etab_{\lambda+\nu^\tau} \otimes \eta_{\mu+\nu}) =  \ipsi(u) (\etab_{\lambda } \otimes \eta_{\mu}).
\]
The proposition follows.
\end{proof}

\subsection{A stabilization property} 

Denote by $\mc{L}^\imath(\lambda, \mu)$ the $\Z[q^{-1}]$-lattice spanned by the canonical basis 
(hence also the $\imath$-canonical basis) of $L^\imath(\lambda, \mu)$.
By $\la \gg 0$ (say, $\la$ is sufficiently large) 
we shall mean that the integers $\langle i, \lambda \rangle$ for all $i$ are sufficiently large (in particular, we have $\la \in X^+$).

Let us first explain the simple idea for this subsection before going to the technical details. We want to study the contraction map
$\pi $ in \eqref{eq:LiLi} 
``at the limit $\nu \mapsto \infty$" and ultimately prove Proposition~\ref{prop:iCBtoiCB}. 
But the map $\pi$ does not map $ \mc{L}^\imath(\lambda+ \nu^\tau, \mu+\nu)$ to $\mc{L}^\imath(\lambda, \mu)$ in general. 
(In our QSP setting, we do not have at our disposal the $\imath$-counterpart of \cite[25.1. 6]{Lu94} which relies on \cite[25.1.2(b), 25.1.4(b)]{Lu94}.)
We instead study the problem as to whether or not the image of
$
(b_1 \diamondsuit_{\zeta} b_2 ) (\eta_{\lambda+\nu^\tau}^\bullet \otimes \eta_{\mu+\nu}) 
$ under $\pi$ lies in $\mc{L}^\imath(\lambda, \mu)$, for {\em fixed} $(b_1, b_2) \in \B_{\Iblack} \times \B$. 
The idea is to show that
$ \pi \big(
(b_1 \diamondsuit_{\zeta} b_2 ) (\eta_{\lambda+\nu^\tau}^\bullet \otimes \eta_{\mu+\nu}) \big) \in \mc{L}^\imath(\lambda, \mu)$,
for $\lambda, \mu, \nu$ sufficiently large. 
 
In this section, we often require that $\lambda$, $\mu \gg 0$ while fixing $\zeta_\imath = \overline{\wb\lambda + \mu} \in X_\imath$. Note that this is always possible, since we have $\overline{\nu + \wb \tau (\nu)} = 0 \in X_\imath$ for any $\nu \in X$. In other words, we can always replace  $\lambda$ with $\lambda+\nu^\tau$ and $\mu$ with $\mu +\nu$, respectively, for $\nu \in X$. 
\begin{lem} \label{lem:contraction}
Let $\zeta_{\imath} \in X_{\imath}$, $\nu \in X^+$ and   $(b_1, b_2)  \in \B_{\Iblack} \times \B$ be fixed. 
Let $\lambda, \mu \in X^+$ be such that $\wb \lambda + \mu = \zeta$ for some $\zeta \in X$ with $\overline{\zeta}= \zeta_\imath$. 
Hence $\zeta' :=\wb(\lambda + \nu^\tau) + \mu +\nu$ satisfies that $\overline{\zeta'} = \zeta_\imath$.
Then, for $\lambda, \mu \gg 0$, the  map
$
\pi_{\lambda,\mu, \nu}: L^{\imath} (\lambda+\nu^\tau, \mu+\nu) \rightarrow L^{\imath} (\lambda, \mu)
$
satisfies 
\[
\pi_{\lambda,\mu, \nu} \big( (b_1 \diamondsuit_{\zeta'} b_2)  (\eta^{\bullet}_{\lambda+\nu^\tau} \otimes \eta_{\mu+\nu}) \big)
\equiv (b_1 \diamondsuit_{\zeta} b_2)  (\eta^{\bullet}_{\lambda } \otimes \eta_{\mu }), \quad \text{ mod }  q^{-1} \mc{L}^\imath(\lambda, \mu).
\]

\end{lem}

\begin{proof} 
The symbol ``$\lambda\mapsto \infty$" below means that $\langle i, \lambda \rangle$ tends to $\infty$ for all $i \in \I$. 
In this proof we denote by $\lim_{\lambda,\mu \mapsto \infty}$ the limit where $\lambda, \mu \mapsto \infty$ while fixing $\overline {\wb\lambda+\mu} = \zeta_{\imath}$.
We shall proceed in two steps. 

(1) First we prove the lemma in the special case when $b_1 = 1$. 
We write $b=b_2$ to simplify the notation here.  Then we have $1 \diamondsuit_{\zeta'} b = b^- \one_{\zeta'}$. Let 
\[
\Delta (b^-) = 1 \otimes b + \sum_{b'\neq 1} a(b', b'') b'^- \otimes b''^- \widetilde{K}_{-|b'|}, \quad \text{ where } b', b'' \in \B.
\]
It follows by \eqref{eq:LiLi} that 
\begin{align*}
&\chi'' ( b^- (\eta^{\bullet}_{\lambda+\nu^\tau} \otimes \eta_{\mu+\nu})) \\
= &(\eta^{\bullet}_{\nu^\tau} \otimes \eta_{\nu} ) \otimes b (\eta^{\bullet}_{\lambda} \otimes \eta_{\mu} )+ \sum a(b', b'') b'^- (\eta^{\bullet}_{\nu^\tau} \otimes \eta_{\nu} )\otimes b''^- \widetilde{K}_{-|b'|}(\eta^{\bullet}_{\lambda} \otimes \eta_{\mu} ),
\end{align*}
and hence 
\begin{equation}\label{eq:contraction1}
\begin{split}
&\pi_{\lambda,\mu, \nu} (  b^- (\eta^{\bullet}_{\lambda+\nu^\tau} \otimes \eta_{\mu+\nu})) \\
 = &b^- (\eta^{\bullet}_{\lambda} \otimes \eta_{\mu} ) + \sum_{b' \neq 1} a(b', b'') \delta^\imath_{\nu}\Big(b'^- (\eta^{\bullet}_{\nu^\tau} \otimes \eta_{\nu} ) \Big) b''^- \widetilde{K}_{-|b'|}(\eta^{\bullet}_{\lambda} \otimes \eta_{\mu} ).
\end{split}
\end{equation}
By Theorem~\ref{thm:Betabullet}, we have $b''^- (\eta^{\bullet}_{\lambda} \otimes \eta_{\mu} ) \in \B(\lambda,\mu) \cup \{0\}$ for any $b'' \in \B$.

Now if $1\neq b'^- \in \U_{\Iblack}^-$, we have 
\[
\delta^\imath_{\nu}\Big(b'^- (\eta^{\bullet}_{\nu^\tau} \otimes \eta_{\nu} ) \Big) = b'^-\Big(  \delta^\imath_{\nu}(\eta^{\bullet}_{\nu^\tau} \otimes \eta_{\nu} ) \Big) = 0.
\]
Therefore we have (if $b'\neq 1$)
\begin{align}
  \label{eq:contraction2}
  \begin{split}
a(b', b'') \delta^\imath_{\nu}\big(b'^- (\eta^{\bullet}_{\nu^\tau} \otimes \eta_{\nu} ) \big) \widetilde{K}_{-|b'|}(\eta^{\bullet}_{\lambda} \otimes \eta_{\mu} ) & =0,
\\
a(b', b'') \delta^\imath_{\nu}\big(b'^- (\eta^{\bullet}_{\nu^\tau} \otimes \eta_{\nu} ) \big) b''^- \widetilde{K}_{-|b'|}(\eta^{\bullet}_{\lambda} \otimes \eta_{\mu} ) & =0.
\end{split}
\end{align}

If  $b'^- \not \in \U^-_{\Iblack}$, we write $|b'| = \sum_{i \in \Iblack} a_{i} i + \sum_{j \in \Iwhite} c_{j} j$ for $a_i, c_j \ge 0$, and some $c_j \neq 0$. Then we have 
\begin{align*}
\widetilde{K}_{-|b'|}(\eta^{\bullet}_{\lambda} \otimes \eta_{\mu} ) & = \prod_{i \in \Iblack}q_i^{-a_i \langle i, \wb\lambda+\mu   \rangle} \prod_{j \in \Iwhite}q_j^{-c_j \langle j, \wb\lambda+\mu   \rangle} \; \eta^{\bullet}_{\lambda} \otimes \eta_{\mu} \\
& = \underbrace{\prod_{i \in \Iblack}q_i^{-a_i \langle i, \wb\lambda+\mu   \rangle} \prod_{j \in \Iwhite}q_j^{-c_j \langle \wb j,\lambda \rangle}  \prod_{j \in \Iwhite}q_j^{-c_j \langle j, \mu   \rangle} }_{=:q^{-s_{b', \lambda,\mu}}}  \; \eta^{\bullet}_{\lambda} \otimes \eta_{\mu}.
\end{align*}
Since $\overline {\wb\lambda+\mu} = \zeta_{\imath} \in X_{\imath}$, $ \prod_{i \in \Iblack}q_i^{-a_i \langle i, \wb\lambda+\mu   \rangle}$ 
depends only on $\zeta_{\imath}$ (since $\Z[\Iblack] \subset Y^\imath$)
and hence is fixed. Note that $\wb j > 0$, for $j \in \Iwhite$. 
So we have 
\begin{equation}\label{eq:limit1}
\lim_{\lambda,\mu \mapsto \infty} -s_{b', \lambda,\mu} =- \infty. 
\end{equation}

On the other hand, since $\nu$ is fixed, we see that 
\[
\delta^\imath_{\nu}\Big(b'^- (\eta^{\bullet}_{\nu^\tau} \otimes \eta_{\nu} ) \Big)  \in q^{s_{\nu}} \Z[q^{-1}], \quad \text{ for some } s_{\nu} \in \N.
\]
We can choose $s_{\nu}$ to be independent of $b'^-$ thanks to Theorem~\ref{thm:Betabullet} and the finite-dimensionality of $L^\imath(\nu^\tau, \nu)$.

Note that $a(b', b'')$ depends only on $b$ and is independent of $\lambda, \mu$. 
Let $t_{b} \in \N$ such that $q^{-t_{b}} a(b',b'') \in \Z[q^{-1}]$ for all $b', b''$. 
Therefore we have 
\begin{equation} \label{eq:contraction3}
a(b', b'') \delta^\imath_{\nu}\Big(b'^- (\eta^{\bullet}_{\nu^\tau} \otimes \eta_{\nu} ) \Big)  \widetilde{K}_{-|b'|}(\eta^{\bullet}_{\lambda} \otimes \eta_{\mu} ) \in q^{-s_{b', \lambda,\mu} + t_{b}+s_{\nu}} \Z[q^{-1}] (\eta^{\bullet}_{\lambda} \otimes \eta_{\mu} ),
\end{equation}
and
\[
a(b', b'') \delta^\imath_{\nu}\Big(b'^- (\eta^{\bullet}_{\nu^\tau} \otimes \eta_{\nu} ) \Big) b''^- \widetilde{K}_{-|b'|}(\eta^{\bullet}_{\lambda} \otimes \eta_{\mu} ) \in q^{-s_{b', \lambda,\mu} + t_{b}+s_{\nu}} \mc{L}^\imath(\lambda, \mu). 
\]
Since there are only finitely many $b'\preceq b$ (and  $b'' \preceq b$), let $s_{b, \lambda,\mu} = \text{min}\{ s_{b', \lambda,\mu} \vert b' \preceq b\}$. Thus the combination of \eqref{eq:contraction2} and \eqref{eq:contraction3} implies (recall $b' \neq 1$)
\begin{equation}\label{eq:contraction4}
a(b', b'') \delta^\imath_{\nu}\Big(b'^- (\eta^{\bullet}_{\nu^\tau} \otimes \eta_{\nu} ) \Big)  \widetilde{K}_{-|b'|}(\eta^{\bullet}_{\lambda} \otimes \eta_{\mu} ) 
\in q^{-s_{b, \lambda,\mu} + t_{b}+s_{\nu}} \Z[q^{-1}] (\eta^{\bullet}_{\lambda} \otimes \eta_{\mu} ), \text{ for all } b' \preceq b. 
\end{equation}

Equation \eqref{eq:limit1} and the finiteness of the set $\{ s_{b', \lambda,\mu} \vert b' \preceq b\}$ imply that 
\begin{equation}\label{eq:limit2}
\lim_{\lambda,\mu \to \infty} (-s_{b, \lambda,\mu} + t_{b} + s_\nu) = \big(- \lim_{\lambda,\mu \to \infty} s_{b, \lambda,\mu}\big) + t_{b} + s_\nu =- \infty.
\end{equation}

Note that $ 1 \diamondsuit_{\zeta}b = b^{-} \one_{\zeta}$ and $ 1 \diamondsuit_{\zeta'}b = b^{-} \one_{\zeta'}$. 
Summarizing, for $\lambda, \mu \gg 0$, we have 
\[
\pi_{\lambda,\mu, \nu} \big(  (1 \diamondsuit_{\zeta'}b) (\eta^{\bullet}_{\lambda+\nu^\tau} \otimes \eta_{\mu+\nu}) \big) 
\equiv  b^- (\eta^{\bullet}_{\lambda} \otimes \eta_{\mu} ) 
\equiv  (1 \diamondsuit_{\zeta}b) (\eta^{\bullet}_{\lambda} \otimes \eta_{\mu} ), 
\text{ mod }  q^{-1}\mc{L}^\imath(\lambda, \mu).
\]
This finishes the proof in the case when $b_1 =1$. 

\vspace{.2cm}

(2) Now we deal with the general case, by letting $b_1 \in \B_{\Iblack}$ and $b_2 \in \B$ be arbitrarily fixed. We can write 
\[
b_1 \diamondsuit_{\zeta'} b_2 =  \sum_{|{}'b_1| \le |b_1|, |{}'b_2|  \le |b_2|} f({}'b_1, {}'b_2){}'b^+_1 {}'b^-_2 \one_{\zeta'}, \quad f({}'b_1, {}'b_2) \in \mA.
\]
Note that $f({}'b_1, {}'b_2) \in \mA$ depends only on ${}' b_1, {}' b_2$ and $\zeta_{\imath} = \overline{\zeta}$ (by Proposition~\ref{prop:PaCB}). 
Since the set $\{({}'b_1, {}'b_2) \vert |{}'b_1| \le |b_1|, |{}'b_2|  \le |b_2| \}$ is finite, there exits $s_{f} \in \N$ such that $f({}'b_1, {}'b_2) \in q^{s_f} \Z[q^{-1}]$,
for all ${}'b_1, {}'b_2$.

The elements ${}'b^+_1  \in \U_{\Iblack}$ commute with the $\U_{\Iblack}$-homomorphism $\pi_{\lambda,\mu,\nu}$. 
 Following Equation \eqref{eq:contraction1} and the notation therein, we have
\begin{align*}
&\pi_{\lambda,\mu, \nu} \Big({}'b_1^+ {}'b_2^- (\eta^{\bullet}_{\lambda+\nu^\tau} \otimes \eta_{\mu+\nu}) \Big) \\
 = &{}'b_1^+ {}'b_2^- (\eta^{\bullet}_{\lambda} \otimes \eta_{\mu} ) + \sum_{b'_2 \neq 1} a({}'b_2', {}'b_2'') 
 \delta^\imath_{\nu}\big({}'b_2'^- (\eta^{\bullet}_{\nu^\tau} \otimes \eta_{\nu} ) \big) {}'b_1^+ {}'b_2''^- \widetilde{K}_{-|{}'b_2'|}(\eta^{\bullet}_{\lambda} \otimes \eta_{\mu} ).
\end{align*}

Let ${s_{{}'b_1, {}'b''_2}} \in \N$ be such that ${}'b^+_1 {}'b''^-_2 \one_{\zeta} \in q^{s_{{}'b_1, {}'b''_2}} \Z[q^{-1}] \dot{\B} \one_{\zeta}$. 
By Proposition~\ref{prop:PaCB}, ${s_{{}'b_1, {}'b''_2}}$ depends only on $\zeta_\imath$, but not on $\zeta$, since $\langle i, \zeta_\imath \rangle = \langle i, \zeta \rangle$ for all $i \in \Iblack$.

Now for a fixed ${}'b_2 \in \B$, there are only finitely many $|{}'b''_{2}| \le |{}'b_2|$ (and $|{}'b'_{2}| \le |{}'b_2|$). 
Let $s_{{}'b_1, {}'b_2} = \max \big \{{s_{{}'b_1, {}'b''_2}}, \; \forall {} \, 'b''_{2} \text{ with } |{}'b''_{2}| \le |{}'b_2|\big \}$. Then we have 
\begin{equation*}
{}'b^+_1 {}'b''^-_2 \one_{\zeta} \in q^{s_{{}'b_1, {}'b_2}} \Z[q^{-1}] \dot{\B} \one_{\zeta}, \qquad \forall \,{}'b''_{2} \text{ with } |{}'b''_{2}| \le  |{}'b_2|.
\end{equation*}
Let $m_{b_1, b_2} = \max \big\{s_{{}'b_1, {}'b_2}, \; \forall {} \, {}'b_1, {}'b_2 \text{ with }  |{}'b_1| \le  |b_1|, |{}'b_2| \le  |b_2| \big\}$. Then we have 
\[
{}'b^+_1 {}'b''^-_2 \one_{\zeta} \in q^{m_{b_1, b_2}} \Z[q^{-1}] \dot{\B} \one_{\zeta}, 
\qquad \text{ for all } |{}'b''_{2}| \le  |{}'b_2|,  |{}'b_2| \le  |b_2|, |{}'b_1| \le |b_1|.
\]

Then thanks to Equation~ \eqref{eq:contraction4} and Theorem~\ref{thm:Betabullet}, we have, for all $|{}'b_2''| \le |{}'b_2|$ and ${}'b_2' \neq 1$, 
 \begin{align*}
 f({}'b_1, {}'b_2) a({}'b_2', {}'b_2'') \delta^\imath_{\nu}\Big({}'b_2'^- (\eta^{\bullet}_{\nu^\tau} \otimes \eta_{\nu} ) \Big) 
 &{}'b_1^+ {}'b_2''^- \widetilde{K}_{-|{}'b_2'|}(\eta^{\bullet}_{\lambda} \otimes \eta_{\mu} ) \\
 & \in q^{-s_{{}'b_2, \lambda,\mu} + t_{{}'b_2}+s_{\nu}+ m_{b_1, b_2}+s_f} \mc{L}^\imath(\lambda, \mu).
 \end{align*}
 We see $\lim_{\lambda,\mu \to \infty} (-s_{{}'b_2, \lambda,\mu} + t_{{}'b_2}+s_{\nu}+ m_{b_1, b_2}+s_f)= -\infty$, 
 for all $|{}'b_2| \le |b_2|$ and ${}'b_2' \neq 1$. Since there are only finitely many ${}'b_2 \preceq b_2$, 
 we can find $\lambda, \mu \gg0$ such that 
 \begin{equation}\label{eq:contraction5}
 \pi_{\lambda,\mu, \nu} \Big( f({}'b_1, {}'b_2)  {}'b_1^+ {}'b_2^- (\eta^{\bullet}_{\lambda+\nu^\tau} \otimes \eta_{\mu+\nu}) \Big) 
 \equiv f({}'b_1, {}'b_2) {}'b_1^+ {}'b_2^- (\eta^{\bullet}_{\lambda} \otimes \eta_{\mu} ) , \text{ mod }  q^{-1} \mc{L}^\imath(\lambda, \mu).
\end{equation} 
 
Now summarizing, we can find $\lambda, \mu \gg 0$ such that  
\begin{align}
& \pi_{\lambda,\mu, \nu} \Big( (b_1 \diamondsuit_{\zeta'} b_2)  (\eta^{\bullet}_{\lambda+\nu^\tau} \otimes \eta_{\mu+\nu}) \Big) \notag\\
= & \pi_{\lambda,\mu, \nu} \Big(\sum_{{}'b_1, {}'b_2} f({}'b_1, {}'b_2){}'b^+_1 {}'b^-_2  (\eta^{\bullet}_{\lambda+\nu^\tau} \otimes \eta_{\mu+\nu}) \Big) \notag\\
\equiv &\sum_{{}'b_1, {}'b_2} f({}'b_1, {}'b_2){}'b^+_1 {}'b^-_2(\eta^{\bullet}_{\lambda} \otimes \eta_{\mu} )  &\text{ mod }  q^{-1} \mc{L}^\imath(\lambda, \mu) \notag\\
\equiv & (b_1  \diamondsuit_{\zeta} b_2) (\eta^{\bullet}_{\lambda} \otimes \eta_{\mu} )    &\text{ mod }  q^{-1} \mc{L}^\imath(\lambda, \mu)\label{eq:contraction7}.
\end{align}
The last identity \eqref{eq:contraction7} follows from Proposition~\ref{prop:PaCB}. This finishes the proof. 
\end{proof}

Now we improve Lemma~\ref{lem:contraction} by letting $\nu \in X^+$ vary.

\begin{lem}
  \label{lem:contractionCB}
Let $\zeta_{\imath} \in X_{\imath}$ and   $(b_1, b_2)  \in \B_{\Iblack} \times \B$ be fixed. 
Then for all $\lambda, \mu \gg 0$ such that $\zeta :=\wb \lambda + \mu = \zeta$ satisfies $\overline{\zeta}= \zeta_\imath$
and for all $\nu \in X^+$, we have 
\[
\pi_{\lambda,\mu, \nu} \big( (b_1 \diamondsuit_{\zeta'} b_2)  (\eta^{\bullet}_{\lambda+\nu^\tau} \otimes \eta_{\mu+\nu}) \big) 
\equiv (b_1 \diamondsuit_{\zeta} b_2)  (\eta^{\bullet}_{\lambda } \otimes \eta_{\mu }), \quad \text{ mod }  q^{-1} \mc{L}^\imath(\lambda, \mu),
\]
where we have denoted  $\zeta' =\wb(\lambda + \nu^\tau) + \mu +\nu$.
\end{lem}

\begin{proof}
Let $\omega_{i} \in X$ such that $\langle  j, \omega_i \rangle = \delta_{i,j}$ for $i,j\in \I$. Now we can apply Lemma~\ref{lem:contraction} to 
$\nu =  \omega_i$ for each $i \in \I$. Since $\I$ is a finite set, clearly when $\lambda, \mu \gg0$, we have 
\[
\pi: L^{\imath} (\lambda+\tau(\omega_i), \mu+ \omega_i) \longrightarrow L^{\imath} (\lambda, \mu), \qquad \forall i\in \I,
\]
satisfying 
\[
\pi_{\lambda,\mu, \omega_i} \big( (b_1 \diamondsuit_{\zeta'} b_2) (\eta^{\bullet}_{\lambda+\omega_i^\tau} \otimes \eta_{\mu+\omega_i}) \big) 
\equiv (b_1 \diamondsuit_{\zeta} b_2)  (\eta^{\bullet}_{\lambda } \otimes \eta_{\mu }), \quad \text{ mod }  q^{-1} \mc{L}^\imath(\lambda, \mu).
\]
Now the lemma follows by induction on $\text{ht}(\nu)$. 
\end{proof}

Recall  from Proposition~\ref{prop:iCBtensorproduct}
the $\imath$-canonical basis $\big \{ \big((b_1 \diamondsuit_{\zeta} b_2 ) (\eta_\lambda^\bullet \otimes \eta_{\mu}) \big)^{\imath} =  (b_1 \diamondsuit_{\zeta_\imath} b_2)^\imath_{\wb\lambda, \mu}  \big\}$  
of $L^{\imath} (\lambda, \mu)$.

\begin{prop}\label{prop:iCBtoiCB}
Let $\zeta_{\imath} \in X_{\imath}$ and   $(b_1, b_2)  \in \B_{\Iblack} \times \B $. 
Then for all $\lambda, \mu \gg 0$ such that $\zeta:=\wb \lambda + \mu$ satisfies $\overline{\zeta}= \zeta_\imath$ and 
for all $\nu \in X^+$, we have 

\[
\pi_{\lambda,\mu,\nu} \Big(   (b_1 \diamondsuit_{\zeta_\imath} b_2)^\imath_{\wb\lambda+\wb\nu^\tau, \mu+\nu}  \Big) = (b_1 \diamondsuit_{\zeta_\imath} b_2)^\imath_{\wb\lambda, \mu}.
\]
\end{prop}

\begin{proof}
We write $\zeta' =\wb(\lambda + \nu^\tau) + \mu +\nu$.
Recall from Proposition~\ref{prop:iCBtensorproduct} that 
\begin{align*}
	\big((b_1 \diamondsuit_{\zeta} b_2 ) (\eta_{\lambda+\nu^\tau}^\bullet \otimes \eta_{\mu+\nu}) \big)^{\imath} 
	&\in  (b_1 \diamondsuit_{\zeta} b_2 ) (\eta_{\lambda+\nu^\tau}^\bullet \otimes \eta_{\mu+\nu})   \\ 
	&\quad + \sum_{(b'_1, b'_2) \ile (b_1, b_2)} q^{-1}\Z[q^{-1}] (b'_1 \diamondsuit_{\zeta} b'_2 ) (\eta_{\lambda+\nu^\tau}^\bullet \otimes \eta_{\mu+\nu}).
\end{align*}
Applying Lemma~\ref{lem:contractionCB} to all (finitely many in total thanks to Lemma~\ref{lem:lefinite}) $(b'_1, b'_2) \in \B_{\Iblack} \times \B$ such that $(b'_1, b'_2) \ile (b_1, b_2)$,
we have
\begin{equation*}
\pi_{\lambda,\mu, \nu} \big(b'_1 \diamondsuit_{\zeta'} b'_2  (\eta^{\bullet}_{\lambda+\nu^\tau} \otimes \eta_{\mu+\nu}) \big)
 \equiv b'_1 \diamondsuit_{\zeta} b'_2  (\eta^{\bullet}_{\lambda } \otimes \eta_{\mu }), \quad \text{ mod }  q^{-1} \mc{L}^\imath(\lambda, \mu),
\end{equation*}
for all $(b'_1, b'_2) \ile (b_1, b_2)$. Hence we have
\begin{equation}
  \label{eq:pibb}
\pi_{\lambda,\mu,\nu} \Big(   (b_1 \diamondsuit_{\zeta_\imath} b_2)^\imath_{\wb\lambda+\wb\nu^\tau, \mu+\nu}  \Big) = (b_1 \diamondsuit_{\zeta_\imath} b_2)^\imath_{\wb\lambda, \mu}, \quad \text{ mod }  q^{-1} \mc{L}^\imath(\lambda, \mu). 
\end{equation}
We know by Proposition~\ref{piipsi=ipsipi} that 
$\pi_{\lambda,\mu,\nu} \Big(   (b_1 \diamondsuit_{\zeta_\imath} b_2)^\imath_{\wb\lambda+\wb\nu^\tau, \mu+\nu}  \Big)$ is $\ipsi$-invariant. 
Therefore the proposition follows by \eqref{eq:pibb} and the characterization property in Proposition~\ref{prop:iCBtensorproduct}(1) and Remark~\ref{rem:uniq} of the $\imath$-canonical basis element $\big((b_1 \diamondsuit_{\zeta} b_2 ) (\eta_\lambda^\bullet \otimes \eta_{\mu}) \big)^{\imath}$. 
\end{proof}

\subsection{Canonical basis on $\Uidot$}

We are in a position to construct the $\imath$-canonical basis of $\Uidot$.
Recall the $\mA$-subalgebra ${}_\mA \Uidot$  of $\Uidot$ from Definition~\ref{def:mAUidot}. 

\begin{thm}
  \label{thm:iCBUi}
		Let ${\zeta_\imath} \in X_{\imath}$ and $(b_1, b_2) \in B_{\Iblack} \times B$. 
	\begin{enumerate}
		\item	
		There is a unique element $u =b_1 \diamondsuit^\imath_{\zeta_\imath} b_2  \in \Uidot$ such that 
			\[
				u (\eta^{\bullet}_\lambda \otimes \eta_\mu) 
				= (b_1 \diamondsuit_{\zeta_\imath} b_2)_{\wb\lambda,\mu}^\imath				\in L^{\imath} (\lambda, \mu), 
			\]
for all $\lambda, \mu \gg0$ with $\overline{\wb \lambda+\mu} = {\zeta_\imath}$.
		\item	The element $b_1 \diamondsuit^\imath_{\zeta_\imath} b_2$ is $\ipsi$-invariant.
		\item	The set $\Bdot^\imath =\{ b_1 \diamondsuit^\imath_{\zeta_\imath} b_2 \big \vert {\zeta_\imath} \in X_\imath, (b_1, b_2) \in B_{\Iblack} \times B \}$ 
		forms a $\Qq$-basis of $\Uidot$ and an $\mA$-basis of ${}_\mA\Uidot$.
	\end{enumerate}
\end{thm}

\begin{proof}
(1) 
For $N_1, N_2 \ge 0$, let $P^\imath(N_1, N_2)$ be the subspace of $\Uidot$ spanned by elements of the form 
$ \ff^{a_1}_{i_1} \ff^{a_2}_{i_2} \cdots \ff^{a_s}_{i_s} b^+ \one_{\zeta}$ for various $\zeta \in X_\imath$ and 
$b \in \B_{\Iblack}$ such that $ \sum_{i=1}^s a_i \le N_1$ and  ${\rm ht} (|b|) \le N_2$. It is easy to see that $\ipsi \Big( P^\imath(N_1, N_2) \Big) = P^\imath(N_1, N_2)$. 

Let $P(N_1, N_2)$ be the $\Qq$-subspace of $\Padot$ spanned by elements of the form $ b^-_2 b^+_1 \one_{\nu}$  such that 
$b_1 \in \B_{\Iblack}$ with $\rm{ht} (|b_1|) \le N_1$ and $b_2 \in \B$ with  $\rm{ht} (|b_2|) \le N_2$, for various $\nu \in X$.

Recall the $\Uidot$-module homomorphism $p_{\imath,\nu}$ in Lemma~\ref{lem:pimath}. 
For any $N''_1, N''_2 \ge 0$ and $\nu \in X$, we can find (sufficiently large) $N_1, N_2$ and (larger) $N'_1, N'_2$ such that
\begin{equation}\label{eq:CB:1}
P(N''_1, N''_2) \one_{\nu} \subset   p_{\imath,\nu} (P^\imath(N_1, N_2)) \subset P(N'_1, N'_2) \one_\nu.
\end{equation}

Consider $\zeta =\wb \lambda+\mu$, for $\la, \mu \in X$ subject to the constraint $\overline{\zeta} =\zeta_\imath$. 
Assume that  $b'_1 \diamondsuit _{\zeta}b'_2 \in P(N''_1, N''_2)$ for all $(b'_1, b'_2) \le_\imath (b_1, b_2)$. 
We shall use below the alternative and more informative notation 
$\big( (b_1 \diamondsuit_{\zeta} b_2)( \eta^\bullet_\lambda \otimes \eta_\mu) \big)^\imath$ 
for $(b_1 \diamondsuit_{\zeta_\imath} b_2)_{\wb\lambda,\mu}^\imath$. 
Then by 
Proposition~\ref{prop:iCBtensorproduct}, we have 
\[
\big((b_1 \diamondsuit_{\zeta} b_2) (\eta^\bullet_\lambda \otimes \eta_\mu) \big)^\imath \in P(N''_1, N''_2) (\eta^\bullet_\lambda \otimes \eta_\mu).
\]
  Note that by the definition of $p_{\imath,\zeta}$ in Lemma~\ref{lem:pimath} we have 
\[
 u  (\eta^\bullet_\lambda \otimes \eta_\mu) = p_{\imath,\zeta} (u ) (\eta^\bullet_\lambda \otimes \eta_\mu), \quad \text{ for } u \in \Uidot \one_{\zeta},
\]
Therefore by \eqref{eq:CB:1} there exists an element $u \in P^\imath(N_1, N_2)\one_{\zeta}$ such that 
\[
 u (\eta^\bullet_\lambda \otimes \eta_\mu)  = \big((b_1 \diamondsuit_{\zeta} b_2) (\eta^\bullet_\lambda \otimes \eta_\mu) \big)^\imath.
 \]
Now take $\lambda, \mu \gg 0$ so that we have the linear isomorphism 
\[
P(N'_1, N'_2) \one_{\zeta} \cong P(N'_1, N'_2) (\eta^\bullet_\lambda \otimes \eta_\mu).
\] 
Hence it follows by \eqref{eq:CB:1} that such $u \in P^\imath(N_1, N_2)\one_{{\zeta_\imath}}$ is unique (which in particular does not depend on 
the choices of $N_1'', N_1, N_1', N_2'', N_2, N_2'$). 
We write $u = (b_1 \diamondsuit^\imath_{\zeta_\imath} b_2)_{\lambda,\mu}$. 

Now assume we can find another such element 
$(b_1 \diamondsuit^\imath_{\zeta_\imath} b_2)_{\lambda+\nu^\tau,\mu +\nu} \in P^\imath(N'_1, N'_2 ) \one_{{\zeta_\imath}}$ for $\nu \in X^+$ (we can always enlarge $N_1'', N_1, N_1', N_2'', N_2, N_2'$).  
Proposition~\ref{prop:iCBtoiCB} implies that (for $\lambda,\mu \gg 0$)
\begin{align*}
(b_1 \diamondsuit^\imath_{\zeta_\imath} b_2)_{\lambda+\nu^\tau,\mu +\nu}  (\eta^\bullet_\lambda \otimes \eta_\mu)  
= &\pi_{\lambda,\mu, \nu} 
\big( (b_1 \diamondsuit^\imath_\zeta b_2)_{\lambda+\nu^\tau,\mu +\nu} (\eta^\bullet_{\lambda + \nu^\tau}\otimes \eta_{\mu+\nu}) \big)  \\
= &\big((b_1 \diamondsuit_{\zeta} b_2) (\eta^\bullet_\lambda \otimes \eta_\mu) \big)^\imath.
\end{align*}
Therefore by the uniqueness of $(b_1 \diamondsuit^\imath_{\zeta_\imath} b_2)_{\lambda,\mu}$, we have $(b_1 \diamondsuit^\imath_{\zeta_\imath} b_2)_{\lambda,\mu} = (b_1 \diamondsuit^\imath_{\zeta_\imath} b_2)_{\lambda+\nu^\tau,\mu +\nu}$. Hence we can define
\[
b_1 \diamondsuit^\imath_{\zeta_\imath} b_2 = \lim_{\lambda, \mu \to \infty} (b_1 \diamondsuit^\imath_{\zeta_\imath} b_2)_{\lambda,\mu}.
\]
This proves (1). 

\vspace{.3cm}
(2)  Since by Lemma~\ref{lem:Uietabotimeseta} we have 
\[
\big( \ipsi (b_1 \diamondsuit^\imath_{\zeta_\imath} b_2) \big) 
(\eta^\bullet_\lambda \otimes \eta_\mu) 
= \ipsi \big((b_1 \diamondsuit^\imath_{\zeta_\imath} b_2) (\eta^\bullet_\lambda \otimes \eta_\mu) \big) 
= (b_1 \diamondsuit^\imath_{\zeta_\imath} b_2) (\eta^\bullet_\lambda \otimes \eta_\mu),
\]
for all $\lambda, \mu \gg 0$. Hence by the uniqueness from (1), we have 
$\ipsi (b_1 \diamondsuit^\imath_{\zeta_\imath} b_2) =b_1 \diamondsuit^\imath_{\zeta_\imath} b_2$. 
This proves (2). 

\vspace{.3cm}
(3) We first show that $b_1 \diamondsuit^\imath_{\zeta_\imath} b_2$ lies in ${}_\mA \Uidot$. Let $\gamma \in X$. Thanks to Lemma~\ref{lem:Uiintegral}, it suffices to prove that $b_1 \diamondsuit^\imath_{\zeta_\imath} b_2 \cdot \one_{\gamma} \in {}_\mA \Udot$. We can assume that $\overline{\gamma} = \zeta_\imath$, otherwise the claim is trivial.  

We can write $\gamma = \mu - \lambda^\tau$ for some $\mu, \lambda \gg 0$. Let $\zeta=\wb \lambda+\mu$. Note that $\overline{\zeta} =\zeta_{\imath}$. Now thanks to Theorem~\ref{thm:mcT}, we have the following isomorphism of ${}_\mA\Uidot$-modules
\[\begin{split}
\mc{T} \otimes \id: &{}_{\mA}L^\imath(\lambda, \mu) \longrightarrow {}^\omega_{\mA} L(\lambda^\tau) \otimes_{\mA} {}_{\mA}L(\mu),\\
	& \eta^\bullet_\lambda \otimes \eta_\mu \mapsto \xi_{-\lambda^\tau} \otimes \eta_{\mu}.
\end{split}
\]

Since $\mu, \lambda \gg 0$, we have $(b_1 \diamondsuit^\imath_{\zeta_\imath} b_2)(\eta^\bullet_\lambda \otimes \eta_\mu) = \big((b_1 \diamondsuit_{\zeta} b_2)(\eta^\bullet_\lambda \otimes \eta_\mu) \big) ^\imath \in {}_{\mA}L^\imath(\lambda, \mu) $, hence $(b_1 \diamondsuit^\imath_{\zeta_\imath} b_2) (\xi_{-\lambda^\tau} \otimes \eta_{\mu}) \in {}^\omega_{\mA} L(\lambda^\tau) \otimes_{\mA} {}_{\mA}L(\mu)$. Since we know $\mu, \lambda \gg 0$, we conclude that $b_1 \diamondsuit^\imath_{\zeta_\imath} b_2 \cdot \one_{\gamma} \in {}_\mA \Udot$.

Now let  $x$ be any element in  ${}_\mA \Uidot \one_{\zeta_\imath}$. We can assume $x \in P^\imath(N_1, N_2) \one_{\zeta_\imath}$.
Take $\la, \mu \in X$ such that $\zeta=\wb \lambda+\mu$ satisfies $\overline{\zeta} =\zeta_\imath$. 
Then by definition of ${}_\mA \Uidot$, we have 
\[
x (\eta^\bullet_\lambda \otimes \eta_\mu)= \sum_{b_1, b_2} f(b_1,b_2) 
\big ( (b_1 \diamondsuit_{\zeta} b_2)(\eta^\bullet_\lambda \otimes \eta_\mu) \big) ^\imath, \qquad \text{for } f(b_1,b_2) \in \mA.
\]
Thanks to the linear isomorphism $P(N'_1, N'_2) \one_{\zeta} \cong P(N'_1, N'_2) \one_{\zeta}  (\eta^\bullet_\lambda \otimes \eta_\mu)$, we have
\[
x = \sum_{b_1, b_2} f(b_1, b_2)  (b_1 \diamondsuit^\imath_{\zeta_\imath} b_2).
\] 
Hence $\big\{ b_1 \diamondsuit^\imath_{\zeta_\imath} b_2 \big \vert  (b_1, b_2) \in B_{\Iblack} \times B \big\}$ spans ${}_\mA \Uidot\one_{\zeta_\imath}$. 
The linear independence can be proved entirely similarly. 
\end{proof}

\begin{rem}
The condition ``$\lambda, \mu \gg0$" in Theorem~\ref{thm:iCBUi}(1) cannot be removed completely 
in general but can likely be much weakened, 
as suggested in the case of QSP of type AI$_1$ ($\I = \{i\}$ and $\Iblack =\emptyset$) with some particular choice of parameter $\kappa_i \neq 0$ (see \cite[\S4.2]{BW13} for the case $\kappa_i =1$). 
But  this condition can be removed for QSP of type AI$_1$  with $\kappa_i =0$. 
That is, for any $ \lambda, \mu, \nu \in X^+$ and $(b_1, b_2) \in \B_{\Iblack} \times \B$, the map
$\pi: L^\imath(\lambda+\nu^\tau, \mu+\nu) \longrightarrow L^\imath(\lambda,\mu)$
sends
$
(b_1 \diamondsuit_{\zeta_\imath} b_2)_{\wb( \lambda+\nu^\tau),\mu+\nu}^\imath 
\mapsto (b_1 \diamondsuit_{\zeta_\imath} b_2)_{\wb\lambda,\mu}^\imath$, where ${\zeta_\imath} =\overline{\wb \lambda+\mu}$.
The proof follows by the same computation as \cite[Lemma~4.8 and Proposition~4.9]{BW13}. 

We conjecture that $\pi$ maps an $\imath$-canonical basis element to  an $\imath$-canonical basis element or zero for general QSP,
and moreover, the strong compatibility could still hold when fixing the parameters properly. 
\end{rem}

Note $(E^{(a)}_{j} \diamondsuit_{{\zeta_\imath}}^{\imath} 1) = E^{(a)}_{j} \one_{{\zeta_\imath}}$, for $j \in \Iblack$.

\begin{cor}
   \label{cor:generator}
The $\mA$-algebra ${}_\mA \Uidot$ is generated by $1  \diamondsuit_{{\zeta_\imath}}^{\imath} F^{(a)}_{i}$ $(i \in \I)$ 
and $E^{(a)}_{j} \one_{{\zeta_\imath}}$ $(j \in \Iblack)$ for various ${\zeta_\imath} \in X_\imath$ and $a \ge 0$.
Moreover, ${}_\mA \Uidot$ is a free $\mA$-module such that $\Uidot =\Qq \otimes_{\mA} {}_\mA \Uidot$. 
\end{cor}
(These generators of the $\mA$-algebra ${}_\mA \Uidot$ are called $\imath$-divided powers.)

\begin{proof}
Let us denote by ${\bf V}$ the $\mA$-subalgebra of ${}_\mA \Uidot$ generated by
$1  \diamondsuit_{{\zeta_\imath}}^{\imath} F^{(a)}_{i}$ $(i \in \I)$ 
and $E^{(a)}_{j} \one_{{\zeta_\imath}}$ $(j \in \Iblack)$. 
Take $\la, \mu \in X^+$ such that ${\zeta_\imath} = \overline{\wb \lambda+\mu}$. 
By an inductive argument entirely similar to the proof of Claim ($\star$) in the proof of Theorem~\ref{thm:intUpsilon}, 
 we have 
\[
{\bf V} (\etab_\lambda \otimes \eta_\mu) = {}_\mA \U  (\etab_\lambda \otimes \eta_\mu) =  {}_\mA \Pa (\etab_\lambda \otimes \eta_\mu) = {}_\mA {L}^\imath(\lambda,\mu).
\]
Therefore for any $(b_1, b_2) \in \B_{\Iblack} \times \B$, we have 
$
(b_1 \diamondsuit^\imath_{\zeta_\imath} b_2)   (\etab_\lambda \otimes \eta_\mu) \in {\bf V} (\etab_\lambda \otimes \eta_\mu).
$
Retaining the notation from the proof of Theorem~\ref{thm:iCBUi}, we further have 
\[
(b_1 \diamondsuit^\imath_{\zeta_\imath} b_2)   (\etab_\lambda \otimes \eta_\mu)
 \in \big({\bf V} \cap P^\imath(N_1, N_2) \big) (\etab_\lambda \otimes \eta_\mu)
\]
for some $N_1$ and $N_2$ independent of choices of $\lambda, \mu$ (such that ${\zeta_\imath} = \overline{\wb \lambda+\mu}$).
Now taking $\lambda ,\mu \to \infty$, following the same argument as in the proof of Theorem~\ref{thm:iCBUi}, 
we conclude that  
$
b_1 \diamondsuit^\imath_{\zeta_\imath} b_2  \in {\bf V}, 
$
that is, every canonical basis element of ${}_\mA \Uidot$ lies in ${\bf V}$. Hence we have ${\bf V} ={}_\mA \Uidot$, and the corollary follows. 
\end{proof}

We have the following improvement of Lemma~\ref{lem:pimath}.

\begin{cor}
  \label{cor:pimath}
For $\lambda \in X$, the map $p_{\imath,\lambda}: {}_\mA \Uidot \one_{\overline{\lambda}} \longrightarrow {}_\mA\Padot \one_{\lambda}$ 
is an isomorphism of (free) $\mc A$-modules. 
\end{cor}

\begin{proof}
It remains to prove the surjectivity. 
We basically rerun the proof of Lemma~\ref{lem:pimath} with the help of Corollary~\ref{cor:generator}. 
Let $b \in \B_{\Iblack}$. 
As an analogue of \eqref{eq:pimath}, we consider  
\begin{equation}
   \label{eq:surj}
p_\imath \big( 
(1  \diamondsuit_{{\zeta_\imath}}^{\imath} F^{(a_1)}_{i_1})
(1  \diamondsuit_{{\zeta_\imath}}^{\imath} F^{(a_2)}_{i_2})
\cdots (1  \diamondsuit_{{\zeta_\imath}}^{\imath} F^{(a_s)}_{i_s})
\big) 
= F^{(a_1)}_{i_1} F^{(a_2)}_{i_2} \cdots F^{(a_s)}_{i_s} b^+ \one_{\lambda} + \text{ lower terms},
\end{equation}
where the lower terms are a $\mA$-linear combination of $F^{(a_1')}_{j_1}  \cdots F^{(a_t')}_{j_t} {b'}^+ \one_{\lambda}$
for various $b \in \B_{\Iblack}$ and $a_j'$ with $a_1' +\ldots + a_t' < a_1 +\ldots + a_s$. 
It follows by an induction on $a_1 +\ldots + a_s$ that 
$F^{(a_1)}_{i_1} F^{(a_2)}_{i_2} \cdots F^{(a_s)}_{i_s} b^+ \one_{\lambda} \in p_{\imath,\lambda} ( {}_\mA \Uidot \one_{\overline{\lambda}})$.  
The surjectivity now follows. 
\end{proof}

\subsection{Canonical bases for Levi subalgebras of $\Ui$} 

For an admissible subdiagram with root datum $\mathbb{J} \subset \I$, 
recall the Levi subalgebra $\U^\imath_{\mathbb{J}}$ of $\Ui$ from Definition~\ref{def:Levi}. 
Define 
\[
\Udot^\imath_{\mathbb{J}} = \bigoplus_{\mu \in X_\imath} \U^\imath_{\mathbb{J}} \one_\mu \subseteq \Uidot.
\]
As $(\U_{\mathbb{J}}, \U^\imath_{\mathbb{J}})$ forms a quantum symmetric pair of finite type, 
all constructions for $\imath$-canonical bases so far are applicable.  
For $\la, \mu \in X^+$, we have $\U^\imath_{\mathbb{J}}$-module $L^\imath_{\mathbb J} (\la, \mu)$ 
(which reduces to $L^\imath (\la, \mu)$ when $\mathbb J =\I$). 
By construction, we have $L^\imath_{\mathbb J} (\la, \mu) \subseteq L^\imath (\la, \mu)$. 
It follows by the uniqueness in Proposition~\ref{prop:iCBtensorproduct} (and its $\mathbb J$-variant) that 
$\B^\imath(\la,\mu) \cap L^\imath_{\mathbb J} (\la, \mu)$ is the $\imath$-canonical basis of $L^\imath_{\mathbb J} (\la, \mu)$. 
Recall from Theorem~\ref{thm:iCBUi} that $\Bdot^\imath$ is the $\imath$-canonical basis of $\Uidot$,
and $\Udot^\imath_{\mathbb{J}}$ admits an $\imath$-canonical basis. 
By the uniqueness in Theorem~\ref{thm:iCBUi}(1) (and its $\mathbb J$-variant),  the  
$\imath$-canonical basis of $\Udot^\imath_{\mathbb{J}}$ coincides with the subset $\Bdot^\imath \cap \Udot^\imath_{\mathbb{J}}$ of the 
$\imath$-canonical basis of $\Uidot$; this follows from 
a rerun of the proof of Theorem~\ref{thm:iCBUi}.  We summarize this as the following.

\begin{prop}
For an admissible subdiagram with root datum $\mathbb{J} \subset \I$, 
the set $\Bdot^\imath \cap \Udot^\imath_{\mathbb{J}}$  forms the 
$\imath$-canonical basis of $\Uidot_{\mathbb{J}}.$
\end{prop}


\subsection{Bilinear forms}
The results in this subsection generalize \cite[Chapter~26]{Lu94}.

Recall the anti-involution $\wp$ in Proposition~\ref{prop:invol}. For ${\mathbb J} \subset \I$, 
let $w_{\mathbb J}$ be the longest element in the parabolic Weyl group $W_{\mathbb J}$.  
Recall \cite[Chapter~19]{Lu94} there is a unique symmetric bilinear form $(\cdot, \cdot)= (\cdot, \cdot)_\lambda: L(\lambda) \times L(\lambda) \rightarrow \Qq$ such that
	\begin{enumerate}
		\item	$(\eta, \eta) =1$;
		\item	$(ux,y) = (x, \wp(u) y)$ for all $x,y \in L(\lambda)$ and $u \in \U$;
		\item	$(x,y) = 0$ for $x \in L(\lambda)_{\mu}$ and $y \in L(\lambda)_{\mu'}$ unless $\mu =\mu'$. 
	\end{enumerate}

\begin{lem}
  \label{lem:Ubilinearform}
For any ${\mathbb J} \subset  \I$, let $\eta^{\mathbb J} = \T^{-1}_{w_{\mathbb J}} (\eta) \in L(\lambda)$ for $\lambda \in X^+$. 
Then we have $(\eta^{{\mathbb J}}, \eta^{\mathbb J}) =1$.
\end{lem}

\begin{proof}
By \eqref{eq:eta-b} (or \cite[39.1.2]{Lu94}), for a reduced expression $w_{\mathbb J} =s_{i_1} s_{i_2} \cdots s_{i_n}$, 
we have $\eta^{\mathbb J} = F^{(a_1)}_{i_1} F^{(a_2)}_{i_2} \cdots F^{(a_n)}_{i_n} \eta$. 
For any $1 \le k \le n$, we write $\eta^{\mathbb J}_k = F^{(a_k)}_{i_k} F^{(a_{k+1})}_{i_{k+1}} \cdots F^{(a_n)}_{i_n} \eta$. By construction, we have $E_{k-1} \eta^{\mathbb J}_k = F_{k} \eta^{\mathbb J}_k=0$. Assuming $(\eta^{\mathbb J}_k, \eta^{\mathbb J}_k ) =1$ for $k\ge 2$, by an easy $\U_q(\mathfrak{sl}_2)$ computation we have 
\[
(\eta^{\mathbb J}_{k-1}, \eta^{\mathbb J}_{k-1}) = (F^{(a_{k-1})}_{i_{k-1}}\eta^{\mathbb J}_{k}, F^{(a_{k-1})}_{i_{k-1}}\eta^{\mathbb J}_{k}) =(\eta^{\mathbb J}_{k}, \wp(F^{(a_{k-1})}_{i_{k-1}})F^{(a_{k-1})}_{i_{k-1}}\eta^{\mathbb J}_{k})  = (\eta^{\mathbb J}_k, \eta^{\mathbb J}_k )  = 1.
\]
The lemma follows by downward induction on $k$.
\end{proof}

For $\lambda, \mu \in X^+$, we define a bilinear pairing 
$(\cdot, \cdot) = (\cdot, \cdot)_{\lambda,\mu}$ on $L(\lambda) \otimes L(\mu)$, 
and hence on the subspace $L^\imath(\la, \mu)$ by restriction, by letting $(x\otimes x', y\otimes y')_{\lambda,\mu} = (x,y)_\lambda  (x',y')_{\mu}$. 
The following is immediate from Lemma~\ref{lem:Ubilinearform}. 

\begin{cor} 
Let $\lambda, \mu \in X^+$. 
We have $(\eta^{\bullet}_\lambda \otimes \eta_{\mu}, \eta^{\bullet}_\lambda \otimes \eta_{\mu})_{\lambda,\mu} =1$.
\end{cor}

\begin{lem}
  \label{lem:bilinearform}
Let $x, y \in \Uidot \one_{{\zeta_\imath}}$. When  $\lambda, \mu$ tends to $\infty$ 
(with $\overline{\wb\lambda+\mu} \in X_{\imath}$ being fixed and equal to ${\zeta_\imath}$), 
$\big(x(\eta^{\bullet}_\lambda \otimes \eta_{\mu}) , y(\eta^{\bullet}_\lambda \otimes \eta_{\mu}) \big)_{\lambda,\mu} \in \Qq$ converges in $\Q((q^{-1}))$ to an element in $\Qq$.
\end{lem}

\begin{proof}
As the bilinear form $(\cdot, \cdot)$ on $L^\imath(\la, \mu)$ is defined by restriction from the one on $L(\la) \otimes L(\mu)$,
we have by \cite[26.2.2]{Lu94} that, for $u, u' \in \Ui$, 
$$
\big( u(\eta^{\bullet}_\lambda \otimes \eta_{\mu}) , u'(\eta^{\bullet}_\lambda \otimes \eta_{\mu}) \big)_{\lambda,\mu} 
= \big(\eta^{\bullet}_\lambda \otimes \eta_{\mu} , \wp(u) u'(\eta^{\bullet}_\lambda \otimes \eta_{\mu}) \big)_{\lambda,\mu}. 
$$ 
Therefore it suffices to prove the lemma for $x = \one_{{\zeta_\imath}}$ thanks to Proposition~\ref{prop:rho}. Using the triangular decomposition of $\U$ we can write 
\[
y  \one_{\wb\lambda+\mu} = \sum_{b_1, b_2 \in \B}f(y;b_1,b_2, \wb\lambda+\mu ) \, b^-_1 b^+_2 \one_{\wb\lambda+\mu},  
\]
with only finitely many $f(y;b_1,b_2, \wb\lambda+\mu )   \in \Qq$ being nonzero. We have 
\[
\big(\eta^{\bullet}_\lambda \otimes \eta_{\mu}, y(\eta^{\bullet}_\lambda \otimes \eta_{\mu}) \big)_{\lambda,\mu} 
=  \sum_{(b_2, b_1) \in \, \B_{\Iblack} \times \B } f(y;b_1,b_2, \wb\lambda+\mu) \,
\big(\eta^{\bullet}_\lambda \otimes \eta_{\mu}, b^-_1 b^+_2 (\eta^{\bullet}_\lambda \otimes \eta_{\mu}) \big)_{\lambda,\mu}.
\]

Recall the triangular decomposition $\U = \U^- \U^0 \U^+$. Following from the embedding $\imath : \Ui \rightarrow \U$ and 
a straightforward computation on the generators,   
 the coefficient $f(y;b_1,b_2, \wb\lambda+\mu )$  results from applying $u  \one_{\wb\lambda+\mu}$ to 
$\big(\eta^{\bullet}_\lambda \otimes \eta_{\mu}, b^-_1 b^+_2 (\eta^{\bullet}_\lambda \otimes \eta_{\mu}) \big)_{\lambda,\mu}$, 
 for  some $u \in \Big \langle \widetilde{K}_{-i} \widetilde{K}_{-\tau i} (i \in \Iwhite), \widetilde{K}_{-i} (\kappa_i \neq 0),K_{\alpha} (\alpha \in Y^\imath) \Big \rangle$. 
  Therefore $f(y;b_1,b_2, \wb\lambda+\mu )$ converges in $\Q((q^{-1}))$ to an element in $\Qq$. 
Moreover, we have 
\[
\big( \eta^{\bullet}_\lambda \otimes \eta_{\mu}, b^-_1 b^+_2 (\eta^{\bullet}_\lambda \otimes \eta_{\mu}) \big)_{\lambda,\mu} 
= \big( \wp(b_1^-)(\eta^{\bullet}_\lambda \otimes \eta_{\mu}), b^+_2 (\eta^{\bullet}_\lambda \otimes \eta_{\mu}) \big)_{\lambda,\mu} =  0, \; \text{ unless } b_1, b_2 \in \B_{\I_\bullet}.
\]
For $b _1, b_2 \in \B_{\I_\bullet}$, the bilinear pairing $\big(\eta^{\bullet}_\lambda \otimes \eta_{\mu}, b^-_1 b^+_2 (\eta^{\bullet}_\lambda \otimes \eta_{\mu}) \big)_{\lambda,\mu} $ 
converges in $\Q((q^{-1}))$ to an element in $\Qq$ since we can regard this bilinear pairing as for $L_{\Iblack}^\imath(\lambda,\mu)$, the $\U_{\Iblack}$-submodule of 
$L^\imath(\lambda,\mu)$ generated by $\eta^{\bullet}_\lambda \otimes \eta_{\mu}$, and then apply \cite[26.2.3]{Lu94} to $\Udot_{\Iblack}$. The lemma follows. 
\end{proof}

\begin{definition}\label{def:Uibilinear}
We define a symmetric bilinear form $(\cdot , \cdot): \Uidot \times \Uidot \rightarrow \Q(q)$  as follows:
\begin{enumerate}
	\item	For $x \in \Uidot\one_{\zeta_\imath}$ and $y \in \Uidot\one_{\zeta_\imath'}$ with ${\zeta_\imath} \neq \zeta_\imath'$, we let $(x, y) =0$.
	\item	For $x, y \in \Uidot \one_{\zeta_\imath}$, we let 
			\[
			(x , y) = \lim_{(\lambda, \mu) \mapsto \infty}
			 \big( x(\eta^{\bullet}_\lambda \otimes \eta_{\mu}) , y(\eta^{\bullet}_\lambda \otimes \eta_{\mu}) \big)_{\lambda,\mu}.
			\]
			(Here $\lim\limits_{(\lambda, \mu) \mapsto \infty}$ is understood as in Lemma~\ref{lem:bilinearform}.)
\end{enumerate}
\end{definition}

We have the following corollary to Lemma~\ref{lem:bilinearform} and its proof.
\begin{cor} 
  \label{cor:formUi}
For all $x, y \in \Uidot$ and $u \in \Ui$, we have 
 $(ux,y) = (x, \wp(u) y)$.
\end{cor}
Let ${\bf A} = \Q[[q^{-1}]] \cap \Qq$.
\begin{thm}
  \label{thm:orth}
	The $\imath$-canonical basis $\Bdot^\imath$ of $\Uidot$ is almost orthonormal in the following sense: 
	for ${\zeta_\imath}, \zeta_\imath' \in X_\imath$ and $(b_1, b_2), (b'_1, b'_2) \in \B_{\Iblack} \times \B$, we have 
\[
(b_1 \diamondsuit^\imath_{\zeta_\imath} b_2, b'_1 \diamondsuit^\imath_{\zeta_\imath'} b'_2) \equiv \delta_{{\zeta_\imath}, \zeta_\imath'} \delta_{b_1, b'_1} \delta_{b_2, b'_2}, \quad \text{ mod } q^{-1} {\bf A}.
\]
In particular, the bilinear form $(\cdot , \cdot)$ on $\Uidot$ is non-degenerate. 
\end{thm}

\begin{proof}
The equality is trivial if ${\zeta_\imath} \neq \zeta_\imath'$.  
Now assume ${\zeta_\imath} = \zeta_\imath'$.
For $\lambda, \mu \gg0$ such that $\overline{\wb\lambda+\mu} ={\zeta_\imath}$, we have 
\begin{align*}
& \big( (b_1 \diamondsuit^\imath_{\zeta_\imath} b_2) (\eta^\bullet_\lambda \otimes \eta_\mu), 
(b'_1 \diamondsuit^\imath_{{\zeta_\imath}} b'_2) (\eta^\bullet_\lambda \otimes \eta_\mu)
\big)_{\la,\mu} 
\\
& =  \big( \big( (b_1 \diamondsuit_{\wb \lambda+\mu} b_2)( \eta^\bullet_\lambda \otimes \eta_\mu) \big)^\imath, 
\big( (b_1' \diamondsuit_{\wb \lambda+\mu} b_2')( \eta^\bullet_\lambda \otimes \eta_\mu) \big)^\imath
\big)_{\la,\mu} 
\\
&\equiv  \delta_{b_1, b'_1} \delta_{b_2, b'_2}, \quad \text{ mod } q^{-1} {\bf A}.
\end{align*}
The first equation above follows by Theorem~\ref{thm:iCBUi}, while the second one follows
by Proposition~\ref{prop:iCBtensorproduct} and \cite[26.3.1(c)]{Lu94}. 
Hence by taking $\lim\limits_{(\lambda, \mu) \to \infty}$ for the above identity and applying Lemma~\ref{lem:bilinearform}  (see also Definition~\ref{def:Uibilinear})
we conclude that
\begin{align*}
\big( b_1 \diamondsuit^\imath_{\zeta_\imath} b_2, b'_1 \diamondsuit^\imath_{{\zeta_\imath}} b'_2
\big) 
\equiv  \delta_{b_1, b'_1} \delta_{b_2, b'_2}, \quad \text{ mod } q^{-1} {\bf A}.
\end{align*}
This proves the theorem.
\end{proof}

The $\imath$-canonical basis $\Bdot^\imath$ admits the following characterization, whose proof is identical to the proof of \cite[Theorem~26.3.1]{Lu94} and hence will be skipped. 
\begin{thm}
  \label{thm:orth2}
	Let $\beta \in \Uidot$. Then $\beta \in \Bdot^\imath \cup (- \Bdot^\imath)$ if and only if $\beta$ satisfies the following three conditions: $\beta \in {}_\mA \Uidot$, $\ipsi (\beta) = \beta$, and $(\beta, \beta) \equiv 1$ mod $q^{-1} {\bf A}$.
\end{thm}

\begin{rem}
For type AIII/AIV with $\Iblack=\emptyset$, 
a geometric construction of the $\imath$-canonical basis of $\Uidot$ was given in \cite{LW15} (built on the earlier construction in \cite{BKLW}),
which is almost orthonormal with respect to some geometric bilinear form.
The identification between the algebraic constructions (in this paper) and the geometric constructions of $\imath$-canonical basis and bilinear from on $\Uidot$
will be addressed elsewhere. 

We further expect various positivity properties for $\imath$-canonical bases for some classes of QSP, similar to Lusztig's canonical bases. 
\end{rem}

\appendix

\section{Integrality of the intertwiners of real rank one}
 \label{sec:Upsilonrank1}

The goal of this appendix is to provide a proof of Theorem~\ref{thm:intUpsilon}(1) 
that the intertwiner $\Upsilon$ lies in (the completion of) the integral form ${}_\mA \U^+$ 
for all quantum symmetric pairs of real rank one; see Table~\ref{table:localSatake}. 

We shall first establish the integrality of $\Upsilon$ in type AIV and then AIII$_{11}$, which has the involution $\tau |_{\Iwhite} \neq 1$. 
These two types are easy and similar to the special case treated in \cite{BW13} (denoted by $\U^\jmath$ therein). 
The integrality of $\Upsilon$ for type AI$_1$ was essentially known in \cite[Lemma~4.6]{BW13}.  
Then we will establish some general properties of $\Upsilon$ for the remaining types with $\tau|_{\Iwhite}=1$. Ultimately it requires 
a tedious type-by-type analysis to complete the proof for all types with $\tau|_{\Iwhite}=1$.

\subsection{Type AIV of rank $n$}

We recall the Satake diagram of type AIV from Table~\ref{table:localSatake}:
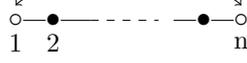
\[
\begin{tikzpicture}	[baseline=6]
		\node at (-0.5,0) {$\circ$};
		\draw[-] (-0.4,0) to (-0.1, 0);
		\node  at (0,0) {$\bullet$};
		\node at (2,0) {$\bullet$};
		\node at (2.5,0) {$\circ$};
		\draw[-] (0.1, 0) to (0.5,0);
		\draw[dashed] (0.5,0) to (1.4,0);
		\draw[-] (1.6,0)  to (1.9,0);
		\draw[-] (2.1,0) to (2.4,0);
		\draw[<->] (-0.5, 0.2) to[out=45, in=180] (0, 0.35) to (2, 0.35) to[out=0, in=135] (2.5, 0.2);
		\node at (-0.5,-.3) {\small 1};
		\node  at (0,-.3) {\small 2};
		\node at (2.5,-.3) {\small n};
	\end{tikzpicture}
\]

\begin{prop}
  \label{prop:AIV:integrality}
In type AIV, we have $\Upsilon_\mu \in {}_\mA \U^+$ for any $\mu \in \N[\I]$.
\end{prop}

\begin{proof}
Recall that $\ff_{1} = F_{1} + \vs_1 \T_{\wb} (E_{n}) \tK^{-1}_1$ and $\ff_{n} = F_{n} + \vs_n \T_{\wb} (E_{1}) \tK^{-1}_n$. Note that 
$$ F_{1} \T_{\wb} (E_{n}) \tK^{-1}_1= q_1^{-2}\T_{\wb} (E_{n}) \tK^{-1}_1 F_1,
\qquad F_{n} \T_{\wb} (E_{1}) \tK^{-1}_n= q_n^{-2}\T_{\wb} (E_{1}) \tK^{-1}_n F_n. 
$$
Introduce the divided powers $B_i^{(a)} = B_i^a /[a]!$, for all $i \in \I$ and $a \in \N$. 
Then we have 
\begin{equation}\label{eq:AIV:dividedpowers}
\ff_1^{(a)} = \sum_{s+t =a} q_1^{st}F^{(s)}_1 (\T_{\wb} (E_{n}) \tK^{-1}_1)^{(t)}, \qquad \ff_n^{(a)} = \sum_{s+t =a} q_n^{st}F^{(s)}_n (\T_{\wb} (E_{1}) \tK^{-1}_n)^{(t)}.
\end{equation}
So $\ff_i^{(a)}$ are $\ipsi$-invariant and integral (i.e. $\ff_i^{(a)} \in {}_\mA \U$), for all $i \in \I$. Now for any $\lambda \in X^+$ and $x \in {}_\mA L(\lambda)$, we can write 
\[
x =\sum c(a_1, \dots, a_k) \ff^{(a_1)}_{i_1} \cdots \ff^{(a_k)}_{i_k} \eta_{\lambda}, \quad \text{ for }  c(a_1, \dots, a_k) \in \mA.
\]
By Corollary~\ref{cor:ipsi-lambda} and using $\ipsi =\Upsilon \psi$ from \eqref{ibar},
we have $\Upsilon (x) =\ipsi  ( \psi(x) ) \in {}_\mA L(\lambda)$.    
Taking $x = \xi$ (the lowest weight vector) and $\lambda \gg 0$, we have $\Upsilon_\mu \in {}_\mA \U^+$ for any $\mu$. 
\end{proof}

\subsection{Type AIII$_{11}$}
Recall the Satake diagram of type AIII$_{11}$ of real rank one from Table~\ref{table:localSatake}:
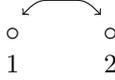
\[
\begin{tikzpicture}[baseline = 12, scale =1.3]
		\node at (-0.5,0.1) {$\circ$};
		\node at (0.5,0.1) {$\circ$};
		\draw[<->] (-0.4, 0.3) to[out=45, in=180] (-0.15, 0.45) to (0.15, 0.45) to[out=0, in=135] (0.4, 0.3);
		\node at (-0.5,-0.2) {\small 1};
		\node at (0.5,-0.2) {\small 2};
	\end{tikzpicture}
\]
Note that the underline Dynkin diagram is not irreducible.  
We have $\ff_i = F_i + \vs_i E_{j} \tK^{-1}_i$ for $1\le i \neq j \le 2$. Defining the divided powers $f_i^{(a)} = f_i^a /[a]!$ as usual, we have 
\[
\ff_i^{(a)} = \sum_{s+t =a} q_i^{st}F^{(s)}_i (E_{j} \tK^{-1}_i)^{(t)} \in {}_\mA \U.
\]

Now we are in a position to use (and choose to omit) the same argument as for Proposition~\ref{prop:AIV:integrality} to obtain the following. 
 
 \begin{prop}
 In type AIII$_{11}$, we have $\Upsilon_\mu \in {}_\mA \U^{+}$ for all $\mu$.
 \end{prop}
 
\begin{rem}
The integrality of the standard divided powers $\ff_i^{(a)} \in {}_\mA \U$ for $i \in \Iwhite$ in types AIV and AIII$_{11}$ distinguishes 
these two types from the remaining ones. 
\end{rem}

\subsection{Type AI$_1$}
  \label{subsec:AI}
Recall the Satake diagram of type AI$_{1}$ from Table~\ref{table:localSatake}:
\[
\begin{tikzpicture}[baseline=0]
		\node  at (0,0) {$\circ$};
		\node  at (0,-.3) {\small 1};
	\end{tikzpicture}
\]
Since there is only one element in $\I$, we shall drop the index $1$ and write 
$\ff = F + q^{-1} E {K}^{-1} + \kappa {K}^{-1}$. This is the only real rank one case when $\kappa$ can be non-zero. 
Set $\Upsilon = \sum_{c \ge 0} \Upsilon_c$ where $\Upsilon_c= \gamma_c E^{(c)}$.
Proposition~\ref{prop:A:integralform} below and its proof are adapted from \cite[Lemma~4.6]{BW13} (where $\kappa =1$).

\begin{prop}\label{prop:A:integralform}
In type AI$_1$, we have $\Upsilon_c \in {}_\mA \U^+$ for all $c \ge 0$.
\end{prop}

\begin{proof}
Recall from \eqref{kappa2} we have $\overline{\kappa} =\kappa$. Equation \eqref{eq:Upsilon} implies that 
\[
 (F + q^{-1} E {K}^{-1} + \kappa {K}^{-1}) \Upsilon = \Upsilon ( F + q E {K} + \kappa {K}), 
\]
which can be rewritten as 
\[
\gamma_{c+1} = -(q-q^{-1}) q^{-c} (q [c] \gamma_{c-1} + \kappa \gamma_c).
\]
It follows by induction on $c$ that $\gamma_c \in \mA$, since by definition we know $\gamma_0 =1$.
\end{proof}


\subsection{Generalities when $\tau|_{\Iwhite}  = 1$} 
\label{sec:tau1}

In this subsection, we assume that the Satake diagrams are real rank one of types with $\tau|_{\Iwhite} =1$ and $\Iblack\neq \emptyset$ 
(i.e., of types AII$_3$, BII, CII, DII, FII); see Table~\ref{table:localSatake}.  
In all these cases, we have the parameters $\kappa_i=0$, for all $i\in \Iwhite$. 
Let 
$$\Iwhite = \{ \bf i \}. 
$$
Then following Theorem~\ref{thm:Upsilon}, we have 
\begin{equation}
  \label{eq:Upc}
\Upsilon = \sum_{c \in \N} \Upsilon_{c},
\end{equation}
where $\Upsilon_c = \Upsilon_{c(\wb {\bf i} +{\bf i})}$ has weight $c(\wb {\bf i} +{\bf i})$. 
Note that \eqref{eq:Upsilon} implies that
$
\ipsi(\ff_{\bf i}) \Upsilon = \Upsilon \psi(\ff_{\bf i})
$,
that is,
\[
(F_{{\bf i}} + \vs_{\bf i} \T_{\wb}(E_{  {\bf i}}) \widetilde{K}^{-1}_{\bf i} ) \Upsilon = \Upsilon (F_{{\bf i}} + \vs^{-1}_{\bf i} \psi (\T_{\wb}(E_{  {\bf i}})) \widetilde{K}_{\bf i} ), \qquad F_j \Upsilon = \Upsilon F_j (j \in \Iblack).
\]
Using \cite[Proposition~3.16]{Lu94}, we have (for $c \ge 1$ and $j \in \Iblack$)
\begin{align}
r_{\bf i} (\Upsilon_c) &= -(q_{\bf i}-q_{\bf i}^{-1}) \Upsilon_{c-1} \cdot \vs_{\bf i}^{-1} \cdot \psi (\T_{\wb}(E_{  {\bf i}})) \label{eq:rank1Upsilon:1} \\
&= -(q_{\bf i}-q_{\bf i}^{-1})(-q_{\bf i})^{\langle {\bf i}, 2 \rho_{\bullet} \rangle}  \cdot \vs_{\bf i}^{-1} \cdot  \Upsilon_{c-1}  \cdot  \T^{-1}_{\wb}(E_{  {\bf i}}) , \notag \\
{}_{\bf i}r (\Upsilon_c) &= -(q_{\bf i}-q_{\bf i}^{-1})\vs_{\bf i} \cdot q_{\bf i}^{\langle {\bf i}, \wb   {\bf i} \rangle} \cdot \T_{\wb}(E_{  {\bf i}})  \Upsilon_{c-1},\label{eq:rank1Upsilon:2}\\
{}_j r (\Upsilon_c) &= r_j (\Upsilon_c) = 0.\label{eq:rank1Upsilon:3}
\end{align}

Recall the shorthand notation
$w^{\bullet} = w_{\bullet} w_0.$ 
Let $\ell (w^{\bullet}) = k$. Note $w^{\bullet}$ is of the form
$$
w^{\bullet} = s_{i_1}   s_{i_2} \cdots s_{i_{k-1}} s_{i_k}, 
\quad i_1 = i_k ={\bf i} \in \Iwhite.
$$
Introduce a shorthand notation $\T_{i_1i_2 \cdots i_{k-1}} =\T_{i_1} \T_{i_2} \cdots \T_{i_{k-1}}$. 
Applying Proposition~\ref{prop:wcirc}, we have 
\begin{equation}\label{eq:Upsilon:c}
\Upsilon_c = \sum \gamma_c(c_1, c_2, \dots, c_k) E^{(c_1)}_{\bf i} \T_{\bf i} (E^{(c_2)}_{i_2}) \cdots \Big( \T_{i_1i_2 \cdots i_{k-1}}  (E^{(c_k)}_{{\bf i}}) \Big). 
\end{equation}
We adopt the convention that $E^{(a)}_{j} = 0$ for any $j \in \I$, with $a < 0$, and $\gamma_{c}(c_1, c_2, \dots, c_k) = 0$ unless all $c_j \ge 0$. We shall write $\gamma_c = \gamma_c(c_1, c_2, \dots, c_k)$ when there is no need to specify each component. 

The following lemma shall be used extensively in this section. 

\begin{lem} \cite[Proposition~8.20]{Jan} \label{lem:Jan820}
For any $w \in W$, if  $w (i') = j' \in X$ for $i, j \in \I$, then we have $ \T_w (E_i) = E_j$.
\end{lem}

\begin{lem}
We have 
\[
\T^{-1}_{\wb} (E_{  {\bf i}}) = \T_{i_1i_2 \cdots i_{k-1}}  (E_{  {\bf i}}).
\]
\end{lem}

\begin{proof}
We have ${\wb}  s_{i_1} s_{i_2} \cdots s_{i_{k-1}} ({{\bf i}'}) = -w_0 ({\bf i}') ={\bf i}'$. The lemma follows by Lemma~\ref{lem:Jan820}.
\end{proof}

\begin{lem}\label{lem:c1neq0}
Let $c \ge 1$. If we have $\gamma_{c-1} (c_1, c_2, \dots, c_k) \in \mA$ for all $(c_1, c_2, \dots, c_k)$, 
then $(1-q_{\bf i}^{-2})^{-c_1}\gamma_{c} (c_1, c_2, \dots, c_k) \in \mA $ for all $(c_1, c_2, \dots, c_k)$ with $c_1 \ge 1$.
\end{lem}

\begin{proof}
Using Lemma~\ref{lem:Jan820} and \eqref{eq:rank1Upsilon:1}, we have 
\begin{align*}
 r_{\bf i} (\Upsilon_c) 
 =  &\sum \gamma_c(c_1, c_2, \dots, c_k) q_{\bf i}^{c_1-1} q_{\bf i}^{\langle {\bf i}, c(\wb {\bf i} + {\bf i}) - c_1 {\bf i} \rangle}E^{(c_1-1)}_{\bf i} \T_{\bf i} (E^{(c_2)}_{i_2}) \cdots \Big( \T_{i_1i_2 \cdots i_{k-1}}  (E^{(c_k)}_{{\bf i}}) \Big) \\
 = &-(q_{\bf i}-q_{\bf i}^{-1})(-q_{\bf i})^{\langle {\bf i}, 2 \rho_{\bullet} \rangle}  \cdot \vs_{\bf i}^{-1} \cdot  \Upsilon_{c-1}  \cdot  \T^{-1}_{\wb}(E_{  {\bf i}}) \\
 = &-(q_{\bf i}-q_{\bf i}^{-1})(-q_{\bf i})^{\langle {\bf i}, 2 \rho_{\bullet} \rangle}  \vs_{\bf i}^{-1}  \cdot \\
 &\sum \gamma_{c-1}(c_1, c_2, \dots, c_k) [c_k+1]_{\bf i} E^{(c_1)}_{\bf i} \T_{\bf i} (E^{(c_2)}_{i_2}) \cdots \Big( \T_{i_1i_2 \cdots i_{k-1}}  (E^{(c_k+1)}_{{\bf i}}) \Big).
\end{align*}

Therefore we have for weight reason that, for $c_1\ge 1$, 
\begin{align*}
& \gamma_c(c_1, c_2, \dots, c_k) q_{\bf i}^{c_1-1} q_{\bf i}^{\langle {\bf i}, c(\wb {\bf i} + {\bf i}) - c_1 {\bf i} \rangle} \\
= &-(q_{\bf i}-q_{\bf i}^{-1})(-q_{\bf i})^{\langle {\bf i}, 2 \rho_{\bullet} \rangle}  \cdot \vs_{\bf i}^{-1} \gamma_{c-1}(c_1-1, c_2, \dots, c_k-1) [c_k]_{\bf i},
\end{align*}
that is,
\begin{equation}\label{eq:Upsilon:general:1}
\begin{split}
& \gamma_c(c_1, c_2, \dots, c_k)   \\
= &-(q_{\bf i}-q_{\bf i}^{-1})q_{\bf i}^{- \langle {\bf i}, c(\wb {\bf i} + {\bf i}) - c_1 {\bf i} \rangle} q_{\bf i}^{1-c_1}(-q_{\bf i})^{\langle {\bf i}, 2 \rho_{\bullet} \rangle}  \vs_{\bf i}^{-1} \gamma_{c-1}(c_1-1, c_2, \dots, c_k-1) [c_k]_{\bf i}.
\end{split}
\end{equation}
It follows by an induction on $c_1$
that $(1-q_{\bf i}^{-2})^{-c_1}  \gamma_c(c_1, c_2, \dots, c_k)  \in \mA$ as long as $c_1 \ge 1$.
\end{proof}

\begin{rem}
We proved a stronger result  than just  $\gamma_c(c_1, c_2, \dots, c_k)  \in \mA$ under our assumption. The importance shall be clear later in this section.
\end{rem}
%

Our strategy is to prove that $ \Upsilon_c \in {}_\mA \U^+$ by induction on $c$. 
The base case at $c=0$ is always true since we have $\Upsilon_0 = 1$. 
For the induction step we shall compute the precise actions of $r_j$ on $\Upsilon_c$ for $j \in \I$ case by case. 
Then thanks to Lemma~\ref{lem:c1neq0}, it suffices to prove that 
\begin{equation}  \label{eq:Uck}
\gamma_c(c_1, c_2, \dots, c_k) \in \mA \text{ for all } c_i, 
 \text{ if } (1-q_i^{-2})^{-c_1}\gamma_{c} (c_1, c_2, \dots, c_k) \in \mA \text{ when } c_1 \ge 1.
\end{equation}
This is what we shall do later in this section case by case.

To facilitate the case-by-case analysis below, let us introduce some shorthand notations. 
For a sequence $i_1 i_2 \dots i_k$ with $i_j \in \I$, we shall often use the shorthand notation 
\[
\T_{i_1i_2\cdots i_k} = \T_{i_1}  \T_{i_2} \cdots \T_{i_k}.
\]
In concrete cases below (with labelings as in Table~\ref{table:localSatake}), the sequence $i_1  \dots i_k$
is naturally partitioned into increasing and decreasing subsequences, 
and we shall insert indices $i_l$ to indicate the local maxima/minima of the sequence. For example,  $\T_{i_1 \cdots i_l \cdots i_k}$ means 
the subsequences $i_1 \cdots i_l$ and $i_l \cdots i_k$ are monotone, and $\T_{i_1 \cdots i_l \cdots i_m \cdots i_k}$ means 
the subsequences $i_1 \cdots i_l$, $i_l \cdots i_m$,  and $i_m \cdots i_k$ are monotone, 
and so on.  Here is a concrete example which occurs in Type CII below:
the shorthand ${2 \cdots n \cdots 1 \cdots n \cdots k}$, for some $1\le k \le n$, means ${2 \ 3 \cdots n-1 \ n\ n-1 \cdots 2\ 1 2 \cdots n-1\ n \ n-1 \cdots k}$.

For $x, y \in \U^+$, we write 
\[
[x,y]_{q^{-1}_i} = xy - q_i^{-1}yx. 
\]

Since the ring $\mA$ is invariant under multiplication by $q^a$ for any $a \in \Z$, 
we shall often use the notation $q^*$ to indicate $q$-powers without computing the precise exponent when it is irrelevant. 

\subsection{Type AII of rank 3}
In this subsection we assume the quantum symmetric pair $(\U, \Ui)$ is of type AII. We label the Satake diagram as follows:
\[
 \begin{tikzpicture}[baseline=0]
		\node at (0,0) {$\bullet$};
		\draw (0.1, 0) to (0.4,0);
		\node  at (0.5,0) {$\circ$};
		\draw (0.6, 0) to (0.9,0);
		\node at (1,0) {$\bullet$};
		\node at (0,-.3) {\small 1};
		\node  at (0.5,-.3) {\small 2};
		\node at (1,-.3) {\small 3};
\end{tikzpicture}	
\]
Take the reduced expression $w^{\bullet}  = s_2 s_1 s_3 s_2$. 
Then we have 
\[
\Upsilon_{c} = \sum \gamma(c_1, c_2 , c_3, c_4) E^{(c_1)}_2 \cdot \T_2 (E^{(c_2)}_1) \cdot \T_2 (E^{(c_3)}_3) \cdot \T_{213} (E^{(c_4)}_2),
\]
where  $c = \hf (c_1 + c_2 + c_3 +c_4)$. (For weight reason $\Upsilon_c = 0$ if $c$ is not an integer.)

We then compute the actions of $r_{1}$ and $r_3$ on those root vectors. Note that we have 
\begin{align*}
\T_2 (E_1) & = E_2 E_1 - q^{-1} E_1 E_2;\\
\T_2 (E_3) & = E_2 E_3 - q^{-1} E_3 E_2;\\
\T_{213}  (E_2)& = (E_2E_1- q^{-1} E_1 E_2)E_3 - q^{-1} E_3 (E_2 E_1 - q^{-1}E_1 E_2).
\end{align*}

The following lemmas follow from straightforward computation.

\begin{lem}
We have 
\begin{align*}
E_1 \cdot \T_2(E_1) = q^{-1} \T_2(E_1) \cdot E_1 &\quad \text{ and } \quad E_3\cdot  \T_2(E_3) = q^{-1} \T_2(E_3)\cdot  E_3,\\
E_1 \cdot \T_1(E_2) = q \T_1(E_2) \cdot E_1 &\quad \text{ and } \quad E_3\cdot  \T_3(E_2) =  q \T_3(E_2)\cdot  E_3,
\end{align*}
\end{lem}

\begin{lem}
We have  
	\begin{enumerate}
		\item	$r_1 (\T_2 (E^{(c_2)}_1)) = (1-q^{-2}) E_2 \cdot\T_2 (E^{(c_2-1)}_1)$;
		\item	$r_3 (\T_2 (E^{(c_2)}_1)) = 0$;
		\item	 $r_1 (\T_2 (E^{(c_3)}_3)) = 0$;
		\item	$r_3 (\T_2 (E^{(c_3)}_3)) = (1-q^{-2}) E_2  \cdot\T_2 (E^{(c_3-1)}_1)$;
		\item	$r_1 (\T_{213}  (E^{(c_4)}_2)) = (1-q^{-2}) \T_2 (E _3) \cdot \T_{213}  (E^{(c_4-1)}_2)$;
		\item	$r_3 (\T_{213}  (E^{(c_4)}_2)) =  (1-q^{-2}) \T_2 (E _1) \cdot \T_{213}  (E^{(c_4-1)}_2)$.
		\end{enumerate}
\end{lem}

\begin{prop}\label{prop:A:integralform}
In type AII, we have $\Upsilon_c \in {}_\mA \U^+$ for all $c \ge 0$.
\end{prop}

\begin{proof}
It suffices to prove the statement \eqref{eq:Uck} by the general discussion in \S \ref{sec:tau1}. Let us assume
\[
(1-q_i^{-2})^{-c_1}\gamma_{c} (c_1, c_2, \dots, c_k) \in \mA \text{ when } c_1 \ge 1.
\]

Since $r_1 (\Upsilon_c) = r_3 (\Upsilon_c) = 0$ by \eqref{eq:rank1Upsilon:3}, we have
\begin{align*}
0= & \frac1{1-q^{-2}}  r_1 (\Upsilon_c) \\
= & \sum \gamma_c (c_1, c_2, c_3, c_4)  [c_3+1] E^{(c_1)}_2 \cdot \T_2 (E^{(c_2)}_1) \cdot \T_2 (E^{(c_3+1)}_3) \cdot \T_{213}  (E^{(c_4-1)}_2)  \\
&  + \sum \gamma_c (c_1, c_2, c_3, c_4)  [c_1+1] q^{c_4 -c_3}E^{(c_1+1)}_2 \cdot \T_2 (E^{(c_2-1)}_1) \cdot \T_2 (E^{(c_3)}_3) \cdot \T_{213}  (E^{(c_4)}_2).
\end{align*}
It follows that
\[
\gamma_c (c_1, c_2+1, c_3+1, c_4)  [c_1+1] q^{c_4 -c_3}   =  -  \gamma_c (c_1+1, c_2, c_3, c_4+1)  [c_3+1] .
\]
Therefore we have
\[
\gamma_c (0, c_2+1, c_3+1, c_4) = -q^{c_3-c_4} \gamma_c  (1, c_2, c_3, c_4+1)  [c_3+1] \in \mA.
\]
It follows that  $\gamma_c (c_1, c_2, c_3, c_4) \in \mA$ if  $c_1, c_2$ are not all zero. On the other hand, we have $\gamma_c (0, 0, c_3, c_4) =0$ for weight reason. The proposition follows.
\end{proof}

\begin{cor}
We have $c_1 = c_4$ and $c_2 =c_3$ whenever $\gamma_c(c_1, c_2, c_3 ,c_4) \neq 0$.
\end{cor}

\subsection{Type DII of rank n $\ge 4$}
In this subsection we assume the quantum symmetric pair $(\U, \Ui)$ is of type DII.
We label the Satake diagram is as follows:
$$		
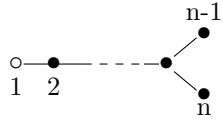
\begin{tikzpicture}[baseline=0]
		\node at (1,0) {$\circ$};
		\node at (1.5,0) {$\bullet$};
		\draw[-] (1.1,0)  to (1.4,0);
		\draw[-] (1.4,0) to (1.9, 0);
		\draw[dashed] (1.9,0) to (2.7,0);
		\draw[-] (2.7,0) to (2.9, 0);
		\node at (3,0) {$\bullet$};
		\node at (3.5, 0.4) {$\bullet$};
		\node at (3.5, -0.4) {$\bullet$};
		\draw (3.1, 0.1) to (3.4, 0.35);
		\draw (3.1, -0.1) to (3.4, -0.35);
		\node at (1,-.3) {\small 1};
		\node at (1.5,-.3) {\small 2};
		\node at (3.5, 0.7) {\small n-1};
		\node at (3.5, -0.6) {\small n};
	\end{tikzpicture}		
$$

We take the reduced expression $w^{\bullet} = s_1 \cdot s_2 \cdots s_{n-2} \cdot s_{n-1} \cdot s_n \cdot s_{n-2} \cdots s_1$. Therefore we can write $\Upsilon_c$ as 
\[
 \Upsilon_c =  \sum \gamma_c(c_1, \dots, c_{2n-2}) \, E^{(c_{1})}_{1} \cdot \T_{1}(E^{(c_{2})}_{2})\cdots (\T_{1\cdots 2} (E^{(c_{2n-2})}_{1})).
\]
For weight reason, we must have  $\gamma_c(c_1, \dots, c_{2n-2}) =0$ unless $\sum^{2n-2}_{i=1} c_i = c$.

\begin{lem}\label{lem:D:rkfirsthalf}
For $k \neq 1$, we have 
\[
{ r_{k}} (\T_{1 \cdots i} (E^{(a)}_{i+1}) ) = 
\begin{cases}
(1-q^{-2}) \, \T_{1 \cdots (i-1)}(E_i)\cdot (\T_{1 \cdots i} (E^{(a-1)}_{i+1}) ), &\text{if } k =i+1;\\
0, &\text{if } k \neq i+1.
\end{cases}
\]
\end{lem}

\begin{proof}Let us first assume that $i \le n-2$. The proof is divided into three cases:
	\begin{enumerate} 
		\item	If $k \ge  i+2$, it is clear that ${ r_{k}} (\T_{1 \cdots i} (E_{i+1}) ) = 0$. 
		\item	For $k \le  i$, we have 
				\begin{align*}
				{r_k } \big(\T_{1 \cdots (k-1)}   \T_{k} \T_{(k+1) \cdots i} (E_{i+1}) \big) 
				= & \,{r_k } \big(\T_{1 \cdots (k-1)} [E_k, \T_{(k+1) \cdots i} (E_{i+1})]_{q^{-1}} \big)\\
				= & \,{r_k } \big( [ \T_{1 \cdots (k-1)} (E_k), \T_{(k+1) \cdots i}(E_{i+1})]_{q^{-1}} \big)\\
				= &\, q^{-1} {r_k } \big( \T_{1 \cdots (k-1)}(E_k) \big) \cdot   \T_{(k+1) \cdots i} (E_{i+1}) \\
				&\, - q^{-1} \, \T_{(k+1) \cdots i} (E_{i+1}) \cdot  {r_k } \big(\T_{1 \cdots (k-1)} (E_k) \big)
				= 0.
				\end{align*}
		\item 	For $k = i+1$, we have  (for $ i\leq n-2$)
				\[
				\T_{1\cdots i} (E_{i+1}) = \T_{1\cdots (i-1)} ([E_i, E_{i+1}]_{q^{-1}}) =   [\T_{1\cdots (i-1)}(E_i) , E_{i+1}]_{q^{-1}}.
				\]
				It follows that
				\[
				{ r_{i+1}} (\T_{1\cdots i} (E_{i+1}) ) =  (1-q^{-2}) \T_{1\cdots (i-1)}(E_i). 
				\] 
				Since $\T_i(E_{i+1}) \cdot E_i = q^{-1} E_i \cdot\T_{i} (E_{i+1})$, we have 
				\[
				{ r_{i+1}} \big(\T_{1\cdots i} (E^{(a)}_{i+1}) \big) =  (1-q^{-2}) \, \T_{1\cdots (i-1)}(E_i)\cdot \T_{1\cdots i}  (E^{(a-1)}_{i+1}).
				\]
	\end{enumerate}
The case $i = n-1$ is entirely similar to the case $i=n-2$, since $\T_{n-1} (E_{n}) = E_n$. The lemma follows.
\end{proof}

\begin{lem}\label{lem:D:rkright}
For $i \neq n, n-1$ and $k \neq 1$, we have 
\begin{align*}
{ r _{k}} \big(\T_{1 \cdots n \cdots i}  (E^{(a)}_{i-1}) \big) 
=&\, 
\begin{cases}
(1 -q^{-2} )\T_{1 \cdots n \cdots (i+1)} (E_{i}) \cdot \T_{1 \cdots n \cdots i} (E^{(a-1)}_{i-1}), &\text{ if } k = i \neq n-2;\\
(1 -q^{-2} )\T_{1 \cdots n}   (E_{n-2}) \cdot \T_{1 \cdots n(n-2)} (E^{(a-1)}_{n-3}), &\text{ if } k = i = n-2;\\
0, &\text{ if } k \neq i.
\end{cases}
\end{align*}
\end{lem}

\begin{proof}The computation is divided into six cases.
	\begin{enumerate}
		\item 	For $k \le i-2$, we have 
		\[
		\T_{1 \cdots n \cdots i}   (E_{i-1}) = [\T_{1 \cdots (k-1)} (E_{k}), \T_{(k+1) \cdots n   \cdots i} (E_{i-1}) ]_{q^{-1}}.\] Therefore we have 
		\begin{align*}
		r_k \Big( \T_{1 \cdots n \cdots i} (E_{i-1}) \Big) 
		= &r_{k} \Big( [\T_{1 \cdots (k-1)} (E_{k}), \T_{(k+1) \cdots n \cdots i}(E_{i-1}) ]_{q^{-1}} \Big) \\
		= & q^{-1} r_k (\T_{1 \cdots (k-1)}  (E_{k})) \cdot \Big(\T_{(k+1) \cdots n \cdots i}(E_{i-1})\Big) \\
		&- q^{-1}  \Big( \T_{(k+1) \cdots n \cdots i}(E_{i-1})\Big) \cdot  r_k (\T_{1 \cdots (k-1)} (E_{k}))
		= 0,
		\end{align*}
		since $r_k (\T_{1 \cdots (k-1)} (E_{k}))$ and $\T_{(k+1) \cdots n \cdots i}(E_{i-1})$ commute.

		\item	For $k=i-1$, we consider 
				\begin{equation}\label{eq:DII:1}
				\begin{split}
				 \T_{1 \cdots n \cdots i} (E_{i-1}) 
				= &\,\T_{1 \cdots n \cdots (i+1)} ([E_{i},E_{i-1}]_{q^{-1}})\\
				= &\, [\T_{1 \cdots n \cdots (i+1)} (E_{i}) , \T_{1 \cdots i} (E_{i-1}) ]_{q^{-1}} 
				=   [\T_{1 \cdots n \cdots (i+1)} (E_{i}) , E_i ]_{q^{-1}}.
				\end{split}
				\end{equation}
				Then since $i-1 \le i+1 -2$, by Case~ (1) we have 
				\begin{align*}
				{r_{i-1}} & \big(\T_{1 \cdots n \cdots i} (E_{i-1}) \big) \\
				&=\, q^{-1} {r_{i-1}} (\T_{1 \cdots n \cdots (i+1)} (E_{i})) \cdot E_{i} - q^{-1} E_{i}  \cdot  {r_{i-1}} (\T_{1 \cdots n \cdots (i+1)} (E_{i}))
				= 0.
				\end{align*}
		
		\item	For $k=i$, following \eqref{eq:DII:1} we have 
				\begin{align*}
				{r_{i}} (\T_{1 \cdots n \cdots i} (E_{i-1})) 
				=&\,  \T_{1 \cdots n \cdots (i+1)} (E_{i})  - q^{-2}     \T_{1 \cdots n \cdots (i+1)}  (E_{i}) \\
				=&\, (1-q^{-2}) \T_{1 \cdots n \cdots (i+1)}  (E_{i}).
				\end{align*}
				More generally we have
				\begin{align*}
				{r_{i}} \big(\T_{1 \cdots n \cdots i} (E^{(a)}_{i-1}) \big) 
				\,= & \,(1-q^{-2}) \T_{1 \cdots n \cdots (i+1)} (E_{i}) \cdot \T_{1 \cdots n \cdots i}(E^{(a-1)}_{i-1}).
				\end{align*}

		\item	For $n-3 \ge k \ge i+1$, we consider 
				\begin{align*}
				\T_{1 \cdots n \cdots i} (E_{i-1})
				= &\, \T_{1 \cdots n \cdots (k+1)} \big( [E_{k}, \T_{(k-1) \cdots i}(E_{i-1}) ]_{q^{-1}} \big)\\
				= &\, [\T_{1 \cdots n \cdots (k+1)}  (E_{k}), \T_{k \cdots (i+1)} (E_{i})]_{q^{-1}}.
				\end{align*}
				Note that $ { r_k} (\T_{k \cdots (i+1)}(E_{i})) = 0$ unless $k=i$. Therefore by Case~ (2) we have 
				\[
				r_{k} (\T_{1 \cdots n \cdots i} (E_{i-1})) =0.
				\]
		
		\item	For $k =n-2$, we consider that
				\begin{align*}
				\T_{1 \cdots n \cdots i} (E_{i-1})
				= &\, \T_{1 \cdots n } \big( [E_{k}, \T_{(k-1) \cdots i}(E_{i-1}) ]_{q^{-1}} \big)
				=  [\T_{1 \cdots n }  (E_{k}), \T_{k \cdots (i+1)} (E_{i})]_{q^{-1}}.
				\end{align*}
				Note that $ { r_k} \big(\T_{k \cdots (i+1)}(E_{i}) \big) = 0$ unless $k=i$. Therefore by Case~ (2) we have 
				\[
				r_{k} (\T_{1 \cdots n \cdots i} (E_{i-1})) =0.
				\]

		\item For $k=n-1$ (the case $k=n$ is similar), we have (true for $i = n-2$ as well)
				\begin{align*}
				\T_{1 \cdots n \cdots i}(E_{i-1})
				= &\, \T_{1 \cdots (n-1)} \Big( [E_{n}, \T_{(n-2) \cdots i} (E_{i-1}) ]_{q^{-1}} \Big)\\
				= &\, [ \T_{1 \cdots (n-1)} (E_{n}), \T_{(n-1) \cdots (i+1)} (E_{i})]_{q^{-1}}.
				\end{align*}
				Note that since by Lemma~\ref{lem:D:rkfirsthalf}
				\[
				r_{n-1} \Big(\T_{1 \cdots (n-1)}  (E_{n}) \Big) = r_{n-1} \Big(   \T_{(n-1) \cdots (i+1)}(E_{i}) \Big) = 0,
				\]
				 we have 
				\[
				r_{n-1} (\T_{1 \cdots n \cdots i} (E_{i-1})) =0.
				\]
		\end{enumerate}
		This completes the proof. 
\end{proof}

\begin{rem}
The computation for $(1)$-$(4)$ in the proof of Lemma~\ref{lem:D:rkright} is essentially a type $A$ computation, 
and will appear very often for the other cases as well.  
\end{rem}

\begin{lem}\label{lem:DII:3}
For $k \neq 1$, we have 
\begin{align*}
 { r _{k}} (\T_{1\cdots n}(E^{(a)}_{n-2})) = &
\begin{cases}
(1 -q^{-2} ) \T_{1\cdots (n-1)}(E_{n}) \cdot \T_{1\cdots n}(E^{(a-1)}_{n-2}), &\text{if } k = n-1;\\
(1 -q^{-2} ) \T_{1 \cdots(n-2) n}(E_{n-1}) \cdot \T_{1\cdots n}(E^{(a-1)}_{n-2}), &\text{if } k = n;\\
0, &\text{otherwise. }
\end{cases}
\end{align*}
\end{lem}
\begin{proof}
Note that 
\begin{align*}
& \T_{1\cdots n} (E_{n-2})\\
= & \T_{1\cdots (n-3)} ( q^{-2} E_{n} E_{n-1} E_{n-2} -q^{-1}E_{n} E_{n-2} E_{n-1} - q^{-1} E_{n-1} E_{n-2} E_{n} + E_{n-2} E_{n-1} E_{n}).
\end{align*}
Therefore  we have $r_k ( \T_{1\cdots n}  (E_{n-2})) = 0$ for $k < n-2$, thanks to Lemma~\ref{lem:D:rkfirsthalf}. 
It is easy to see $r_{n-2} ( \T_{1\cdots n}  (E_{n-2})) = 0$ as well by a direct computation using Lemma~\ref{lem:D:rkfirsthalf}. 

On the other hand, we have
\begin{align*}
 { r_{n-1}} ( \T_{1\cdots n} (E_{n-2}))
=&\, q^{-3} E_{n} \cdot   \T_{1\cdots (n-3)}  (E_{n-2}) -q^{-1}E_{n} \cdot  \T_{1\cdots (n-3)} (E_{n-2}) \\
&\, - q^{-2}  \T_{1\cdots (n-3)} (E_{n-2})  \cdot E_{n} + \T_{1\cdots (n-3)}(E_{n-2})  \cdot E_{n} \\
= &\, -q^{-2}  \T_{1\cdots (n-2)}(E_{n})  + \T_{1\cdots (n-2)} (E_{n})\\
= &\, (1-q^{-2}) \T_{1\cdots (n-1)}  (E_{n}).
\end{align*}
Now since $ E_n \cdot \T_n (E_{n-2}) = q \T_n (E_{n-2}) \cdot E_n$, we have 
\[
 { r _{n-1}} ( \T_{1\cdots n} (E^{(a)}_{n-2}))  = (1 -q^{-2} )  \T_{1\cdots (n-1)} (E_{n}) \cdot  \T_{1\cdots n} (E^{(a-1)}_{n-2}).
\]
The computation of ${ r_{n}} ( \T_{1\cdots n}  (E^{(a)}_{n-2}))$ is entirely similar. The lemma follows.
\end{proof}

\begin{prop}\label{prop:D:integralform}
For quantum symmetric pairs of type DII of rank n $\ge 4$, we have 
$\Upsilon_c \in {_\mathcal{A}\U^+}$ for all $c \ge 0$.
\end{prop}

\begin{proof}
Recall by the general discussion in \S \ref{sec:tau1} 
it suffices to prove the following statement (which implies \eqref{eq:Uck}): 
\[
\gamma_{c}(c_1,\dots,c_{2n-2}) \in \mA \text{ for all } c_i, \text{ if }\gamma_{c}(c_1,\dots,c_{2n-2}) \in \mA \text{ when } c_1 >0.
\]
 
We compare the coefficient of the following terms in the identity $r_{k}(\Upsilon_c)=0$:
\[
E^{(c_{1})}_{1} \cdot \T_{1}(E^{(c_{2})}_{2})\cdots (\T_{1\cdots 2} (E^{(c_{2n-2})}_{1})) \text{ with } c_{k-1}=1, c_j=0 \text{ for } j<k-1.
\]
We obtain that
\[
(1-q^{-2}) \gamma_c (0,\dots, c_{k-1}=0, c_k,\dots) \in  (1-q^{-2})\sum \gamma_c (\dots, c_{k-1}=1, c_k-1,\dots)\cdot \mA.
\]
Therefore thanks to Lemma~\ref{lem:c1neq0} for the base case, we have inductively:
\begin{equation}\label{eq:D:int:1}
\gamma_c (c_1,\dots) \in \mA, \text{ if } c_{k} >0.
\end{equation}
The proposition follows.
\end{proof}


\subsection{Other types}
The proof of part (1) of Theorem~\ref{thm:intUpsilon} for QSP of type BII of rank  $n\ge 2$, type CII  of rank $n\ge 3$, and type FII follows from entirely similar computation as type DII of rank $n \ge 4$. The precise details can be found in the (longer) appendix of the arXiv Version~ 1 of this paper, and shall be omitted here.

\begin{table}[!h]
\caption{Satake diagrams of irreducible symmetric pairs}
\label{table:Satake}
\begin{tabular}{| c | c || c | c |}
\hline
{AI} & 
$	\begin{tikzpicture}[baseline=0]
		\node  at (0,0) {$\circ$};
		\node at (2,0) {$\circ$};
		\node at (0.5, 0) {$\circ$};
		\draw[-] (0.1, 0) to (0.4,0);
		\draw[dashed] (0.6,0) to (1.5,0);
		\draw[-] (1.5,0)  to (1.9,0);
	\end{tikzpicture}
$ 
&
\begin{tikzpicture}[baseline=0]
\node at (0, -0.4) {DIII};
\end{tikzpicture}
&
$	\begin{tikzpicture}[baseline=0]
		\node at (0,0) {$\bullet$};
		\draw (0.1, 0) to (0.4,0);
		\node  at (0.5,0) {$\circ$};
		\draw (0.6, 0) to (0.9,0);
		\node at (1,0) {$\bullet$};
		\draw[-] (1.1,0)  to (1.4,0);
		\draw[dashed] (1.4, 0) to (2.1, 0);
		\draw[-] (2.1,0) to (2.4,0);
		\node at (2.5, 0) {$\circ$};
		\draw[-] (2.6,0.1) to (2.9, 0.35);
		\node at (3, 0.4) {$\bullet$};
		\draw[-] (2.6, - 0.1) to (2.9, -0.35);
		\node at (3,-0.4) {$\circ$};
	\end{tikzpicture}	
$

\\
\cline{1-2} \cline{4-4} 
AII & 
$	\begin{tikzpicture}[baseline=0]
		\node at (-0.5,0) {$\bullet$};
		\draw[-] (-0.5,0) to (-0.1, 0);
		\node  at (0,0) {$\circ$};
		\node at (0.5,0) {$\bullet$};
		\node at (2,0) {$\circ$};
		\node at (2.5,0) {$\bullet$};
		\draw[-] (0.1, 0) to (0.5,0);
		\draw[-] (0.6,0) to (0.8,0);
		\draw[dashed] (0.8,0) to (1.6,0);
		\draw[-] (1.6,0)  to (1.9,0);
		\draw[-] (2.1,0) to (2.4,0);
	\end{tikzpicture}
$ &

& 
$	\begin{tikzpicture}[baseline=0]
		\node at (0,0) {$\bullet$};
		\draw (0.1, 0) to (0.4,0);
		\node  at (0.5,0) {$\circ$};
		\draw (0.6, 0) to (0.9,0);
		\node at (1,0) {$\bullet$};
		\draw[-] (1.1,0)  to (1.4,0);
		\draw[dashed] (1.4, 0) to (2.1, 0);
		\draw[-] (2.1,0) to (2.4,0);
		\node at (2.5, 0) {$\circ$};
		\draw[-] (2.6,0) to (2.9, 0);
		\node at (3,0) {$\bullet$};
		\draw[-] (3.1,0.1) to (3.4, 0.35);
		\node at (3.5,0.4) {$\circ$};
		\draw[<->] (3.6, 0.35) to[out=-45, in=45] (3.6, -0.35);
		\draw[-] (3.1, -0.1) to (3.4, -0.35);
		\node at (3.5,-0.4) {$\circ$};
	\end{tikzpicture}	
$

\\
\hline
\begin{tikzpicture}[baseline=0]
\node at (0, -2) {AIII};
\end{tikzpicture}
& 
$	\begin{tikzpicture}[baseline=0]
		\node at (-0.5,0) {$\circ$};
		\draw[-] (-0.4,0) to (-0.1, 0);
		\node  at (0,0) {$\circ$};
		\node at (2,0) {$\circ$};
		\draw[-] (0.1, 0) to (0.5,0);
		\draw[dashed] (0.5,0) to (1.4,0);
		\draw[-] (1.6,0)  to (1.9,0);
		\draw[-] (2.1,-0.05) to (2.4, -0.3);
		\node at (2.5, -0.4) {$\bullet$};
		\draw[-] (2.5, -0.5) to (2.5, -0.7);
		\draw[dashed] (2.5, -0.7) to (2.5,-1.6);
		\draw[-] (2.5,-1.6) to (2.5, -2);
		\node at (2.5, -2.1) {$\bullet$};
		\node at (2, -2.5) {$\circ$};
		\draw (2.4, -2.2) to (2.1, -2.45);
			\node at (-0.5,-2.5) {$\circ$};
		\draw[-] (-0.4,-2.5) to (-0.1, -2.5);
		\node  at (0,-2.5) {$\circ$};
		\draw[-] (0.1, -2.5) to (0.5, -2.5);
		\draw[dashed] (0.5, -2.5) to (1.4, -2.5);
		\draw[-] (1.6, -2.5)  to (1.9, -2.5);
		\draw[<->] (-0.6, -0.1) to[out=215, in=90] (-0.8, -0.5) to (-0.8, -2) to[out=-90,in=135] (-0.6, -2.4);
		\draw[<->] (-0.1, -0.1) to[out=215, in=90] (-0.3, -0.5) to (-0.3, -2) to[out=-90,in=135] (-0.1, -2.4);
		\draw[<->] (1.9, -0.1) to[out=215, in=90] (1.7, -0.5) to (1.7, -2) to[out=-90,in=135] (1.9, -2.4);
	\end{tikzpicture}
$
&  \begin{tikzpicture}[baseline=0]
\node at (0, -1) {EI};
\end{tikzpicture}
&
$
	\begin{tikzpicture}[baseline = 20, scale=1.5]
		\node at (-1,0) {$\circ$};
		\draw (-0.9,0) to (-0.6,0);
		\node at (-0.5,0) {$\circ$};
		\draw (-0.4,0) to (-0.1,0);
		\node at (0,0) {$\circ$};
		\draw (0.1,0) to (0.4,0);
		\node at (0.5,0) {$\circ$};
		\draw (0.6,0) to (0.9,0);
		\node at (1,0) {$\circ$};
		\draw (0,-0.1) to (0,-0.4);
		\node at (0,-0.5) {$\circ$};
	\end{tikzpicture}
$

\\
\cline{2-2} \cline{3-4}
 & 
$	\begin{tikzpicture}[baseline=0]
		\node at (-0.5,0) {$\circ$};
		\draw[-] (-0.4,0) to (-0.1, 0);
		\node  at (0,0) {$\circ$};
		\node at (2,0) {$\circ$};
		\draw[-] (0.1, 0) to (0.5,0);
		\draw[dashed] (0.5,0) to (1.4,0);
		\draw[-] (1.6,0)  to (1.9,0);
		\draw (2.1, -0.1) to (2.4, -0.4);
		\node at (2.5, -0.5) {$\circ$};
		\draw (2.1, -0.9) to  (2.4, -0.6);
		\node at (-0.5,-1) {$\circ$};
		\draw[-] (-0.4, -1) to (-0.1, -1);
		\node  at (0,-1) {$\circ$};
		\node at (2,-1) {$\circ$};
		\draw[-] (0.1, -1) to (0.5,-1);
		\draw[dashed] (0.5,-1) to (1.4,-1);
		\draw[-] (1.6,-1)  to (1.9,-1);
		\draw[<->] (-0.6, -0.1) to[out=215, in=90] (-0.8, -0.5) to (-0.8, -0.5) to[out=-90,in=135] (-0.6, -0.9);
		\draw[<->] (-0.1, -0.1) to[out=215, in=90] (-0.3, -0.5) to (-0.3, -0.5) to[out=-90,in=135] (-0.1, -0.9);
		\draw[<->] (1.9, -0.1) to[out=215, in=90] (1.7, -0.5) to (1.7, -0.5) to[out=-90,in=135] (1.9, -0.9);
\end{tikzpicture}
$&
\begin{tikzpicture}[baseline=0]
\node at (0, -0.4) {EII};
\end{tikzpicture}&
$
	\begin{tikzpicture}[baseline = 12, scale =1.3]
		\node at (-1,0) {$\circ$};
		\draw (-0.9,0) to (-0.6,0);
		\node at (-0.5,0) {$\circ$};
		\draw (-0.4,0) to (-0.1,0);
		\node at (0,0) {$\circ$};
		\draw (0.1,0) to (0.4,0);
		\node at (0.5,0) {$\circ$};
		\draw (0.6,0) to (0.9,0);
		\node at (1,0) {$\circ$};
		\draw (0,-0.1) to (0,-0.4);
		\node at (0,-0.5) {$\circ$};
		\draw[<->] (-0.9, 0.1) to[out=45, in=180] (-0.5, 0.4) to (0.5, 0.4) to[out=0, in=135] (0.9, 0.1);
		\draw[<->] (-0.4, 0.1) to[out=45, in=180] (-0.15, 0.2) to (0.15, 0.2) to[out=0, in=135] (0.4, 0.1);
	\end{tikzpicture}
$
\\
\hline
\begin{tikzpicture}[baseline=0]
\node at (0, -0.3) {AIV};
\end{tikzpicture} &
\begin{tikzpicture}	[baseline=6]
		\node at (-0.5,-0.3) {$\circ$};
		\draw[-] (-0.4,-0.3) to (-0.1, -0.3);
		\node  at (0,-0.3) {$\bullet$};
		\node at (2,-0.3) {$\bullet$};
		\node at (2.5,-0.3) {$\circ$};
		\draw[-] (0.1, -0.3) to (0.5,-0.3);
		\draw[dashed] (0.5,-0.3) to (1.4,-0.3);
		\draw[-] (1.6,-0.3)  to (1.9,-0.3);
		\draw[-] (2.1,-0.3) to (2.4,-0.3);
		\draw[<->] (-0.5, -0.1) to[out=45, in=180] (0, 0.1) to (2, 0.1) to[out=0, in=135] (2.5, -0.1);
	\end{tikzpicture}
	&
\begin{tikzpicture}[baseline=0]
\node at (0, -0.3) {EIII};
\end{tikzpicture}&
\begin{tikzpicture}[baseline = 8, scale =1.3]
		\node at (-1,0) {$\circ$};
		\draw (-0.9,0) to (-0.6,0);
		\node at (-0.5,0) {$\bullet$};
		\draw (-0.4,0) to (-0.1,0);
		\node at (0,0) {$\bullet$};
		\draw (0.1,0) to (0.4,0);
		\node at (0.5,0) {$\bullet$};
		\draw (0.6,0) to (0.9,0);
		\node at (1,0) {$\circ$};
		\draw (0,-0.1) to (0,-0.4);
		\node at (0,-0.5) {$\circ$};
		\draw[<->] (-0.9, 0.1) to[out=45, in=180] (-0.5, 0.3) to (0.5, 0.3) to[out=0, in=135] (0.9, 0.1);
	\end{tikzpicture}

\\
\hline
\begin{tikzpicture}[baseline=0]
\node at (0, -0.3) {BI};
\end{tikzpicture} & 
$	\begin{tikzpicture}[baseline=8]
		\node at (-0.5,0) {$\circ$};
		\draw[-] (-0.4,0) to (-0.1, 0);
		\draw[dashed] (-0.1, 0) to (0.6,0);
		\draw[-] (0.6, 0) to (0.9, 0);
		\node at (1,0) {$\circ$};
		\node at (1.5,0) {$\bullet$};
		\draw[-] (1.1,0)  to (1.4,0);
		\draw[-] (1.4,0) to (1.9, 0);
		\draw[dashed] (1.9,0) to (2.7,0);
		\draw[-] (2.7,0) to (2.9, 0);
		\node at (3,0) {$\bullet$};
		\draw[-implies, double equal sign distance]  (3.1, 0) to (3.7, 0);
		\node at (3.8,0) {$\bullet$};
	\end{tikzpicture}	
$
&
\begin{tikzpicture}[baseline=0]
\node at (0, -0.4) {EIV};
\end{tikzpicture}&
\begin{tikzpicture}[baseline = 6, scale =1.3]
		\node at (-1,0) {$\circ$};
		\draw (-0.9,0) to (-0.6,0);
		\node at (-0.5,0) {$\bullet$};
		\draw (-0.4,0) to (-0.1,0);
		\node at (0,0) {$\bullet$};
		\draw (0.1,0) to (0.4,0);
		\node at (0.5,0) {$\bullet$};
		\draw (0.6,0) to (0.9,0);
		\node at (1,0) {$\circ$};
		\draw (0,-0.1) to (0,-0.4);
		\node at (0,-0.5) {$\bullet$};
	\end{tikzpicture}

\\
\hline
\begin{tikzpicture}[baseline=0]
\node at (0, -0.1) {BII};
\end{tikzpicture} & 
$	\begin{tikzpicture}[baseline=0, scale=1.5]
		\node at (1,0) {$\circ$};
		\node at (1.5,0) {$\bullet$};
		\draw[-] (1.1,0)  to (1.4,0);
		\draw[-] (1.4,0) to (1.9, 0);
		\draw[dashed] (1.9,0) to (2.7,0);
		\draw[-] (2.7,0) to (2.9, 0);
		\node at (3,0) {$\bullet$};
		\draw[-implies, double equal sign distance]  (3.1, 0) to (3.7, 0);
		\node at (3.8,0) {$\bullet$};
	\end{tikzpicture}	
$
&
\begin{tikzpicture}[baseline=0]
\node at (0, -0.2) {EV};
\end{tikzpicture}&
\begin{tikzpicture}[baseline = 0, scale =1]
		\node at (-1.5,0) {$\circ$};
		\draw (-1.4, 0) to (-1.1, 0);
		\node at (-1,0) {$\circ$};
		\draw (-0.9,0) to (-0.6,0);
		\node at (-0.5,0) {$\circ$};
		\draw (-0.4,0) to (-0.1,0);
		\node at (0,0) {$\circ$};
		\draw (0.1,0) to (0.4,0);
		\node at (0.5,0) {$\circ$};
		\draw (0.6,0) to (0.9,0);
		\node at (1,0) {$\circ$};
		\draw (0,-0.1) to (0,-0.4);
		\node at (0,-0.5) {$\circ$};
	\end{tikzpicture}

\\
\hline
\begin{tikzpicture}[baseline=0]
\node at (0, -0.2) {CI};
\end{tikzpicture} & 
$	\begin{tikzpicture}[baseline=6, scale=1.5]
		\node at (1,0) {$\circ$};
		\draw[-] (1.1,0)  to (1.4,0);
		\draw[dashed] (1.4,0) to (2.7,0);
		\draw[-] (2.7,0) to (2.9, 0);
		\node at (3,0) {$\circ$};
		\draw[implies-, double equal sign distance]  (3.1, 0) to (3.7, 0);
		\node at (3.8,0) {$\circ$};
	\end{tikzpicture}	
$
&
\begin{tikzpicture}[baseline=0]
\node at (0, -0.15) {EIV};
\end{tikzpicture}&
\begin{tikzpicture}[baseline = 0, scale =1]
		\node at (-1.5,0) {$\bullet$};
		\draw (-1.4, 0) to (-1.1, 0);
		\node at (-1,0) {$\circ$};
		\draw (-0.9,0) to (-0.6,0);
		\node at (-0.5,0) {$\bullet$};
		\draw (-0.4,0) to (-0.1,0);
		\node at (0,0) {$\circ$};
		\draw (0.1,0) to (0.4,0);
		\node at (0.5,0) {$\circ$};
		\draw (0.6,0) to (0.9,0);
		\node at (1,0) {$\circ$};
		\draw (0,-0.1) to (0,-0.4);
		\node at (0,-0.5) {$\bullet$};
	\end{tikzpicture}
\\
\hline
 
  \begin{tikzpicture}[baseline=0]
\node at (0, -0.5) {CII};
\end{tikzpicture}  
& 
$	\begin{tikzpicture}[baseline=6]
		\node at (-1,0) {$\bullet$};
		\draw[-] (-0.9, 0) to (-0.6,0);
		\node at (-0.5,0) {$\circ$};
		\draw[-] (-0.4,0) to (-0.1, 0);
		\draw[-] (-0.1, 0) to (0.1,0);
		\draw[dashed] (0.1, 0) to (0.8, 0);
		\draw (0.8, 0) to (0.9, 0);
		\node  at (0,0) {$\bullet$};
		\node at (1,0) {$\circ$};
		\node at (1.5,0) {$\bullet$};
		\draw[-] (1.1,0)  to (1.4,0);
		\draw[-] (1.4,0) to (1.9, 0);
		\node at (2,0) {$\bullet$};
		\draw (1.9, 0) to (2.1, 0);
		\draw[dashed] (2.1,0) to (2.7,0);
		\draw[-] (2.7,0) to (2.9, 0);
		\node at (3,0) {$\bullet$};
		\draw[implies-, double equal sign distance]  (3.1, 0) to (3.7, 0);
		\node at (3.8,0) {$\bullet$};
	\end{tikzpicture}	
$
&
\begin{tikzpicture}[baseline=0]
\node at (0, -0.15) {EVII};
\end{tikzpicture} &
\begin{tikzpicture}[baseline = 0, scale =1]
		\node at (-1.5,0) {$\circ$};
		\draw (-1.4, 0) to (-1.1, 0);
		\node at (-1,0) {$\circ$};
		\draw (-0.9,0) to (-0.6,0);
		\node at (-0.5,0) {$\bullet$};
		\draw (-0.4,0) to (-0.1,0);
		\node at (0,0) {$\bullet$};
		\draw (0.1,0) to (0.4,0);
		\node at (0.5,0) {$\bullet$};
		\draw (0.6,0) to (0.9,0);
		\node at (1,0) {$\circ$};
		\draw (0,-0.1) to (0,-0.4);
		\node at (0,-0.5) {$\bullet$};
	\end{tikzpicture}
\\
\cline{2-4}
 &
$	\begin{tikzpicture}[baseline=4]
		\node at (0,0) {$\bullet$};
		\draw (0.1, 0) to (0.4,0);
		\node  at (0.5,0) {$\circ$};
		\draw (0.6, 0) to (0.9,0);
		\node at (1,0) {$\bullet$};
		\draw[-] (1.1,0)  to (1.4,0);
		\draw[dashed] (1.4, 0) to (2.1, 0);
		\draw[-] (2.1,0) to (2.4,0);
		\node at (2.5, 0) {$\circ$};
		\draw[-] (2.6,0) to (2.9, 0);
		\node at (3,0) {$\bullet$};
		\draw[implies-, double equal sign distance]  (3.1, 0) to (3.7, 0);
		\node at (3.8,0) {$\circ$};
	\end{tikzpicture}	
$
&
\begin{tikzpicture}[baseline=0]
\node at (0, -0.15) {EVIII};
\end{tikzpicture} &
\begin{tikzpicture}[baseline = 0, scale =1]
		\node at (-2,0) {$\circ$};
		\draw (-1.9,0) to (-1.6,0);
		\node at (-1.5,0) {$\circ$};
		\draw (-1.4, 0) to (-1.1, 0);
		\node at (-1,0) {$\circ$};
		\draw (-0.9,0) to (-0.6,0);
		\node at (-0.5,0) {$\circ$};
		\draw (-0.4,0) to (-0.1,0);
		\node at (0,0) {$\circ$};
		\draw (0.1,0) to (0.4,0);
		\node at (0.5,0) {$\circ$};
		\draw (0.6,0) to (0.9,0);
		\node at (1,0) {$\circ$};
		\draw (0,-0.1) to (0,-0.4);
		\node at (0,-0.5) {$\circ$};
	\end{tikzpicture}

\\
\hline 
&
$	\begin{tikzpicture}[baseline=0]
		\node at (-0.5,0) {$\circ$};
		\draw[-] (-0.4,0) to (-0.1, 0);
		\draw[dashed] (-0.1, 0) to (0.6,0);
		\draw[-] (0.6, 0) to (0.9, 0);
		\node at (1,0) {$\circ$};
		\node at (1.5,0) {$\bullet$};
		\draw[-] (1.1,0)  to (1.4,0);
		\draw[-] (1.4,0) to (1.9, 0);
		\draw[dashed] (1.9,0) to (2.7,0);
		\draw[-] (2.7,0) to (2.9, 0);
		\node at (3,0) {$\bullet$};
		\node at (3.5, 0.4) {$\bullet$};
		\node at (3.5, -0.4) {$\bullet$};
		\draw (3.1, 0.1) to (3.4, 0.35);
		\draw (3.1, -0.1) to (3.4, -0.35);
	\end{tikzpicture}	
$
&
\begin{tikzpicture}[baseline=0]
\node at (0, 0) {EIX};
\end{tikzpicture} &
\begin{tikzpicture}[baseline = -5, scale =1]
		\node at (-2,0) {$\circ$};
		\draw (-1.9,0) to (-1.6,0);
		\node at (-1.5,0) {$\circ$};
		\draw (-1.4, 0) to (-1.1, 0);
		\node at (-1,0) {$\circ$};
		\draw (-0.9,0) to (-0.6,0);
		\node at (-0.5,0) {$\bullet$};
		\draw (-0.4,0) to (-0.1,0);
		\node at (0,0) {$\bullet$};
		\draw (0.1,0) to (0.4,0);
		\node at (0.5,0) {$\bullet$};
		\draw (0.6,0) to (0.9,0);
		\node at (1,0) {$\circ$};
		\draw (0,-0.1) to (0,-0.4);
		\node at (0,-0.5) {$\bullet$};
	\end{tikzpicture}
\\
\cline{2-4}
\begin{tikzpicture}[baseline=0]
\node at (0, 0) {DI};
\end{tikzpicture}   & 
$	\begin{tikzpicture}[baseline=0]
		\node at (1.5,0) {$\circ$};
		\draw[-] (1.6,0) to (1.9, 0);
		\draw[dashed] (1.9,0) to (2.7,0);
		\draw[-] (2.7,0) to (2.9, 0);
		\node at (3,0) {$\circ$};
		\node at (3.5, 0.4) {$\circ$};
		\node at (3.5, -0.4) {$\circ$};
		\draw (3.1, 0.1) to (3.4, 0.35);
		\draw (3.1, -0.1) to (3.4, -0.35);
		\draw[<->] (3.6, 0.35) to[out=-45, in=45] (3.6, -0.35);
	\end{tikzpicture}	
$&
\begin{tikzpicture}[baseline=0]
\node at (0, 0) {FI};
\end{tikzpicture} &
\begin{tikzpicture}[baseline=0][scale=1.5]
	\node at (0,0) {$\circ$};
	\draw (0.1, 0) to (0.4,0);
	\node at (0.5,0) {$\circ$};
	\draw[-implies, double equal sign distance]  (0.6, 0) to (1.2,0);
	\node at (1.3,0) {$\circ$};
	\draw (1.4, 0) to (1.7,0);
	\node at (1.8,0) {$\circ$};
\end{tikzpicture}
\\
\cline{2-4}
&
$	\begin{tikzpicture}[baseline=0]
		\node at (1.5,0) {$\circ$};
		\draw[-] (1.6,0) to (1.9, 0);
		\draw[dashed] (1.9,0) to (2.7,0);
		\draw[-] (2.7,0) to (2.9, 0);
		\node at (3,0) {$\circ$};
		\node at (3.5, 0.4) {$\circ$};
		\node at (3.5, -0.4) {$\circ$};
		\draw (3.1, 0.1) to (3.4, 0.35);
		\draw (3.1, -0.1) to (3.4, -0.35);
	\end{tikzpicture}	
$
&
\begin{tikzpicture}[baseline=0]
\node at (0, 0) {FII};
\end{tikzpicture} &
\begin{tikzpicture}[baseline=0][scale=1.5]
	\node at (0,0) {$\bullet$};
	\draw (0.1, 0) to (0.4,0);
	\node at (0.5,0) {$\bullet$};
	\draw[-implies, double equal sign distance]  (0.6, 0) to (1.2,0);
	\node at (1.3,0) {$\bullet$};
	\draw (1.4, 0) to (1.7,0);
	\node at (1.8,0) {$\circ$};
\end{tikzpicture}
\\
\hline
\begin{tikzpicture}[baseline=0]
\node at (0, 0) {DII};
\end{tikzpicture} 
&
$	\begin{tikzpicture}[baseline=0]
		\node at (1,0) {$\circ$};
		\node at (1.5,0) {$\bullet$};
		\draw[-] (1.1,0)  to (1.4,0);
		\draw[-] (1.4,0) to (1.9, 0);
		\draw[dashed] (1.9,0) to (2.7,0);
		\draw[-] (2.7,0) to (2.9, 0);
		\node at (3,0) {$\bullet$};
		\node at (3.5, 0.4) {$\bullet$};
		\node at (3.5, -0.4) {$\bullet$};
		\draw (3.1, 0.1) to (3.4, 0.35);
		\draw (3.1, -0.1) to (3.4, -0.35);
	\end{tikzpicture}	
$
&
\begin{tikzpicture}[baseline=0]
\node at (0, 0) {G};
\end{tikzpicture} 
&
	\begin{tikzpicture}[baseline=0]
		\node at (0.1,0) {$\circ$};
		\node at (0.9,0) {$\circ$};
		\node at (0.5,0) {$\Rrightarrow$};
	\end{tikzpicture}
\\
\hline
\end{tabular}
\end{table}


\end{document}